\documentclass[11pt, a4paper,leqno]{amsart}
\usepackage{amsmath,amsthm,amscd,amssymb,amsfonts, amsbsy}
\usepackage{latexsym}
\usepackage{txfonts}
\usepackage{exscale}
\usepackage{mathrsfs}

\usepackage[T1]{fontenc}

\usepackage[colorlinks,citecolor=blue,pagebackref,hypertexnames=false, breaklinks]{hyperref}
\usepackage[usenames,dvipsnames]{color}

\newcommand{\mm}{\mathfrak{m}}

\calclayout
\allowdisplaybreaks

    \def\Xint#1{\mathchoice
    {\XXint\displaystyle\textstyle{#1}}%
    {\XXint\textstyle\scriptstyle{#1}}%
    {\XXint\scriptstyle\scriptscriptstyle{#1}}%
    {\XXint\scriptscriptstyle\scriptscriptstyle{#1}}%
    \!\int}
    \def\XXint#1#2#3{{\setbox0=\hbox{$#1{#2#3}{\int}$}
    \vcenter{\hbox{$#2#3$}}\kern-.5\wd0}}
    \def\dashint{\Xint-}

\def\Xint#1{\mathchoice
    {\XXint\displaystyle\textstyle{#1}}%
    {\XXint\textstyle\scriptstyle{#1}}%
    {\XXint\scriptstyle\scriptscriptstyle{#1}}%
    {\XXint\scriptscriptstyle\scriptscriptstyle{#1}}%
    \!\int}
    \def\XXint#1#2#3{{\setbox0=\hbox{$#1{#2#3}{\int}$}
    \vcenter{\hbox{$#2#3$}}\kern-.5\wd0}}

\renewcommand{\chi}{{\bf 1}}

\theoremstyle{plain}
\newtheorem{theorem}[equation]{Theorem}
\newtheorem{lemma}[equation]{Lemma}

\newtheorem{proposition}[equation]{Proposition}

\theoremstyle{definition}
\newtheorem{definition}[equation]{Definition}

\theoremstyle{remark}
\newtheorem{remark}[equation]{Remark}

\newtheorem{notation}[equation]{Notation}

\numberwithin{equation}{section}

\def \R{ \mathbb{R} }
\def \N{ \mathbb{N} }
\def \Z{ \mathbb{Z} }

\def \Scal{ \mathcal{S} }
\def \hh{ \mathrm{H} }
\def \pp{ \mathrm{P} }
\def \Gcal { \mathcal{G} }
\def \Grm{ \mathrm{G} }
\def \Ncal { \mathcal{N} }
\def \mol{ (w,p,\varepsilon,M)-\textrm{molecule}}
\def \p{ (w,p,\varepsilon,M)-\textrm{representation}}
\def \iint{\int\!\!\!\int}
\def\div{\mathop{\rm div}}
\renewcommand{\Re}{{\rm Re}\,}

\DeclareMathOperator{\supp}{supp}

\begin{document}
\allowdisplaybreaks
\author{Jos\'e Mar{\'\i}a Martell}
\address{Jos\'e Mar{\'\i}a Martell
\\
Instituto de Ciencias Matem\'aticas CSIC-UAM-UC3M-UCM
\\
Consejo Superior de Investigaciones Cient{\'\i}ficas
\\
C/ Nicol\'as Cabrera, 13-15
\\
E-28049 Madrid, Spain} \email{chema.martell@icmat.es}

\title[Weighted Hardy spaces associated with elliptic operators]{Weighted Hardy spaces associated with elliptic operators.
\\[0.3cm]
{\small Part II: Characterizations of $H^1_L(\lowercase{w})$}}

\author{Cruz Prisuelos-Arribas}

\address{Cruz Prisuelos-Arribas
Instituto de Ciencias Matem\'aticas, CSIC-UAM-UC3M-UCM
\\
Consejo Superior de Investigaciones Cient{\'\i}ficas
\\
C/ Nicol\'as Cabrera, 13-15
\\
E-28049 Madrid, Spain} \email{cruz.prisuelos@icmat.es}

\thanks{The research leading to these results has received funding from the European Research
Council under the European Union's Seventh Framework Programme (FP7/2007-2013)/ ERC
agreement no. 615112 HAPDEGMT.
Both authors  acknowledge financial support from the Spanish Ministry of Economy and Competitiveness, through the ``Severo Ochoa Programme for Centres of Excellence in R\&D'' (SEV-2015-0554). Both authors would like to thank P. Auscher for his useful comments and suggestions. 
}

\date{ January 10, 2017. \textit{Revised}: \today}


\begin{abstract}
Given a Muckenhoupt weight $w$ and a second order divergence form elliptic operator $L$, we consider different versions of the weighted Hardy space $H^1_L(w)$ defined by 
conical square functions and non-tangential maximal functions associated with the heat and Poisson semigroups generated by $L$. We show that all of them are isomorphic and also that $H^1_L(w)$ admits a molecular characterization.  One of the advantages of our methods is that our assumptions extend naturally the unweighted theory developed by 
S. Hofmann and S. Mayboroda in \cite{HofmannMayboroda} and we can immediately recover the unweighted case.  Some of our tools consist in establishing weighted norm inequalities for the non-tangential maximal functions, as well as comparing them with some conical square functions in weighted Lebesgue spaces. 
\end{abstract}

\maketitle

\tableofcontents

\bigskip
\section{Introduction}
This is the second of a series of three papers whose aim is to study and develop a theory for weighted Hardy spaces associated with different operators. 
The study of Hardy spaces began in the early 1900s in the context of Fourier series and complex analysis in one variable. It was not until 1960 when the theory in $\R^n$ started developing by
E.M. Stein and G. Weiss (\cite{SteinWeiss}). A few years later R.R. Coifman in \cite{Coifman} and R.H. Latter in \cite{Latter} gave an atomic decomposition of the Hardy spaces $H^p$, $0<p\leq 1$. This atomic decomposition turns out to be a very important tool when studying the boundedness of some singular integral operators, since
in most cases checking the action of the operator in question on these simpler elements (atoms) suffices to conclude its boundedness in the corresponding Hardy space.

Another stage in the progress of the theory of Hardy spaces was done by J. Garc\'ia-Cuerva in \cite{GarciaCuerva} (see also \cite{StrombergTorchinsky}) when he considered $\R^n$ with the measure given by a Muckenhoupt weight. These spaces were called weighted Hardy spaces, and among other contributions, he also characterized them using an atomic decomposition.

In general, the development of the theory of Hardy spaces has contributed to give us a better understanding of some other topics as in the theory of singular integrals operators, maximal functions, multiplier operators, etc. However, there are some operators that escape from the theory of these classical Hardy spaces. These are, for example, the operators associated with a second order divergence form elliptic operator $L$, such as 
the conical square functions and non-tangential maximal functions defined by the heat and Poisson  semigroups generated by the operator $L$, (see \eqref{square-H-1}-\eqref{square-P-3} and \eqref{nontangential1}--\eqref{nontangential2} for the precise definitions of these operators).

The theory of Hardy spaces associated with elliptic operators $L$ was initiated in an unpublished work by P. Auscher, X.T. Duong and A. McIntosh \cite{Auscher-Duong-McIntosh}. P. Auscher and E. Russ in \cite{Auscher-Russ} considered the case on which the heat kernel associated with $L$ is smooth and satisfies pointwise Gaussian bounds, this occurs for instance for real symmetric  operators.  There, among other things, it was shown that the corresponding  Hardy space associated with $L$ agrees with the classical Hardy space.
In the setting of Riemannian manifolds satisfying the doubling volume property,
Hardy spaces associated with the Laplace-Beltrami operator are introduced in
\cite{Auscher-McIntosh-Russ} by P. Auscher, A. McIntosh and E. Russ and it is
shown that they admit several characterizations.  Simultaneously, in the
Euclidean setting, the study of  Hardy spaces related to the conical square
functions and non-tangential maximal functions associated with the heat and
Poisson  semigroups generated by divergence form elliptic operators $L$ was
taken by S. Hofmann and S. Mayboroda in \cite{HofmannMayboroda}, for $p=1$. The
new point was that only a form of decay weaker than pointwise bounds
and satisfied in many occurrences was enough to develop a theory. This was
followed later on by a second article of S. Hofmann, S. Mayboroda,
and A. McIntosh  \cite{HofmannMayborodaMcIntosh}, for a general $p$ and
simultaneously by an article of R. Jiang and D. Yang \cite{JiangYang}.
 A natural line of study in the context of these new Hardy spaces is the development of a weighted theory for them, as J. Garc\'ia-Cuerva did in the classical setting. Some interesting progress has been done in this regard by T.A. Bui, J. Cao, L.D. Ky, D. Yang, and S. Yang in \cite{BuiCaoKyYangYang,  BuiCaoKyYangYang:II}. The results obtained in \cite{BuiCaoKyYangYang:II} 
in the particular case $\varphi(x,t):=tw(x)$,  where $w$ is a Muckenhoupt weight, give characterizations of the weighted Hardy spaces that, however, only recover part of the results obtained in the unweighted case by simply taking $w=1$.  

In this paper we take a further step, and present a different approach to the theory of weighted Hardy spaces $H^1_L(w)$ (the general case $H^p_L(w)$ will be treated in the forthcoming paper \cite{Prisuelos}) associated with a second order divergence form elliptic operator, which naturally generalizes the unweighted setting developed in \cite{HofmannMayboroda}. We define weighted Hardy spaces associated with the conical square functions considered in \eqref{square-H-1}--\eqref{square-P-3} which are written in terms of the heat and Poisson semigroups generated by the elliptic operator. Also, we use  non-tangential maximal functions as defined in \eqref{nontangential1}--\eqref{nontangential2}. We show that the corresponding spaces are all isomorphic and admit a molecular characterization. This is particularly useful to prove different properties of these spaces as happens in the classical setting and in the context of second order divergence form elliptic operators considered in \cite{HofmannMayboroda}.

Some of the ingredients that are crucial in the present work are taken from the first part of this series of papers \cite{MartellPrisuelos}, where we already obtained optimal ranges for the weighted norm inequalities satisfied by the heat and Poisson conical square functions associated with the elliptic operator. Here, we need to obtain analogous results for the non-tangential maximal functions  associated with the heat and Poisson semigroups (see Section \ref{section:NT}). All these weighted norm inequalities for the conical square  functions and the non-tangential maximal functions, along with the important fact that our molecules belong naturally to weighted Lebesgue spaces, allow us to impose natural conditions that in particular lead to fully recover the results obtained in \cite{HofmannMayboroda} by simply taking the weight identically one. It is relevant to note that in \cite{BuiCaoKyYangYang, BuiCaoKyYangYang:II} their molecules belong to unweighted Lebesgue spaces and also their ranges of boundedness of the conical square functions are smaller. This makes their hypothesis somehow stronger (although sometimes they cannot be compared with ours) and, despite making a very big effort to present a very general theory, the unweighted case does not follow immediately from their work.

The plan of this paper is as follows. In the next section we present some preliminaries concerning Muckenhoupt weights, elliptic operators and introduce the conical square functions and non-tangential maximal functions. In Section \ref{section-main} we define the different versions of the weighted Hardy spaces and state our main results. Section \ref{section:aux} contains some auxiliary results. Sections \ref{SH} and \ref{SKPG} deal with the characterization of the weighted Hardy spaces defined in terms of square functions associated with the heat and Poisson semigroups, respectively. Finally, in Section \ref{section:NT} we study the non-tangential maximal functions and the weighted Hardy spaces associated with them. 

\section{Preliminaries}
\subsection{Muckenhoupt weights}

We will work with Muckenhoupt weights $w$, which are locally integrable positive functions. We say that $w\in A_1$ if, for every ball $B\subset \mathbb{R}^n$, there holds
$$
\dashint_Bw(x) \, dx\leq C w(y), \quad \textrm{for a.e. } y\in B,
$$
or, equivalently, $\mathcal{M}_u w\le C\,w$  a.e. where $\mathcal{M}_u$ denotes the uncentered Hardy-Littlewood
 maximal operator over balls in $\R^n$.
For each $1<p<\infty$, we say that $w\in A_p$ if it satisfies
$$\left(\dashint_Bw(x) \, dx\right)\left(\dashint_Bw(x)^{1-p'} \, dx\right)^{p-1}
\leq C, \quad \forall B\subset \mathbb{R}^n.$$
The reverse H\"older classes are defined as follows: for each $1<s<\infty$,
$w\in RH_s$ if, for every ball $B\subset \mathbb{R}^n$, we have
$$\left(\dashint_Bw(x)^s \, dx\right)^{\frac{1}{s}}\leq C\dashint_Bw(x) \, dx.$$
For $s=\infty$, $w\in RH_{\infty}$ provided that there exists a constant $C$ such that for every ball $B\subset \mathbb{R}^n$
$$
w(y)\leq C\dashint_Bw(x) \, dx, \quad \textrm{for a.e. }y\in B.
$$
Notice that we have excluded the case $q = 1$ since the class $RH_1$ consists of all the
weights, and that is the way $RH_1$ is understood in what follows.

We sum up some of the properties of these classes in the following result, see for instance \cite{GCRF85},
\cite{Duo}, or \cite{Grafakos}. %For $(vii)$ see \cite{JN}.

\begin{proposition}\label{prop:weights}\
\begin{enumerate}
\renewcommand{\theenumi}{\roman{enumi}}
\renewcommand{\labelenumi}{$(\theenumi)$}
\addtolength{\itemsep}{0.2cm}

\item $A_1\subset A_p\subset A_q$ for $1\le p\le q<\infty$.

\item $RH_{\infty}\subset RH_q\subset RH_p$ for $1<p\le q\le \infty$.

\item If $w\in A_p$, $1<p<\infty$, then there exists $1<q<p$ such
that $w\in A_q$.

\item If $w\in RH_s$, $1<s<\infty$, then there exists $s<r<\infty$ such
that $w\in RH_r$.

\item $\displaystyle A_\infty=\bigcup_{1\le p<\infty} A_p=\bigcup_{1<s\le
\infty} RH_s$.

\item If $1<p<\infty$, $w\in A_p$ if and only if $w^{1-p'}\in
A_{p'}$.

\item For every $1<p<\infty$, $w\in A_p$ if and only if $\mathcal{M}$ is bounded
on $L^p(w)$. Also, $w\in A_1$ if and only if $\mathcal{M}$ is bounded from $L^1(w)$ into $L^{1,\infty}(w)$, where $\mathcal{M}$ denotes the centered Hardy-Littlewood maximal operator. 

%\item If $1\le q\le \infty$ and $1\le s<\infty$, then $\displaystyle w\in A_q \cap RH_s$ if and only if $ w^{s}\in A_{s\,(q-1)+1}$.

\end{enumerate}
\end{proposition}

\medskip

For a weight $w\in A_{\infty}$, define
\begin{align}\label{rw}
r_w:=\inf\{1\leq r<\infty : w\in A_{r}\},
\qquad
s_w:=\inf\{1\leq s<\infty : w\in RH_{s'}\}.
\end{align}
Notice that according to our definition $s_w$ is the conjugated exponent of the one defined in \cite[Lemma 4.1]{AuscherMartell:I}.
Given $0\le p_0<q_0\le \infty$, $w\in A_{\infty}$, and according to \cite[Lemma 4.1]{AuscherMartell:I} we have
\begin{align}\label{intervalrs}
\mathcal{W}_w(p_0,q_0):=\left\{p : p_0<p<q_0, w\in A_{\frac{p}{p_0}}\cap RH_{\left(\frac{q_0}{p}\right)'}\right\}
=
\left(p_0r_w,\frac{q_0}{s_w}\right).
\end{align}
If $p_0=0$ and $q_0<\infty$ it is understood that the only condition that stays is $w\in RH_{\left(\frac{q_0}{p}\right)'}$. Analogously, if $0<p_0$ and $q_0=\infty$ the only assumption is $w\in A_{\frac{p}{p_0}}$. Finally $\mathcal{W}_w(0,\infty)=(0,\infty)$.

We recall some properties of Muckenhoupt weights. Let $w$ be a  weight in $A_{\infty}$,
if $w\in A_{r}$, $1\leq r<\infty$,
for every ball $B$ and every measurable set $E\subset B$,
\begin{align}\label{pesosineq:Ap}
\frac{w(E)}{w(B)}\ge [w]_{A_{r}}^{-1}\left(\frac{|E|}{|B|}\right)^{r}.
\end{align}
This implies in particular that $w$ is a doubling measure:
\begin{align}\label{doublingcondition}
w(\lambda B)
\le
[w]_{A_r}\,\lambda^{n\,r}w(B),
\qquad \forall\,B,\ \forall\,\lambda>1.
\end{align}
Besides, if $w\in RH_{s'}$, $1\leq s<\infty$,
\begin{align}\label{pesosineq:RHq}
\frac{w(E)}{w(B)}
\leq
[w]_{RH_{s'}}\left(\frac{|E|}{|B|}\right)^{\frac{1}{s}}.
\end{align}
\subsection{Elliptic operators}
Let $A$ be an $n\times n$ matrix of complex and
$L^\infty$-valued coefficients defined on $\R^n$. We assume that
this matrix satisfies the following ellipticity (or \lq\lq
accretivity\rq\rq) condition: there exist
$0<\lambda\le\Lambda<\infty$ such that
\begin{align}\label{matrix:Aproperties}
\lambda\,|\xi|^2
\le
\Re A(x)\,\xi\cdot\bar{\xi},
\quad\qquad\mbox{and}\qquad\quad
|A(x)\,\xi\cdot \bar{\zeta}|
\le
\Lambda\,|\xi|\,|\zeta|,
\end{align}
for all $\xi,\zeta\in\mathbb{C}^n$ and almost every $x\in \R^n$. We have used the notation
$\xi\cdot\zeta=\xi_1\,\zeta_1+\cdots+\xi_n\,\zeta_n$ and therefore
$\xi\cdot\bar{\zeta}$ is the usual inner product in $\mathbb{C}^n$. 
Associated with this matrix we define the second order divergence
form elliptic operator
\begin{align}\label{matrix:A}
L f
=
-\div(A\,\nabla f),
\end{align}
which is understood in the standard weak sense as a maximal-accretive operator on $L^2(\R^n)$ with domain $\mathcal{D}(L)$ by means of a
sesquilinear form.

As in \cite{Auscher} and \cite{AuscherMartell:II}, we denote respectively by $(p_-(L),p_+(L))$ and $(q_-(L),q_+(L))$ the maximal open intervals on which the heat semigroup $\{e^{-tL}\}_{t>0}$ and its 
gradient $\{\sqrt{t}\nabla_y e^{-tL}\}_{t>0}$
are uniformly bounded on $L^p(\mathbb{R}^n)$:
\begin{align}\label{p-}
p_-(L) &:= \inf\left\{p\in(1,\infty): \sup_{t>0} \|e^{-t^2L}\|_{L^p(\mathbb{R}^n)\rightarrow L^p(\mathbb{R}^n)}< \infty\right\},
\\[4pt]
p_+(L)& := \sup\left\{p\in(1,\infty) : \sup_{t>0} \|e^{-t^2L}\|_{L^p(\mathbb{R}^n)\rightarrow L^p(\mathbb{R}^n)}< \infty\right\},
\label{p+}
\end{align}
\begin{align}\label{q-}
q_-(L) &:= \inf\left\{p\in(1,\infty): \sup_{t>0} \|t\nabla_y e^{-t^2L} \|_{L^p(\mathbb{R}^n)\rightarrow L^p(\mathbb{R}^n)}< \infty\right\},
\\[4pt]
q_+(L)& := \sup\left\{p\in(1,\infty) : \sup_{t>0} \|t\nabla_y e^{-t^2L} \|_{L^p(\mathbb{R}^n)\rightarrow L^p(\mathbb{R}^n)}< \infty\right\}.\label{q+}
\end{align}
From \cite{Auscher} (see also \cite{AuscherMartell:II}) we know that $p_-(L)=1$ and  $p_+(L)=\infty$ if $n=1,2$; and if $n\ge 3$ then $p_-(L)<\frac{2\,n}{n+2}$ and $p_+(L)>\frac{2\,n}{n-2}$. Moreover, $q_-(L)=p_-(L)$, $ q_+(L)^*\le p_+(L)$ (where $q_+(L)^*$ is the Sobolev exponent of $q_+(L)$ as defined below), and we always have $q_+(L)>2$, with $q_+(L)=\infty$ if $n=1$.

Note that in place of the semigroup $\{e^{-t L}\}_{t>0}$ we are using its rescaling $\{e^{-t^2 L}\}_{t>0}$. We do so since all the ``heat'' square functions are written using the latter and also because in the context of the off-diagonal estimates discussed below it will simplify some computations.

Besides, for every $K\in\mathbb{N}_0$ and $0<q<\infty$ let us set
$$
q^{K,*}:=
\left\{
\begin{array}{ll}
\dfrac{q\,n}{n-(2K+1)\,q}, &\quad\mbox{ if}\quad(2K+1)\,q<n,
\\[10pt]
\infty, &\quad\mbox{ if}\quad(2K+1)\,q\ge n.
\end{array}
\right.
$$
Corresponding to the case $K=0$, we write $q^{*}:=q^{0,*}$.
%%%%%%%%%%%%%%%%%%%%%%%%%%%%%%%%%%%
%%%%%%%%%%%%%%%%%%%%%%%%%%%%%%%%%%%
\subsection{Off-diagonal estimates}\label{section:OD}
%%%%%%%%%%%%%%%%%%%%%%%%%%%%%%%%%%%

We briefly recall the notion of off-diagonal estimates. Let $\{T_t\}_{t>0}$ be a family of linear operators
and let $1\le p\leq q\le \infty$. We say that $\{T_t\}_{t>0}$ satisfies $L^p(\R^n)-L^q(\R^n)$ off-diagonal estimates of exponential type, denoted by $\{T_t\}_{t>0}\in \mathcal{F}_\infty(L^p\rightarrow L^q)$, if for all closed sets $E$, $F$, all $f$, and all $t>0$ we have
$$
\|T_{t}(f\,\chi_E)\,\chi_F\|_{L^q(\mathbb{R}^n)}
\leq
Ct^{-n\left(\frac{1}{p}-\frac{1}{q}\right)}e^{-c\frac{d(E,F)^2}{t^2}}\|f\,\chi_E\|_{L^p(\mathbb{R}^n)}.
$$
Analogously, given $\beta>0$, we say that $\{T_t\}_{t>0}$ satisfies $L^p-L^q$ off-diagonal estimates of polynomial type with order $\beta>0$, denoted by $\{T_t\}_{t>0}\in \mathcal{F}_{\beta}(L^p\rightarrow L^q)$ if for all closed sets $E$, $F$, all $f$, and all $t>0$ we have
$$
\|T_{t}(f\,\chi_E)\,\chi_F\|_{L^q(\mathbb{R}^n)}
\leq
Ct^{-n\left(\frac{1}{p}-\frac{1}{q}\right)}\left(1+\frac{d(E,F)^2}{t^2}
    \right)^{-\left(\beta+\frac{n}{2}\left(\frac{1}{p}-\frac{1}{q}\right)\right)}
\|f\,\chi_E\|_{L^p(\mathbb{R}^n)}.
$$

\medskip

The heat and Poisson semigroups satisfy respectively off-diagonal estimates of exponential and polynomial type. Before making this precise, let us recall the definition of $p_-(L)$, $p_+(L)$, $q_-(L)$, and $q_+(L)$ in \eqref{p-}--\eqref{p+} and in \eqref{q-}--\eqref{q+}.
The importance of these parameters stems from the fact that, besides giving the maximal intervals on which either the heat semigroup or its gradient are uniformly bounded, they characterize the maximal open intervals on which off-diagonal estimates of exponential type hold (see \cite{Auscher} and \cite{AuscherMartell:II}). More precisely, for every $m\in \N_0$, there hold
$$
\{(t^2L)^me^{-t^2L}\}_{t>0}\in \mathcal{F}_\infty(L^p-L^q) \quad \textrm{ for all} \quad p_-(L)<p\leq q<p_+(L)$$
and
$$
\{t\nabla_ye^{-t^2L}\}_{t>0}\in \mathcal{F}_\infty(L^p-L^q) \quad \textrm{ for all} \quad q_-(L)<p\leq q<q_+(L).
$$
From these off-diagonal estimates we have, for every $m\in \N_0$,
\begin{align*}
\{(t\sqrt{L}\,)^{2m}e^{-t\sqrt{L}}\}_{t>0}, \
\in \mathcal{F}_{m+\frac{1}{2}}(L^p\rightarrow L^q),
\end{align*}
for all $p_-(L)<p\leq q< p_+(L)$, and
\begin{align*}
 &\{t\nabla_{y}(t^2L)^me^{-t^2L}\}_{t>0}, \ \{t\nabla_{y,t}(t^2L)^me^{-t^2L}\}_{t>0}\in \mathcal{F}_\infty(L^p\rightarrow L^q),
\\[4pt]
&  \{t\nabla_{y}(t\sqrt{L}\,)^{2m}e^{-t\sqrt{L}}\}_{t>0}\in \mathcal{F}_{m+1}(L^p\rightarrow L^q), \,
\{t\nabla_{y,t}(t\sqrt{L}\,)^{2m}e^{-t\sqrt{L}}\}_{t>0}\in \mathcal{F}_{m+\frac{1}{2}}(L^p\rightarrow L^q),
\end{align*}
for all $q_-(L)<p\leq q< q_+(L)$, (see \cite[Section 2]{MartellPrisuelos}).

%%%%%%%%%%%%%%%%%%%%%%%%%%%%%%%%%%%

\subsection{Conical square functions and non-tangential maximal functions}

The operator $-L$ generates a $C^0$-semigroup $\{e^{-t L}\}_{t>0}$ of contractions on $L^2(\mathbb{R}^n)$ which is called the heat semigroup. Using this semigroup and the corresponding Poisson semigroup $\{e^{-t\,\sqrt{L}}\}_{t>0}$, one can define different conical square functions which all have an expression of the form
\begin{align}\label{squarealpha}
Q^{\alpha} f(x)
=\left(\iint_{\Gamma^{\alpha}(x)}|T_t f(y)|^2 \frac{dy \, dt}{t^{n+1}}\right)^{\frac{1}{2}},
\qquad
x\in\R^n,
\end{align}
where $\alpha>0$ and $\Gamma^{\alpha}(x): =\{(y,t)\in \R^{n+1}_+: |x-y|<\alpha t\}$ denotes the cone (of aperture $\alpha$) with vertex at $x\in\R^n$ (see \eqref{cone}).  When $\alpha=1$ we just write $Qf(x)$ and $\Gamma(x)$.
More precisely, we introduce the following conical square functions written in terms of the heat semigroup $\{e^{-t L}\}_{t>0}$ (hence the subscript $\hh$): for every $m\in \mathbb{N}$,
\begin{align} \label{square-H-1}
\mathcal{S}_{m,\hh}f(x) & = \left(\iint_{\Gamma(x)}|(t^2L)^{m} e^{-t^2L}f(y)|^2 \frac{dy \, dt}{t^{n+1}}\right)^{\frac{1}{2}},
\end{align}
and, for every $m\in \mathbb{N}_0:=\mathbb{N}\cup\{0\}$,
\begin{align}
\mathrm{G}_{m,\hh}f(x)& =\left(\iint_{\Gamma(x)}|t\nabla_y(t^2L)^m e^{-t^2L}f(y)|^2 \frac{dy \, dt}{t^{n+1}}\right)^{\frac{1}{2}},
\label{square-H-2}\\[4pt]
\mathcal{G}_{m,\hh}f(x)
&
=
\left(\iint_{\Gamma(x)}|t\nabla_{y,t}(t^2L)^m e^{-t^2L}f(y)|^2 \frac{dy \, dt}{t^{n+1}}\right)^{\frac{1}{2}}.
\label{square-H-3}
\end{align}

In the same manner, let us consider conical square functions associated with the Poisson semigroup $\{e^{-t \sqrt{L}}\}_{t>0}$ (hence the subscript $\pp$):  given $K\in \mathbb{N}$,
\begin{align}
\mathcal{S}_{K,\pp}f(x)
&=
\left(\iint_{\Gamma(x)}|(t\sqrt{L}\,)^{2K} e^{-t\sqrt{L}}f(y)|^2 \frac{dy \, dt}{t^{n+1}}\right)^{\frac{1}{2}},
\label{square-P-1}
\end{align}
and for every $K\in \mathbb{N}_0$,
\begin{align}
\mathrm{G}_{K,\pp}f(x)
&=\left(\iint_{\Gamma(x)}|t\nabla_y (t\sqrt{L}\,)^{2K} e^{-t\sqrt{L}}f(y)|^2 \frac{dy \, dt}{t^{n+1}}\right)^{\frac{1}{2}},
\label{square-P-2}
\\[4pt]
\mathcal{G}_{K,\pp}f(x)
&=
\left(\iint_{\Gamma(x)}|t\nabla_{y,t}(t\sqrt{L}\,)^{2K} e^{-t\sqrt{L}}f(y)|^2 \frac{dy \, dt}{t^{n+1}}\right)^{\frac{1}{2}}.
\label{square-P-3}
\end{align}
Corresponding to the cases $m=0$ or $K=0$ we simply write $\mathrm{G}_{\hh}f:=\mathrm{G}_{0,\hh}f$,
$\mathcal{G}_{\hh}f:=\mathcal{G}_{0,\hh}f$,  $\mathrm{G}_{\pp}f:=\mathrm{G}_{0,\pp}f$, and
$\mathcal{G}_{\pp}f:=\mathcal{G}_{0,\pp}f$. Besides, we set $\Scal_{\hh}f:=\Scal_{1,\hh}f$, $\Scal_{\pp}f:=\Scal_{1,\pp}f$.

We also introduce the non-tangential maximal functions $\Ncal_{\hh}$ and $\Ncal_{\pp}$ associated respectively with the heat and Poisson semigroups:
\begin{equation}\label{nontangential1}
\mathcal{N}_{\hh}f(x)=\sup_{(y,t)\in \Gamma(x)}\left(\int_{B(y,t)}
|e^{-t^2L}f(z)|^2 \frac{dz}{t^n}\right)^{\frac{1}{2}}
\end{equation}
and
\begin{equation}\label{nontangential2}
 \mathcal{N}_{\pp}f(x)=\sup_{(y,t)\in \Gamma(x)}\left(\int_{B(y,t)}
|e^{-t\sqrt{L}}f(z)|^2 \frac{dz}{t^n}\right)^{\frac{1}{2}}.
\end{equation}

\section{Definitions and main results}\label{section-main}

%%%%%%%%%%%%%%%%%%%%%%%%%%%%%%%%%%%%%%%%%%%%%%%%%%%%%%%%%%%%%%%%%%%%%%%%%

%%%%%%%%%%%%%%%%%%%%%%%%%%%%%%%%%%%%%%%%%%%%%%%%%%%%%%%%%%%%%%%%%%%%%%%%%
%%%%%%%%%%%%%%%%%%%%%%%%%%%%%%%%%%%%%%%%%%%%%%%%%%%%%%%%%%%%%%%%%%%%%%%%%

As in the classical setting our weighted Hardy spaces will admit several characterizations using molecules, conical square functions, or non-tangential maximal functions. They will extend the definitions and results obtained in the unweighted case  in \cite{HofmannMayboroda}, to weights $w\in A_{\infty}$ such that $\mathcal{W}_w(p_-(L),p_+(L))\neq \emptyset$.

\subsection{Molecular weighted Hardy spaces}

To set the stage, we take a molecular version of the weighted Hardy space as the original definition, and we shall show that all the other definitions are isomorphic to that one and one another. In order to formalize the notion of molecules and molecular decomposition we introduce the following notation: given a cube $Q\subset \R^n$ we set
\begin{align}\label{100}
C_1(Q):=4Q, \quad  C_i(Q):=2^{i+1}Q\backslash 2^{i}Q,\quad\textrm{for}\quad i\geq 2, \quad
 \textrm{and}\quad \quad Q_i=2^{i+1}Q, \quad\textrm{for}\quad i\geq 1. 
\end{align}
\begin{definition}[Molecules and molecular representation]\label{moleculas}
Let $w\in A_{\infty}$, $p\in \mathcal{W}_w(p_-(L),p_+(L))$,
$\varepsilon >0$, and $ M\in \N$ such that $M>\frac{n}{2}\left(r_w-\frac{1}{p_-(L)}\right)$.
\begin{list}{$(\theenumi)$}{\usecounter{enumi}\leftmargin=1cm \labelwidth=1cm\itemsep=0.2cm\topsep=.0cm \renewcommand{\theenumi}{\alph{enumi}}}

\item \textbf{Molecules:} We say that a function $\mm\in L^p(w)$, (belonging to the range of $L^k$ in $L^p(w)$), is a $\mol$ if, for some cube $Q\subset \mathbb{R}^n$, $\mm$ satisfies
\begin{align*}
\|\mm\|_{mol,w}:=\sum_{i\geq 1}2^{i\varepsilon}w(2^{i+1}Q)^{1-\frac{1}{p}}\sum_{k=0}^{M}\left\|\left((\ell(Q)^2L)^{-k}\mm\right)\chi_{C_i(Q)}\right\|_{L^p(w)}<1.
\end{align*}
Henceforth, we refer to the previous expression as the molecular $w$-norm of $\mm$. Additionally, any cube $Q$ satisfying that expression, is called a cube associated with $\mm$. Besides, note that
 if $\mm$ is a $\mol$, in particular we have
\begin{align}\label{astast}
\left\|\left((\ell(Q)^2L)^{-k}\mm\right)
    \chi_{C_i(Q)}\right\|_{L^p(w)}
    \leq 2^{-i\varepsilon}w(2^{i+1}Q)^{\frac{1}{p}-1}, \quad i=1,2,\dots; \, k=0, 1,\ldots, M.
\end{align}

\item  \textbf{Molecular representation:} For any function $f\in L^p(w)$, we say that the sum $\sum_{i\in \N}\lambda_i\mm_i $ is a $(w,p,\varepsilon,M)-$ representation of $f$, if the following conditions are satisfied:
\begin{list}{$(\theenumi)$}{\usecounter{enumi}\leftmargin=1cm \labelwidth=1cm\itemsep=0.2cm\topsep=.2cm \renewcommand{\theenumi}{\roman{enumi}}}

\item $\displaystyle\{\lambda_i\}_{i\in \N}\in \ell^1$.
\item For every $i\in\N$, $\mm_i$ is a $\mol$.
\item  $f=\sum_{i\in \N}\lambda_i\mm_i$ in $L^p(w)$.
\end{list}

\end{list}

\end{definition}
These objects are a weighted version of the ones defined in \cite{HofmannMayboroda} in the unweighted case.

We finally define the molecular weighted Hardy spaces.
\begin{definition}[Molecular weighted Hardy spaces]
For $w\in A_{\infty}$, $p\in \mathcal{W}_w(p_-(L),p_+(L))$, $\varepsilon>0$, and $M\in \N$ such that $M>\frac{n}{2}\left(r_w-\frac{1}{p_-(L)}\right)$, we
define the molecular weighted Hardy space $H^1_{L,p,\varepsilon,M}(w)$  as the completion of the set
$$
\mathbb{H}_{L,p,\varepsilon,M}^1(w):=\left\{f=\sum_{i=1}^{\infty}\lambda_i\mm_i
: \sum_{i=1}^{\infty}\lambda_i\mm_i \ \textrm{is a} \ \p\textrm{ of }f\right\},
$$
with respect to the norm,
$$
\|f\|_{\mathbb{H}_{L,p,\varepsilon,M}^1(w)}:=\inf\left\{\sum_{i=1}^{\infty}|\lambda_i|:
\sum_{i=1}^{\infty}\lambda_i\mm_i \ \textrm{is a} \ \p\textrm{ of }f\right\}.
$$
\end{definition}

We shall show below that the Hardy spaces $H^1_{L,p,\varepsilon,M}(w)$ do not depend on the choice of the allowable parameters $p$, $\varepsilon$, and $M$. Hence, at this point, it is convenient for us to make a choice of these parameters and define the weighted Hardy space as the one associated with this choice:

\begin{notation}\label{notation:H1w}
From now on, we fix $w\in A_{\infty}$, $p_0\in \mathcal{W}_w(p_-(L),p_+(L))$, $\varepsilon_0>0$, and $M_0\in \N$ such that $M_0>\frac{n}{2}\left(r_w-\frac{1}{p_-(L)}\right)$ and set $H_L^1(w):=H_{L,p_0,\varepsilon_0,M_0}^1(w)$.
\end{notation}

\subsection{Weighted Hardy spaces associated with operators}
We next define other versions of the molecular weighted Hardy spaces defined above using different operators.

\begin{definition}[Weighted Hardy spaces associated with an operator]
Let $w\in A_\infty$ and take $q\in \mathcal{W}_w(p_-(L),p_+(L))$. Given a sublinear operator $\mathcal{T}$  acting on functions of $L^q(w)$ 
we define the 
weighted Hardy space $H^{1}_{\mathcal{T},q}(w)$  as the completion of the set
\begin{align}\label{defw}
\mathbb{H}^{1}_{\mathcal{T},q}(w):=\left\{f\in L^{q}(w):\mathcal{T}f\in L^1(w)\right\},
\end{align}
with respect to the norm
 \begin{align}\label{defw2}
\|f\|_{\mathbb{H}^{1}_{\mathcal{T},q}(w)}:=\|\mathcal{T}f\|_{L^1(w)}.
\end{align}
\end{definition} 

In our results $\mathcal{T}$ will be any of the square functions
presented in \eqref{square-H-1}--\eqref{square-P-3}, or the non-tangential maximal functions defined in \eqref{nontangential1}--\eqref{nontangential2}.

%%%%%%%%%%%%%%%%%%%%%%%%%%%%%%%%%%%%%%%%%%%%%%%%%%%%%%%%%%%%%%%%%%%%
%%%%%%%%%%%%%%%%%%%%%%%%%%%%%%%%%%%%%%%%%%%%%%%%%%%%%%%%%%%%%%%%%%%

\subsection{Main results}\label{section:DR}

\begin{theorem}\label{thm:hardychartz}
Given $w\in A_{\infty}$, let $H^1_L(w)$ be the fixed molecular Hardy space as in Notation \ref{notation:H1w}. For every $p\in\mathcal{W}_w(p_-(L),p_+(L))$, $\varepsilon>0$, and $ M\in \N$ such that $M>\frac{n}{2}\left(r_w-\frac{1}{p_-(L)}\right)$, the following spaces are isomorphic to $H^1_L(w)$ (and therefore one another) with equivalent norms
$$
H^1_{L,p,\varepsilon,M}(w);
\qquad
H^1_{\Scal_{m,\hh},p}(w),\ m\in \N;
\qquad
H^1_{\Grm_{m,\hh},p}(w),\ m\in \N_0;
\quad \textrm{and}
\qquad
H^1_{\Gcal_{m,\hh},p}(w),\ m\in \N_0.
$$
In particular, none of these spaces depend (modulo isomorphisms) on the choice of the allowable parameters $p$, $\varepsilon$, $M$, and $m$. 
\end{theorem}

\begin{theorem}\label{thm:hardychartzPoisson}
Given $w\in A_{\infty}$, let $H^1_L(w)$ be the fixed molecular Hardy space as in Notation \ref{notation:H1w}. For every $p\in\mathcal{W}_w(p_-(L),p_+(L))$, the following spaces are isomorphic to $H^1_L(w)$ (and therefore one another) with equivalent norms
$$
H^1_{\Scal_{K,\pp},p}(w),\ K\in \N;
\qquad
H^1_{\Grm_{K,\pp},p}(w),\ K\in \N_0;
\qquad\textrm{and}\qquad
H^1_{\Gcal_{K,\pp},p}(w),\ K\in \N_0.
$$
In particular, none of these spaces depend (modulo isomorphisms) on the choice of $p$, and $K$.
\end{theorem}

\begin{theorem}\label{thm:hardychartzNontangential}
Given $w\in A_{\infty}$, let $H^1_L(w)$ be the fixed molecular Hardy space as in Notation \ref{notation:H1w}. For every $p\in\mathcal{W}_w(p_-(L),p_+(L))$, the following spaces are isomorphic to $H^1_L(w)$ (and therefore one another) with equivalent norms
$$
H^1_{\Ncal_{\hh},p}(w)
\qquad\textrm{and}\qquad
H^1_{\Ncal_{\pp},p}(w).
$$
In particular, none of these spaces depend (modulo isomorphisms) on the choice of $p$.
\end{theorem}

\section{Auxiliary results}\label{section:aux}
In this section we introduce some notation and establish some auxiliary results that will be very useful in order to simplify the proofs of Theorems \ref{thm:hardychartz}, \ref{thm:hardychartzPoisson}, and \ref{thm:hardychartzNontangential}.

Let $\R_+^{n+1}$ be the upper-half space, that is, the set of points $(y,t)\in \R^n\times \R$ with $t>0$. Given $\alpha>0$ and $x\in \mathbb{R}^n$ we define the cone of aperture  $\alpha$ with vertex at $x$ by
\begin{align}\label{cone}
\Gamma^{\alpha}(x):=\{(y,t)\in \R_+^{n+1} : |x-y|<\alpha t\}.
\end{align}
When $\alpha=1$ we simply write $\Gamma(x)$. For a closed set $E$ in $\mathbb{R}^n$, set
\begin{align}\label{ralphacones}
\mathcal{R}^{\alpha}(E):=\bigcup_{x\in E}\Gamma^{\alpha}(x).
\end{align}
When $\alpha=1$ we simplify the notation by writing $\mathcal{R}(E)$ instead of $\mathcal{R}^1(E)$.

Besides, for a function $F$ defined in $\R^{n+1}_+$ and for every $x\in \R^n$,   let us consider
\begin{align}\label{normacono}
|\|F\||_{\Gamma(x)}:=\left(\iint_{\Gamma(x)}|F(y,t)|^2\frac{dy\,dt}{t^{n+1}}\right)^{\frac{1}{2}}.
\end{align}
Using  ideas from \cite[Lemma 5.4]{HofmannMayboroda}, we obtain the  following result:
\begin{lemma}\label{lema:comparacion,SH-GH,SK-GK}
For all $w\in A_{\infty}$ and  $f\in L^{2}(\R^n)$.
There hold
\begin{list}{$(\theenumi)$}{\usecounter{enumi}\leftmargin=1cm \labelwidth=1cm\itemsep=0.2cm\topsep=.2cm \renewcommand{\theenumi}{\alph{enumi}}}

\item $\|\Scal_{m,\hh}f\|_{L^{p}(w)}\lesssim \|\Grm_{m-1,\hh}f\|_{L^{p}(w)}$,
for all $m\in \N$ and $0<p<\infty$,

\item  $\|\Scal_{K,\pp}f\|_{L^{p}(w)}\lesssim \|\Grm_{K-1,\pp}f\|_{L^{p}(w)}$, for all $K\in \N$ and $0<p<\infty$.
\end{list}
Furthermore, one can see that  ($a$) and ($b$) hold for all functions $f\in L^q(w)$ with $w\in A_\infty$ and $q\in \mathcal{W}_{w}(p_-(L), p_+(L))$. 

\end{lemma}
\begin{proof}
We start by proving part ($a$).
Fix $x\in \R^n$ and $t>0$, and consider 
$$
B:=B(x, t),\quad  \widetilde{f}(y):=(t^2L)^{m-1}e^{-\frac{t^2}{2}L}f(y),\quad \textrm{and}\quad H(y):=\widetilde{f}(y)-(\widetilde{f})_{4B},
$$
where
$(\widetilde{f})_{4B}=\dashint_{4B}\widetilde{f}(y)\,dy$.
Then, applying the fact that $\{t^2Le^{-t^2L}\}_{t>0}\in \mathcal{F}_{\infty}(L^2\rightarrow L^2)$ and  that $t^2Le^{-t^2L}1=t^2L1=0$ (see \cite{Auscher}), we  obtain that
\begin{align*}
&\left(\int_{B}
|t^2Le^{-\frac{t^2}{2}L}\widetilde{f}(y)|^2\,dy\right)^{\frac{1}{2}}
=
\left(\int_{B}
|t^2Le^{-\frac{t^2}{2}L}H(y)|^2\,dy\right)^{\frac{1}{2}}
\\
&\qquad
\lesssim
\left(\int_{B}
|t^2Le^{-\frac{t^2}{2}L}(H\chi_{4B})(y)|^2\,dy\right)^{\frac{1}{2}}
+
\sum_{j\geq 2}\left(\int_{B}
|t^2Le^{-\frac{t^2}{2}L}(H\chi_{C_j(B)})(y)|^2\,dy\right)^{\frac{1}{2}}
\\
&\qquad
\lesssim
\left(\int_{4B}
|H(y)|^2\,dy\right)^{\frac{1}{2}}
+
\sum_{j\geq 2}e^{-c4^{j}}\left(\int_{2^{j+1}B}
|H(y)|^2\,dy\right)^{\frac{1}{2}}=:I+\sum_{j\geq 2}e^{-c4^{j}}I_j.
\\
\end{align*}
By Poincar\'e inequality, we conclude that
\begin{align*}
I\lesssim  t\left(\int_{8B}|\nabla_y \widetilde{f}(y)|^2\,dy\right)^{\frac{1}{2}},
\end{align*}
and that
\begin{align*}
I_j&
\lesssim 
\left(\int_{2^{j+1}B}|\widetilde{f}(y)-(\widetilde{f})_{2^{j+1}B}|^2\,dy\right)^{\frac{1}{2}}
+|2^{j+1}B|^{1/2}
\sum_{k=2}^{j}|(\widetilde{f})_{2^{k}B}-(\widetilde{f})_{2^{k+1}B}|
\\
&
\lesssim
|2^{j+1}B|^{1/2}
\sum_{k=2}^{j}\left(\dashint_{2^{k+1}B}|\widetilde{f}(y)-(\widetilde{f})_{2^{k+1}B}|^2\,dy\right)^{\frac{1}{2}}
\\
&
\lesssim
\sum_{k=2}^{j}2^{(j-k)n/2}2^{k} t\left(\int_{2^{k+2}B}|\nabla_y \widetilde{f}(y)|^2\,dy\right)^{\frac{1}{2}}.
\end{align*}
Then,
\begin{align*}
\left(\int_{B}
|t^2Le^{-\frac{t^2}{2}L}\widetilde{f}(y)|^2\,dy\right)^{\frac{1}{2}}
&
\lesssim  t\left(\int_{8B}|\nabla_y\widetilde{f}(y)|^2dy\right)^{\frac{1}{2}}
+
\sum_{j\geq 2} e^{-c4^j}\sum_{k=2}^{j}2^{\frac{n(j-k)}{2}+k}
 t\left(\int_{2^{k+2}B}| \nabla_y\widetilde{f}(y)|^2\,dy\right)^{\frac{1}{2}}
\\
&
\lesssim
\sum_{j\geq 1} e^{-c4^j}
 \left(\int_{2^{j+2}B}|t\nabla_y \widetilde{f}(y)|^2\,dy\right)^{\frac{1}{2}},
\end{align*}
and therefore 
\begin{align*}
\Scal_{m,\hh}f(x)\lesssim 
\sum_{j\geq 1} e^{-c4^j}
\Grm_{m-1,\hh}^{2^{j+3}}f(x),
\end{align*}
recall the definition of $\Grm_{m-1,\hh}^{2^{j+3}}$ in \eqref{squarealpha} and \eqref{square-H-2}.
Then, for every $0<p<\infty$ and $w\in A_{\infty}$, taking the $L^p(w)$ norm in both sides of the previous inequality and applying change of angles (see \cite[Proposition 3.2]{MartellPrisuelos}), 
we conclude that
\begin{align*}
\|\Scal_{m,\hh}f\|_{L^p(w)}\lesssim \sum_{j\geq 1} e^{-c4^j}
\left\|\Grm_{m-1,\hh}^{2^{j+3}}f\right\|_{L^p(w)}
\lesssim \|\Grm_{m-1,\hh}f\|_{L^p(w)}\sum_{j\geq 1}e^{-c4^j}
\lesssim \|\Grm_{m-1,\hh}f\|_{L^p(w)}.
\end{align*}

%%%%%%%%%%%%%%%%%%%%%%%%%%%%%%%%%%%%%%%%%%%%%%%%%%%%%%%%%%%%%%%%
\medskip

As for part ($b$), fix $w\in A_{\infty}$, $f\in L^2(\R^n)$, and $0<p<\infty$, and note that following the same argument of \cite[Lemma 5.4]{HofmannMayboroda}\footnote{We want to thank Steve Hofmann for sharing with us this argument that was omitted in \cite[Lemma 5.4]{HofmannMayboroda}.}, there exist a dimensional constant $k_0\in \N$ and $C_1>0$ such that for all $K\in \N$ and $k\in \N_0$.
$$
\Scal_{K,\pp}^{2^{k}}f(x)\leq C_1\left(\Grm_{K-1,\pp}^{2^{k+k_0}}f(x)\right)^{\frac{1}{2}}\left(\Scal_{K,\pp}^{2^{k+k_0}}f(x)\right)^{\frac{1}{2}},
$$
where recall the definitions of $\Scal_{K,\pp}^{2^{k}}$ and $\Grm_{K-1,\pp}^{2^{k+k_0}}$ in \eqref{squarealpha}, \eqref{square-P-1}, and \eqref{square-P-2}.
Now, for some $R>0$, to be determinate later, consider
$$
\Scal^*f(x):=\sum_{k=0}^{\infty}R^{-k}\Scal_{K,\pp}^{2^k}f(x)\quad \textrm{and}\quad
\Grm^*f(x):=\sum_{k=0}^{\infty}R^{-k}\Grm_{K-1,\pp}^{2^k}f(x).
$$
By the above inequality, and using Young's inequality, we have
\begin{align}\label{finiterevised}
\Scal^*f(x)&\leq \sum_{k=0}^{\infty}R^{-(k+k_0)}\left(C_1^2R^{2k_0}
\Grm_{K-1,\pp}^{2^{k+k_0}}f(x)\right)^{\frac{1}{2}}\left(\Scal_{K,\pp}^{2^{k+k_0}}f(x)\right)^{\frac{1}{2}}
\\\nonumber
&\leq
\frac{1}{2}\left(C_1^2R^{2k_0}\sum_{k=0}^{\infty}R^{-(k+k_0)}
\Grm_{K-1,\pp}^{2^{k+k_0}}f(x)+\sum_{k=0}^{\infty}R^{-(k+k_0)}
\Scal_{K,\pp}^{2^{k+k_0}}f(x)\right)
\\\nonumber
&\leq \frac{1}{2}\left(R^{2k_0}C_1^2\Grm^*f(x)+\Scal^*f(x)\right).
\end{align}
Besides, since $\Scal_{K,\pp}$ is bounded from $L^2(\R^n)$ to $L^2(\R^n)$ (see for instance  \cite{MartellPrisuelos}), applying \cite[Proposition 4, Section 3]{CoifmanMeyerStein} or \cite{Auscherangles}, and choosing $R>2^{\frac{n}{2}+1}$, we have that
$$
\|\Scal^*f\|_{L^2(\R^n)}\leq\sum_{k=0}^{\infty}R^{-k}\|\Scal_{K,\pp}^{2^k}f\|_{L^2(\R^n)}
\lesssim \sum_{k=0}^{\infty}R^{-k}2^{\frac{kn}{2}}\|\Scal_{K,\pp}f\|_{L^2(\R^n)}
\lesssim \sum_{k=0}^{\infty}R^{-k}2^{\frac{kn}{2}}\|f\|_{L^2(\R^n)}<\infty,
$$
hence $\Scal^*f(x)<\infty$ a. e. $x\in \R^n$.
Then, by \eqref{finiterevised},
\begin{align*}
\Scal_{K,\pp}f(x)\leq \Scal^{*}f(x)\leq CR^{2k_0} \Grm^*f(x).
\end{align*}
Hence,  taking the $L^p(w)$ norm in the previous inequality, by \cite[Proposition 3.29]{MartellPrisuelos}, we conclude that,
for $r_0>\max\{p/2,r_w\}$ and $R=2^{\frac{nr_0}{p}+1}>2^{\frac{n}{2}+1}$,
$$
\|\Scal_{K,\pp}f\|_{L^p(w)}\lesssim \sum_{k=0}^{\infty}R^{-(k-2k_0)}\|\Grm_{K-1,\pp}^{2^k}f\|_{L^p(w)}
\lesssim R^{2k_0}\sum_{k=0}^{\infty}R^{-k}2^{\frac{knr_0}{p}}
\|\Grm_{K-1,\pp}f\|_{L^p(w)}\lesssim \|\Grm_{K-1,\pp}f\|_{L^p(w)}.
$$
Following the explanation of \cite[Remark 4.22]{MartellPrisuelos} we conclude $(a)$ and $(b)$ for all functions $f\in L^q(w)$ with $w\in A_{\infty}$ and $q\in \mathcal{W}_w(p_-(L),p_+(L))$.
\end{proof}
%

%%%%%%%%%%%%%%%%%%%%%%%%%%%%%%%%%%%%%%%%%%%%%%%%%%%%%%%%%
%%%%%%%%%%%%%%%%%%%%%%%%%%%%%%%%%%%%%%%%%%%%%%%%%%%%%%%%%
To conclude this section we present some estimates for $\mol$s.
\begin{lemma}\label{lemma:m-h}
Given $p>p_0$,  $w\in A_{\frac{p}{p_0}}$, $\varepsilon>0$, and $M\in \N$, let $\mm$ be a $\mol$ and let $Q$ be a cube associated with $\mm$.  For every $i\geq 1$ and $k=0,1,2,\ldots, M$, there holds
\begin{list}{$(\theenumi)$}{\usecounter{enumi}\leftmargin=1cm \labelwidth=1cm\itemsep=0.2cm\topsep=.2cm \renewcommand{\theenumi}{\alph{enumi}}}

\item[]  $\left\|\left((\ell(Q)^2L)^{-k}\mm\right)\chi_{C_i(Q)}\right\|_{L^{p_0}(\R^n)}\lesssim 2^{-i\varepsilon}w(2^{i+1}Q)^{-1}|2^{i+1}Q|^{\frac{1}{p_0}}$.
\end{list}
\end{lemma}
\begin{proof}
 Using H\"older's inequality, \eqref{astast}, and the fact that $w\in A_{\frac{p}{p_0}}$, we have that
\begin{align*}
&\left\|\left((\ell(Q)^2L)^{-k}\mm\right)\chi_{C_i(Q)}\right\|_{L^{p_0}(\mathbb{R}^n)}
\\
&\qquad\leq \left(\int_{C_i(Q)}|(\ell(Q)^2L)^{-k}\mm(y)|^{p}w(y) \, dy\right)^{\frac{1}{p}}
\left(\dashint_{2^{i+1}Q}w(y)^{1-\left(\frac{p}{p_0}\right)'} \, dy\right)^{\frac{1}{p}\left(\frac{p}{p_0}-1\right)}|2^{i+1}Q|^{
\frac{1}{p_0}-\frac{1}{p}}
\\
&\qquad\lesssim 2^{-i\varepsilon}w(2^{i+1}Q)^{\frac{1}{p}-1}
\left(\dashint_{2^{i+1}Q}w(y) \, dy\right)^{-\frac{1}{p}}|2^{i+1}Q|^{
\frac{1}{p_0}-\frac{1}{p}}
\\
& \qquad\lesssim 2^{-i\varepsilon}
w(2^{i+1}Q)^{-1}|2^{i+1}Q|^{\frac{1}{p_0}}.
\end{align*}
\end{proof}

%%%%%%%%%%%%%%%%%%%%%%%%%%%%%%%%%%%%%%%%%%%%%%%%%%%%%%%%%%%%%%%%%%%%%%%%%%%
%%%%%%%%%%%%%%%%%%%%%%%%%%%%%%%%%%%%%%%%%%%%%%%%%%%%%%%%%%%%%%%%%%%%%%%%%%%
\section{Characterization of the weighted Hardy spaces defined by square functions associated with the heat semigroup}\label{SH}
%%%%%%%%%%%%%%%%%%%%%%%%%%%%%%%%%%%%%%%%%%%%%%%%%%%%%%%%%%%%%%%%%%%%%%%%%%%%%%
%%%%%%%%%%%%%%%%%%%%%%%%%%%%%%%%%%%%%%%%%%%%%%%%%%%%%%%%%%%%%%%%%%%%%%%%%%%%%%

Theorem \ref{thm:hardychartz} follows at once from the following proposition:

\begin{proposition}\label{lema:SH-1}
Let $w\in A_{\infty}$, $p,q\in \mathcal{W}_w(p_-(L),p_+(L))$,  $\varepsilon>0$, $K\in \N_0$, and $M\in \N$ be
 such that $M>\frac{n}{2}\left(r_w-\frac{1}{p_-(L)}\right)$. Then
\begin{list}{$(\theenumi)$}{\usecounter{enumi}\leftmargin=1cm \labelwidth=1cm\itemsep=0.2cm\topsep=.2cm \renewcommand{\theenumi}{\alph{enumi}}}

\item 
$
\mathbb{H}^1_{L,p,\varepsilon,M}(w)=\mathbb{H}^1_{\Scal_{m,\hh},p}(w)
$
with equivalent norms, for all $m\in \N$.

\item 
$
H^1_{\Scal_{m,\hh},p}(w)$ and $H_{\Scal_{m,\hh},q}^1(w)
$ are isomorphic, for all $m\in \N$.

\item  
 $
\mathbb{H}_{L,p,\varepsilon,M}^1(w)= \mathbb{H}_{\Grm_{m,\hh},p}^1(w)= \mathbb{H}_{\Gcal_{m,\hh},p}^1(w),
$  with equivalent norms, for all $m\in \N_{0}$.
\end{list}

\end{proposition}

In order to prove Proposition \ref{lema:SH-1} we need to show that, for $m\in \N_0$, the $L^1(w)$ norms of the square functions $\Scal_{m+1,\hh}$, $\Grm_{m,\hh}$, and  $\Gcal_{m,\hh}$, applied to $\mol$s, are uniformly controlled. Moreover, we shall show in Proposition \ref{prop:contro-mol-SF}  that all the square functions in \eqref{square-H-1}--\eqref{square-P-3} satisfy those uniform estimates. That proposition follows from the following general result:

\begin{proposition}\label{prop:acotacion-T}
Let $w\in A_{\infty}$ and let $\{\mathcal{T}_t\}_{t>0}$ be a family of sublinear operators satisfying the following conditions:

\begin{list}{$(\theenumi)$}{\usecounter{enumi}\leftmargin=1cm \labelwidth=1cm\itemsep=0.2cm\topsep=.2cm \renewcommand{\theenumi}{\alph{enumi}}}

\item $\{\mathcal{T}_t\}_{t>0}\in \mathcal{F}_{\infty}(L^{p_0}\rightarrow L^2)$ for all $p_-(L)<p_0\leq 2$.

\item $\widehat{S}f(x):=\left(\iint_{\Gamma(x)}|\mathcal{T}_tf(y)|^2\frac{dy \ dt}{t^{n+1}}\right)^{\frac{1}{2}}$ is bounded on $L^p(w)$ for every $p\in \mathcal{W}_w(p_-(L),p_+(L))$.

\item There exists $C>0$ so that for every $t>0$ there holds $\mathcal{T}_{t}=C\,\mathcal{T}_{\frac{t}{\sqrt{2}}}\circ e^{-\frac{t^2}{2}L}$.

\item For every $\lambda>0$, there exists $C_\lambda>0$ such that for every $t>0$ it follows that
$$
\mathcal{T}_{\sqrt{1+\lambda}\,t}=C_\lambda\,\mathcal{T}_{t}\circ e^{-\lambda t^2L}.
$$
\end{list}
Then, for every $\mm$, a $\mol$ with $p\in \mathcal{W}_w(p_-(L),p_+(L))$, $\varepsilon>0$, and $M>\frac{n}{2}\left(r_w-\frac{1}{p_-(L)}\right)$, it follows that $
\|\widehat{S}\mm\|_{L^1(w)}\lesssim 1,
$
with constants independent of $\mm$.
\end{proposition}
Assuming this result momentarily we obtain the following:
\begin{proposition}\label{prop:contro-mol-SF}
Let $S$ be any of the square functions considered in \eqref{square-H-1}--\eqref{square-P-3}. For every $w\in A_{\infty}$ and $\mm$ a $\mol$ with $p\in \mathcal{W}_w(p_-(L),p_+(L))$, $\varepsilon>0$, and  $M>\frac{n}{2}\left(r_w-\frac{1}{p_-(L)}\right)$, there hold
\begin{list}{$(\theenumi)$}{\usecounter{enumi}\leftmargin=1cm \labelwidth=1cm\itemsep=0.2cm\topsep=.2cm \renewcommand{\theenumi}{\alph{enumi}}}
\item$
\|S \mm\|_{L^1(w)}\leq C.
$

\item For all $f\in \mathbb{H}_{L,p,\varepsilon,M}^1(w)$, 
$
\|S f\|_{L^1(w)}\lesssim \|f\|_{\mathbb{H}_{L,p,\varepsilon,M}^1(w)}.
$
\end{list}
\end{proposition}
\begin{proof}
Assuming $(a)$ let us prove $(b)$. Fix $w\in A_{\infty}$ and take $p\in \mathcal{W}_w(p_-(L),p_+(L))$, $\varepsilon>0$, and $M\in \N$ such that $M>\frac{n}{2}\left(r_w-\frac{1}{p_-(L)}\right)$. Then,
for $f\in \mathbb{H}_{L,p,\varepsilon,M}^1(w)$,
 there exists a $\p$ of $f$, $f=\sum_{i=1}^{\infty}{\lambda}_i {\mm}_i$, such that
 $$
 \sum_{i=1}^{\infty}|{\lambda}_i| \leq 2\|f\|_{\mathbb{H}_{L,p,\varepsilon,M}^1(w)}.
 $$
On the other hand, since $\sum_{i=1}^{\infty}{\lambda}_i {\mm}_i$ converges in $L^p(w)$ and since for any choice of $S$, we have that $S$ is a sublinear operator bounded on $L^p(w)$ (see \cite[Theorems 1.12 and 1.13]{MartellPrisuelos}) and by part ($a$), we have
\begin{align*}
\|Sf\|_{L^1(w)}=\left\|S\left(\sum_{i=1}^{\infty}{\lambda}_i {\mm}_i\right)\right\|_{L^1(w)}
\leq \sum_{i=1}^{\infty}|{\lambda}_i|\, \|S {\mm}_i\|_{L^1(w)}\leq C \sum_{i=1}^{\infty}|{\lambda}_i|
\lesssim  \|f\|_{\mathbb{H}_{L,p,\varepsilon,M}^1(w)},
\end{align*}
as desired.

As for part ($a$), we
first show the desired estimate for $\Gcal_{\hh}$. To this end, notice that
$
|t\nabla_{y,t}e^{-t^2L}f|^2= |t\nabla_{y}e^{-t^2L}f|^2+4|t^2Le^{-t^2L}f|^2.
$
Besides, both
$
\mathcal{T}_t:=t\nabla_{y}e^{-t^2L}$
and 
$
\mathcal{T}_t:=t^2Le^{-t^2L}
$
satisfy the hypotheses of Proposition \ref{prop:acotacion-T}:
 $(a)$ follows from the off-diagonal estimates satisfied by the families $\{t\nabla_ye^{-t^2L}\}_{t>0}$ and $\{t^2Le^{-t^2L}\}_{t>0}$  (see Section \ref{section:OD}); $(b)$ is contained in \cite[Theorem 1.12, part $(a)$]{MartellPrisuelos}; and finally $(c)$ and $(d)$ follow from easy calculations. Thus we can apply  Proposition \ref{prop:acotacion-T} and obtain the desired estimate for $\Gcal_{\hh}$.

To obtain the estimates for the other square functions we can use \cite[Theorems 1.14 and 1.15, Remark 4.22]{MartellPrisuelos}, and the fact that $\Scal_{\hh}f\leq \frac{1}{2}\Gcal_{\hh}f$. Easy details are left to the interested reader.
\end{proof}

\subsection{Proof of Proposition \ref{prop:acotacion-T}}
Fix $w\in A_{\infty}$, $p\in \mathcal{W}_w(p_-(L),p_+(L))$, $\varepsilon>0$, $M>\frac{n}{2}\left(r_w-\frac{1}{p_-(L)}\right)$, 
and $\mm$ a $\mol$.
Let $Q$ be a cube associated with $\mm$. Since $w\in A_{\frac{p}{p_-(L)}}$ we can pick $p_-(L)<p_0<2$, close enough to $p_-(L)$, so that $w\in A_{\frac{p}{p_0}}$ and simultaneously
\begin{align}\label{choice-M}
M>\frac{n}{2}\left(\frac{r_wp_0}{p_-(L)}-\frac{1}{p_0}\right).
\end{align}

For every $j,i\geq 1$, consider $Q_i:=2^{i+1}Q$,  $\mm_i:=\mm\chi_{C_i(Q)}$, and $C_{ji}:=C_j(Q_i)$.
Note that
$$
|\mathcal{T}_t\mm(y)|\leq|\mathcal{T}_t\mm(y)|\chi_{(0,\ell(Q))}(t)+|\mathcal{T}_t\mm(y)|\chi_{[\ell(Q),\infty)}(t)=:F_1(y,t)+F_2(y,t),
$$
and therefore, recalling \eqref{normacono},
\begin{align*}
\|\widehat{S}\mm\|_{L^1(w)}
\leq
\big\||\|F_1\||_{\Gamma(\cdot)}\big\|_{L^1(w)}+\big\||\|F_
2\||_{\Gamma(\cdot)}\big\|_{L^1(w)}
 =: I+II.
\end{align*}
We estimate each term in turn. Note first that
$$
F_1(y,t)
\leq\sum_{i\geq 1}|\mathcal{T}_t\mm_i(y)|\chi_{(0,\ell(Q))}(t)=:\sum_{i\geq 1}F_{1,i}(y,t).
$$
Then,
\begin{align}\label{acotacionmol:I}
I\lesssim \sum_{i\geq 1}\big\||\|F_{1,i}\||_{\Gamma(\cdot)}\big\|_{L^1(16Q_i,w)}
+ \sum_{j\geq 4}\sum_{i\geq 1}\big\||\|F_{1,i}\||_{\Gamma(\cdot)}\big\|_{L^1(C_{ji},w)}
=:\sum_{i\geq 1}I_i + \sum_{j\geq 4}\sum_{i\geq 1}I_{ji}.
\end{align}
For $I_i$, apply H\"older's inequality, hypothesis $(b)$, \eqref{doublingcondition}, and \eqref{astast} (for $k=0$), to obtain
\begin{align}\label{Ii}
I_i
\le
\|\widehat{S}\mm_i\|_{L^1(16Q_i,w)}
\lesssim
w(16Q_i)^{\frac1{p'}}
\|\widehat{S}\mm_i\|_{L^p(w)}
\lesssim w(Q_i)^{\frac1{p'}} \|\mm_i\|_{L^p(w)}
\le 2^{-i\varepsilon}.
\end{align}
To estimate $I_{ji}$,  note that, for every $j\geq 4$ and $i\geq 1$, $0< t< \ell(Q)$, and $x\in C_{ji}$, it follows that
$
B(x,t)\subset 2^{j+2}Q_i\setminus 2^{j-1}Q_i.
$
This, hypothesis $(a)$, and Lemma \ref{lemma:m-h} imply that
\begin{multline*}
 \left(\int_{B(x,t)}|\mathcal{T}_t\mm_i(y)|^2\,dy\right)^{\frac{1}{2}}\leq
  \left(\int_{2^{j+2}Q_i\setminus 2^{j-1}Q_i}|\mathcal{T}_t\mm_i(y)|^2\,dy\right)^{\frac{1}{2}}
\\
\leq t^{-n\left(\frac{1}{p_0}-\frac{1}{2}\right)}e^{-c\frac{4^j\ell(Q_i)^2}{t^2}}
 \|\mm_i\|_{L^{p_0}(\mathbb{R}^n)}
 \lesssim
t^{-n\left(\frac{1}{p_0}-\frac{1}{2}\right)}e^{-c\frac{4^j\ell(Q_i)^2}{t^2}} 2^{-i\varepsilon}
w(Q_i)^{-1}|Q_i|^{\frac{1}{p_0}}.
\end{multline*}
Then,  \eqref{doublingcondition} and easy calculations lead to
\begin{multline*}
I_{ji}
\lesssim 2^{-i\varepsilon}
w(Q_i)^{-1}|Q_i|^{\frac{1}{p_0}}
\int_{C_{ji}}\left(\int_{0}^{\ell(Q)}
t^{-2n\left(\frac{1}{p_0}-\frac{1}{2}\right)}e^{-c\frac{4^j\ell(Q_i)^2}{t^2}}
\frac{dt}{t^{n+1}}\right)^{\frac{1}{2}}w(x)dx
\\
%&=2^{-i\varepsilon}
%w(Q_i)^{-1}|Q_i|^{\frac{1}{p_0}}w(2^{j+1}Q_i)
%\left(\int_0^{\ell(Q)}t^{-2n\left(\frac{1}{p_0}-\frac{1}{2}\right)}e^{-c\frac{4^j\ell(Q_i)^2}{t^2}}
%\frac{dt}{t^{n+1}}\right)^{\frac{1}{2}}
%\\
%\nonumber
%&\leq 2^{-i\varepsilon}2^{jnr}
%w(Q_i)^{-1}|Q_i|^{\frac{1}{p_0}}w(2^{j+1}Q_i)
%\left(\int_0^{\ell(Q)}t^{-\frac{2n}{p_0}}e^{-c\frac{4^j\ell(Q_i)^2}{t^2}}
%\frac{dt}{t}\right)^{\frac{1}{2}}
%\\
\lesssim
2^{-i\varepsilon}
w(Q_i)^{-1}|Q_i|^{\frac{1}{p_0}}w(2^{j+1}Q_i)
\big(4^j\,\ell(Q_i)^2)^{-\frac{n}{2p_0}}
\left(\int_{2^{j+i}}^{\infty}s^{\frac{2n}{p_0}}e^{-cs^2}
\frac{d s}{s}\right)^{\frac{1}{2}}
\lesssim 2^{-i\varepsilon}
e^{-c4^{j}}.
\end{multline*}
Plugging this and \eqref{Ii} into \eqref{acotacionmol:I}, we finally conclude the desired estimate for $I$:
\begin{align}\label{acotacionmoleculasI}
I\leq \sum_{i\geq 1}2^{-i\varepsilon}+\sum_{j\geq 4} \sum_{i\geq 1} 2^{-i\varepsilon}
e^{-c4^{j}}\lesssim 1.
\end{align}

We turn now to estimate $II$. First, set
$$
B_Q:=\left(I-e^{-\ell(Q)^2L}\right)^M \qquad \textrm{and}\qquad A_{Q}:=I-B_Q,
$$
and observe that
\begin{align}\label{acotacionmoleculasII}
F_2(y,t)
\le
|\mathcal{T}_tA_Q\mm(y)|\chi_{[\ell(Q),\infty)}(t)+
|\mathcal{T}_tB_Q\mm(y)|\chi_{[\ell(Q),\infty)}(t)=:F_3(y,t)+F_4(y,t).
\end{align}
We start estimating the term related to $F_3$. To do that, consider
$$
h(y):=\sum_{i\ge 1} h_i(y)
:=
\sum_{i\ge 1}((\ell(Q)^2L)^{-M}\mm(y)) \,\chi_{C_i(Q)}(y)
,
$$
and note that
$$
F_3(y,t)
\leq
\sum_{i\geq 1}|\mathcal{T}_tA_Q(\ell(Q)^2L)^Mh_i(y)|\chi_{[\ell(Q),\infty)}(t).
$$
Then, we obtain
\begin{align*}
\big\||\|F_3\||_{\Gamma(\cdot)}\big\|_{L^1(w)}
&\leq
\sum_{i\geq1} \left\|\left(\iint_{\Gamma(\cdot)}|\mathcal{T}_tA_Q(\ell(Q)^2L)^Mh_i(y)|^2\chi_{[\ell(Q),\infty)}(t)\frac{dy\,dt}{t^{n+1}}\right)^{\frac{1}{2}}\right\|_{L^1(16Q_i,w)}
\\
&\qquad\qquad+\sum_{j\geq 4}
\sum_{i\geq 1} \left\|\left(\iint_{\Gamma(\cdot)}|\mathcal{T}_tA_Q(\ell(Q)^2L)^Mh_i(y)|^2\chi_{[\ell(Q),\infty)}(t)\frac{dy\,dt}{t^{n+1}}\right)^{\frac{1}{2}}\right\|_{L^1(C_{ji},w)}
\\
&=:
\sum_{i\geq 1}II_i
+\sum_{j\geq 4}
\sum_{i\geq 1}II_{ji}.
\end{align*}
Before estimating $II_i$ and $II_{ji}$, note that by \cite[Proposition 5.8]{AuscherMartell:II} one can easily obtain that the operator $A_Q(\ell(Q)^2L)^M$ is bounded on $L^p(w)$ uniformly on $Q$ since $p\in \mathcal{W}_w(p_-(L),p_+(L))$ and
\begin{align*}
A_Q(\ell(Q)^2L)^M
=
(I-(I-e^{-\ell(Q)^2L})^M)(\ell(Q)^2L)^M
=\sum_{k=1}^{M}C_{k,M} (k\ell(Q)^2 L)^M e^{-k\ell(Q)^2L}.
\end{align*}
This, H\"older's inequality,  hypothesis $(b)$, \eqref{doublingcondition}, and \eqref{astast} imply
\begin{align}\label{Ii-1}
II_i
\leq
w(16Q_i)^{\frac1{p'}}\,\|A_Q(\ell(Q)^2L)^Mh_i\|_{L^p(w)}
\lesssim
w(Q_i)^{\frac1{p'}}\, \|h_i\|_{L^p(w)}
\leq
2^{-i\varepsilon}.
\end{align}
We turn now to estimate $II_{ji}$. Note that for every $x\in C_{ji}$, $j\geq 4$, $i\geq 1$
$$
\big\{(y,t):y\in B(x,t), \ t\geq\ell(Q)\big\}
\subset E_1\cup E_2\cup E_3,
$$
where
$$
E_1:=\big(2^{j+2}Q_i\setminus 2^{j-1}Q_i\big)\times \big[\ell(Q),2^{j-2}\ell(Q_i)\big],
\qquad
E_2:=
2^{j}Q_i\times \big(2^{j-2}\ell(Q_i),\infty\big),
$$
and
$$
E_3:=
\Big(\bigcup_{l\geq j}C_l(Q_i)\Big)\times \big(2^{j-2}\ell(Q_i),\infty\Big).
$$
Consequently,
\begin{align*}
II_{ji}\leq w(2^{j+1}Q_i)\sum_{l=1}^{3}\left(\iint_{E_l}|\mathcal{T}_tA_Q(\ell(Q)^2L)^Mh_i(y)|^2\frac{dy\,dt}{t^{n+1}}\right)^{\frac{1}{2}}
=:
w(2^{j+1}Q_i)\sum_{l=1}^{3}G_l.
 \end{align*}
Now observe that hypothesis $(c)$ implies
$$
|\mathcal{T}_tA_Q(\ell(Q)^2L)^Mh_i|=C\,|\mathcal{T}_{\frac{t}{\sqrt{2}}}e^{-\frac{t^2}{2}L}A_Q(\ell(Q)^2L)^Mh_i|.
$$
Besides,
\begin{align*}
e^{-\frac{t^2}{2}L}A_Q(\ell(Q)^2L)^M
=
\sum_{k=1}^M C_{k,M}\left(\frac{\ell(Q)^2}{s_{Q,t}^2}\right)^M\big(s_{Q,t}^2L\big)^Me^{-s_{Q,t}^2L},
\quad\textrm{where}\quad s_{Q,t}:=\left(k\ell(Q)^2+\frac{t^2}{2}\right)^{\frac{1}{2}}.
\end{align*}
Then, applying hypothesis $(a)$, the fact that $\{(t^2 L)^M e^{-t^2L}\}_{t>0}\in \mathcal{F}_{\infty}(L^{p_0}\rightarrow L^{p_0})$ together with \cite[Lemma 2.1]{MartellPrisuelos} (see also \cite[Lemma 2.3]{HofmannMartell}), and Lemma \ref{lemma:m-h},  we have
\begin{align*}
G_1
& \lesssim \sum_{k=1}^M\left(\int_{\ell(Q)}^{2^{j-2}\ell(Q_i)}
\Bigg(\frac{\ell(Q)^2}{s_{Q,t}^2}\Bigg)^{2M}
\int_{2^{j+2}Q_i\setminus 2^{j-1}Q_i}
|\mathcal{T}_{\frac{t}{\sqrt{2}}}\big(s_{Q,t}^2L \big)^M e^{-s_{Q,t}^2L}h_i(y)|^2dy \frac{dt}{t^{n+1}}\right)^{\frac{1}{2}}
\\
& \lesssim
\left(\int_{\ell(Q)}^{2^{j-2}\ell(Q_i)}
\ell(Q)^{4M}t^{-4M-\frac{2n}{p_0}
}e^{-c\frac{4^{j+i}\ell(Q)^2}{t^2}}
\frac{dt}{t}\right)^{\frac{1}{2}}2^{-i\varepsilon}w(Q_i)^{-1}|Q_i|^{\frac{1}{p_0}}
%\\
%&\lesssim
%\left(\int_{2}^{2^{j+i}}
%s^{4M+\frac{2n}{p_0}
%}e^{-cs^2}
%\frac{ds}{s}\right)^{\frac{1}{2}}(2^{j+i}\ell(Q))^{-\left(2M+\frac{n}{p_0}\right)}
%\ell(Q)^{2M}2^{-i\varepsilon}w(Q_i)^{-1}|Q_i|^{\frac{1}{p_0}}
\\
&\lesssim
2^{-j\left(2M+\frac{n}{p_0}\right)}2^{-i(2M+\varepsilon)}
w(Q_i)^{-1}.
 \end{align*}
Similarly,
\begin{align*}
G_2
\lesssim
\left(\int_{2^{j-2}\ell(Q_i)}^{\infty}
\ell(Q)^{4M}t^{-4M-\frac{2n}{p_0}}
\frac{dt}{t}\right)^{\frac{1}{2}}2^{-i\varepsilon}w(Q_i)^{-1}|Q_i|^{\frac{1}{p_0}}
 \lesssim
2^{-j\left(2M+\frac{n}{p_0}\right)}
2^{-i(2M+\varepsilon)}w(Q_i)^{-1},
 \end{align*}
 and
\begin{multline*}
G_3
\lesssim\sum_{l\geq j}
\left(\int_0^{\infty}
s^{4M+\frac{2n}{p_0}
}e^{-cs^2}
\frac{ds}{s}\right)^{\frac{1}{2}}(2^{(l+i)}\ell(Q))^{-\left(2M+\frac{n}{p_0}\right)}
\ell(Q)^{2M}2^{-i\varepsilon}w(Q_i)^{-1}|Q_i|^{\frac{1}{p_0}}
\\
\lesssim
2^{-j\left(2M+\frac{n}{p_0}\right)}
2^{-i(2M+\varepsilon)}w(Q_i)^{-1}.
 \end{multline*}
Collecting the estimates for $G_1$, $G_2$, and $G_3$ gives us
$$
II_{ji}\lesssim \frac{w(2^{j+1}Q_i)}{w(Q_i)}2^{-j\left(2M+\frac{n}{p_0}\right)}
2^{-i(2M+\varepsilon)}
\lesssim
2^{-j\left(2M+\frac{n}{p_0}-\frac{r_wp_0n}{p_-(L)}\right)}
2^{-i(2M+\varepsilon)},
$$
where we have used that $w\in A_{\frac{r_wp_0}{p_-(L)}}$, by the definition of $r_w$ and the fact that $p_-(L)<p_0$, and \eqref{doublingcondition}.
By this and by \eqref{Ii-1}, we conclude that \eqref{choice-M} yields
\begin{align}\label{acotacionmoleculasF3}
 \big\||\|F_3\||_{\Gamma(\cdot)}\big\|_{L^1(w)}\lesssim \sum_{i\geq 1}2^{-i\varepsilon}+\sum_{j\geq 4}\sum_{i\geq 1}2^{-j\left(2M+\frac{n}{p_0}-\frac{r_wp_0n}{p_-(L)}\right)}
2^{-i(2M+\varepsilon)}\lesssim 1.
 \end{align}

We next estimate $F_4$:
\begin{align*}
\big\||\|F_4\||_{\Gamma(\cdot)}\big\|_{L^1(w)}
&\leq
\sum_{i\geq 1} \left\|\left(\iint_{\Gamma(\cdot)}|\mathcal{T}_tB_Q\mm_i(y)|^2\chi_{[\ell(Q),\infty)}(t)\frac{dy\,dt}{t^{n+1}}\right)^{\frac{1}{2}}\right\|_{L^1(16Q_i,w)}
\\
&\qquad\quad+
\sum_{i\geq 1}\sum_{j\geq 4} \left\|\left(\iint_{\Gamma(\cdot)}|\mathcal{T}_tB_Q\mm_i(y)|^2\chi_{[\ell(Q),\infty)}(t)\frac{dy\,dt}{t^{n+1}}\right)^{\frac{1}{2}}\right\|_{L^1(C_{ji},w)}
\\
&=:
\sum_{i\geq 1}III_i
+
\sum_{i\geq 1}\sum_{j\geq 4}III_{ji}.
\end{align*}

Note that the fact that the semigroup $\{e^{-tL}\}_{t>0}$ is uniformly bounded on $L^p(w)$, since it was assumed that $p\in \mathcal{W}_w(p_-(L),p_+(L))$ (see \cite[Proposition 5.8]{AuscherMartell:II}), easily gives that $B_Q$ is bounded on $L^p(w)$ uniformly in $Q$. Hence, H\"older's inequality,
hypothesis $(b)$, and \eqref{astast} (for $k=0$), yield 
\begin{align}\label{IIIi-1}
III_{i}\lesssim w(16Q_i)^{\frac1{p'}}\,\|\widehat{S} B_Q \mm_i\|_{L^p(w)}
\lesssim w(16Q_i)^{\frac1{p'}}\, \|\mm_i\|_{L^p(w)}
\lesssim 2^{-i\varepsilon}.
\end{align}
Now, change the variable $t$ into $\sqrt{1+M}t$ and use hypothesis $(d)$ to obtain
\begin{multline*}
III_{ji}
\lesssim \left\|
\left(\iint_{\Gamma^{\sqrt{1+M}}(\cdot)}
|\mathcal{T}_{\sqrt{1+M}t}B_Q\mm_i(y)|^2\chi_{[\ell(Q)/\sqrt{1+M},\infty)}(t)\frac{dy\,dt}{t^{n+1}}
\right)^{\frac{1}{2}}\right\|_{L^1(C_{ji},w)}
\\
\approx
\left\|
\left(\iint_{\Gamma^{\sqrt{1+M}}(\cdot)}
|\mathcal{T}_{t}e^{-Mt^2L}B_Q\mm_i(y)|^2\chi_{[\ell(Q)/\sqrt{1+M},\infty)}(t)\frac{dy\,dt}{t^{n+1}}
\right)^{\frac{1}{2}}\right\|_{L^1(C_{ji},w)}.
\end{multline*}
Now, note that for  $\widehat{E}_1, \widehat{E}_2$ closed subsets in $\mathbb{R}^n$, and $f\in L^{p_0}(\mathbb{R}^n)$ such that $\textrm{supp}(f)\subset \widehat{E}_1$,  we have
\begin{align}\label{AB}
&\left\|e^{-Mt^2L}B_Qf\right\|_{L^{p_0}(\widehat{E}_2)}
=
\left\|(e^{-t^2L}-e^{-(t^2+\ell(Q)^2)L})^M f\right\|_{L^{p_0}(\widehat{E}_2)}
=
\left\|\left(\int_{0}^{\ell(Q)^2}\partial_r e^{-(r+t^2)L}\,dr\right)^M f\,  \right\|_{L^{p_0}(\widehat{E}_2)}
\\ \nonumber
&\leq
\int_{0}^{\ell(Q)^2}\!\!\!\cdots \int_{0}^{\ell(Q)^2}\left\|
\big((r_1+\dots+r_M+M t^2) L \big)^M e^{-(r_1+\dots+r_M+M t^2) L}f\right\|_{L^{p_0}(\widehat{E}_2)} \ \frac{dr_1 \cdots dr_M}{(r_1+\dots+r_M+M t^2)^M}
\\ \nonumber
&\lesssim
\int_{0}^{\ell(Q)^2}\!\!\!\cdots \int_{0}^{\ell(Q)^2}
e^{-c\frac{d(\widehat{E}_1,\widehat{E}_2)^2}{r_1+\dots+r_M+M t^2}}\frac{dr_1 \cdots dr_M}{(r_1+\dots+r_M+M t^2)^M}\|f\|_{L^{p_0}(\widehat{E}_1)}
\\ \nonumber
&
\lesssim
\left(\frac{\ell(Q)^2}{t^2}\right)^M
e^{-c\frac{d(\widehat{E}_1,\widehat{E}_2)^2}{t^2+\ell(Q)^2}}\|f\|_{L^{p_0}(\widehat{E}_1)},
\end{align}
where we have used that  $\{(t^2 L)^M e^{-t^2L}\}_{t>0}\in \mathcal{F}_{\infty}(L^{p_0}\rightarrow L^{p_0})$ since $p_-(L)<p_0<2<p_+(L)$.

On the other hand, setting $\theta_M=(1+M)^{-\frac12}$, for every $x\in C_{ji}$, we have
$$
\left\{(y,t):y\in B(x,\theta_M^{-1}t),\ \theta_M\ell(Q)\leq t<\infty\right\}
\subset \widetilde{E}_1\cup \widetilde{E}_2\cup \widetilde{E}_3,
$$
where
$$
\widetilde{E}_1:
=
\big(2^{j+2}Q_i\setminus 2^{j-1}Q_i\big)
\times
\big[\theta_M \ell(Q),2^{j-2}\theta_M\ell(Q_i)\big],
\qquad
\widetilde{E}_2:
=
2^{j}Q_i\times \big(2^{j-2}\theta_M\ell(Q_i),\infty\big),
$$
and
$$
\widetilde{E}_3:=
\Big(\bigcup_{l\geq j}C_l(Q_i)\Big)\times \big(
2^{j-2}\theta_M \ell(Q_i),\infty\big).
$$
Then we have that
\begin{align*}
III_{ji}
\lesssim w(2^{j+1}Q_i)
\sum_{l=1}^{3}
\left(\iint_{\widetilde{E}_l}|\mathcal{T}_{t}e^{-Mt^2L}B_Q\mm_i(y)|^2\frac{dy\,dt}{t^{n+1}}\right)^{\frac{1}{2}}=:w(2^{j+1}Q_i)
\sum_{l=1}^{3}\widetilde{G}_l.
\end{align*}
At this point we proceed much as in the estimates of $G_1$, $G_2$, and $G_3$. Applying \eqref{AB}, we obtain that
$$
III_{ji}\lesssim \frac{w(2^{j+1}Q_i)}{w(Q_i)}2^{-j\left(2M+\frac{n}{p_0}\right)}
2^{-i(2M+\varepsilon)}
\lesssim
2^{-j\left(2M+\frac{n}{p_0}-\frac{r_wp_0n}{p_-(L)}\right)}
2^{-i(2M+\varepsilon)},
$$
where we have used that $w\in A_{\frac{r_wp_0}{p_-(L)}}$, by the definition of $r_w$ and the fact that $p_-(L)<p_0$, and \eqref{doublingcondition}.
From this and \eqref{IIIi-1}, we conclude that \eqref{choice-M} yields
\begin{align*}
 \big\||\|F_4\||_{\Gamma(\cdot)}\big\|_{L^1(w)}\lesssim \sum_{i\geq 1}2^{-i\varepsilon}+\sum_{i\geq 1}\sum_{j\geq 4}2^{-j\left(2M+\frac{n}{p_0}-\frac{r_wp_0n}{p_-(L)}\right)}
2^{-i(2M+\varepsilon)}\lesssim 1.
 \end{align*}
By this, \eqref{acotacionmoleculasF3}, and  \eqref{acotacionmoleculasII}, we conclude that
$
II\lesssim 1,
$
which, together with \eqref{acotacionmoleculasI}, gives the desired estimate:
$
\|\widehat{S}\mm\|_{L^1(w)}\lesssim 1.
$
\qed

We devote the remaining of this section to proving Proposition \ref{lema:SH-1}.

\subsection{Proof of Proposition \ref{lema:SH-1}, part ($a$)}
Fix $w\in A_{\infty}$, $p\in \mathcal{W}_w(p_-(L),p_+(L))$, $\varepsilon>0$, and $m, M\in \N$ such that $M>\frac{n}{2}\left(r_w-\frac{1}{p_-(L)}\right)$.

 For all $f\in \mathbb{H}_{L,p,\varepsilon,M}^1(w)$, applying Proposition \ref{prop:contro-mol-SF}, we obtain that
\begin{align}\label{SH-HL}
\|\Scal_{m,\hh}f\|_{L^1(w)}\lesssim \|f\|_{\mathbb{H}_{L,p,\varepsilon,M}^1(w)}.
\end{align}
Then, since in particular $f\in L^p(w)$, we conclude that $f\in \mathbb{H}^1_{\Scal_{m,\hh},p}(w)$, and hence 
$
\mathbb{H}^1_{L,p,\varepsilon,M}(w)\subset \mathbb{H}^1_{\Scal_{m,\hh},p}(w).
$

As for proving the converse inclusion,
we shall show that for all
$f\in \mathbb{H}^1_{\Scal_{m,\hh},p}(w)$  we can find a $\p$ of $f$, i.e. $f=\sum_{i=1}^{\infty}\lambda_i\mm_i$,
 such that
$$
\sum_{i=1}^{\infty}|\lambda_i|\lesssim \|\Scal_{m,\hh}f\|_{L^1(w)}.
$$
Following some ideas of \cite[Lemma 4.2]{HofmannMayboroda},
 for each $l\in \Z$ and for some $0<\gamma<1$ to be chosen later, we set
$$
O_l:=\{x\in \mathbb{R}^n:\Scal_{m,\hh}f(x)>2^l\}, \quad E_l^*:=\left\{x\in \mathbb{R}^n: \frac{|E_l\cap B(x,r)|}{|B(x,r)|}\geq \gamma,\, \textrm{for all}\,r>0\right\},
$$
$E_l:=\R^n\setminus O_l,$ and
$
O_l^*:=\R^n\setminus E_l^*=\left\{x\in \mathbb{R}^n: \mathcal{M}(\chi_{O_l})(x)>1-\gamma\right\},
$
where $\mathcal{M}$ is the centered Hardy-Littlewood maximal operator.
We have that $O_l$ and $O_l^*$
are open, and that $O_{l+1}\subseteq O_l$,
 $O_{l+1}^*\subseteq O_{l}^*$, and $O_l\subseteq O_l^*$.
 Besides, since $w\in A_{\infty}$ then $\mathcal{M}:L^r(w)\rightarrow L^{r,\infty}(w)$, for every $r>r_w$. Also, $\|\Scal_{m,\hh}f\|_{L^p(w)}\lesssim \|f\|_{L^p(w)}<\infty$, because $p\in\mathcal{W}_w(p_-(L),p_+(L))$ (see \cite[Theorem 1.12]{MartellPrisuelos}). Hence
    \begin{align}\label{OL}
    w(O^*_l)\leq c_{\gamma,r}w(O_l)\lesssim \frac{1}{2^{lp}}\|\Scal_{m,\hh}f\|_{L^p(w)}^p\lesssim \frac{1}{2^{lp}}\|f\|_{L^p(w)}^p<\infty, \quad \forall\, l\in \Z,
   \end{align}
and $E_l^*$ cannot be empty. Therefore, for each $l$, we can take 
  a Whitney decomposition $\{Q_l^j\}_{j\in \N}$,
  of $O_l^*$:
$$
O_l^*=\bigcup_{j\in \N}Q_l^j,\qquad \qquad\textrm{diam}(Q_l^j)\leq d(Q_l^j,\mathbb{R}^n\setminus O_l^*)\leq 4\textrm{diam}(Q_l^j),
$$
and the cubes $Q_l^j$ have disjoint interiors.
Finally, define, for each $j\in \N$ and $l\in \Z$, the sets
\begin{align}\label{T}
T_l^j:=(Q_l^j\times (0,\infty))\bigcap \left(\widehat{O_l^*}\setminus \widehat{O_{l+1}^{*}}\right),
\end{align}
where $\widehat{O_l^*}:=\left\{(x,t)\in \R_+^{n+1}: d(x,\mathbb{R}^n\setminus O_l^*)\geq t \right\}$ and we note that 
$\R^n\setminus \widehat{O_l^*}=\mathcal{R}(E_l^*)$, (see \eqref{ralphacones}).

Let us show that 
\begin{align}\label{suppT_tf(y)}
\supp T_sf(x)\!:=\!\supp \,(s^2L)^me^{-s^2L}f(x)\subset \!
\left(\bigcup_{l\in \Z}\! \left(\widehat{O^*_l}\setminus\widehat{O^*_{l+1}}\right)\right)\!\bigcup \mathbb{F}_1\!\bigcup \mathbb{F}_2=
\left(\bigcup_{l\in \Z,j\in \N} \!\!T_l^j\right)\bigcup \mathbb{F}_1\!\bigcup \mathbb{F}_2,
\end{align}
where $\mathbb{F}_1:=\bigcap_{l\in \Z}\widehat{O^*_l}$ and $\mathbb{F}_2\subset \R^{n+1}_+\setminus \bigcup_{l\in \Z}\widehat{O^*_l}$ with $\mu(\mathbb{F}_1):=\iint_{\R^{n+1}_+}\chi_{\mathbb{F}_1}(y,s)\frac{dy\,ds}{s}=0=\mu(\mathbb{F}_2)$. The fact that $\mu(\mathbb{F}_1)=0$ follows easily. Indeed, note first
that, by \eqref{OL}, and 
since $O^*_{l+1}\subset O^*_{l}$, we conclude that
$$
w(\cap_{l\in \Z}O^*_l)=\lim_{l\rightarrow \infty}w(O^*_l)
\lesssim \lim_{l\rightarrow \infty}
\frac{1}{2^{lp}}=0.
$$
Consequently $|\cap_{l\in \Z}O^*_l|=0$, since the Lebesgue measure and the measure given by $w$ are mutually absolutely continuous.  Hence, clearly 
\begin{align*}
\mu(\mathbb{F}_1)=
\int_{0}^{\infty}\int_{\R^n}\chi_{\mathbb{F}_1}(x,s)\,\frac{dy\,ds}{s}
\leq \lim_{N\rightarrow \infty}
\int_{N^{-1}}^N|\cap_{l\in \Z}O^*_l|\,\frac{ds}{s}=0.
\end{align*}
Finally let us find $\mathbb{F}_2$, and hence obtain
\eqref{suppT_tf(y)}.
Note that 
\begin{multline*}
\R^{n+1}_+=\left(\bigcup_{l\in \Z} \left(\widehat{O^*_l}\setminus\widehat{O^*_{l+1}}\right)\right)\bigcup
\left(\R^{n+1}_+\setminus \bigcup_{l\in \Z} \left(\widehat{O^*_l}\setminus\widehat{O^*_{l+1}}\right)\right)
\\
=
\left(\bigcup_{l\in \Z} \left(\widehat{O^*_l}\setminus\widehat{O^*_{l+1}}\right)\right)
\bigcup\mathbb{F}_1\bigcup
\left(\R^{n+1}_+\setminus \bigcup_{l\in \Z} \widehat{O^*_l}\right).
\end{multline*}
Then, it suffices to show that 
\begin{align}\label{cero}
T_tf(y)=0, \quad \mu-\textrm{a.e.}\, (y,t)\in\R^{n+1}_+\setminus\bigcup_{l\in \Z} \widehat{O}^*_l.
\end{align}
Consider $\mathbb{F}$ the set of Lebesgue points of $|T_sf(x)|^2$ as a function of the variables $(x,s)\in \R^{n+1}_+$ for the measure $dxds$ which is mutually absolutely continuous with respect to $\mu$. Note that $\|\Scal_{m,\hh}f\|_{L^p(w)}<\infty$
implies that
$|T_sf(x)|^2\in L^1_{loc}(\R^{n+1}_+,dx\,ds)$, and hence
$\mu(\R^{n+1}_+\setminus \mathbb{F})=0$. 
 To conclude \eqref{cero}, we observe that $\R^{n+1}_+\setminus\left(\bigcup_{l\in \Z} \widehat{O^*_l}\right)=\bigcap_{l\in \Z}\mathcal{R}(E^*_l)$ (recall the definition of $\mathcal{R}(E^*_l)$ in \eqref{ralphacones}), and then we just need to prove that
\begin{align}\label{cero-2}
T_tf(y)=0,\quad \forall (y,t)\in \bigcap_{l\in \Z}\mathcal{R}(E^*_l)\cap  \mathbb{F}.
\end{align}
On the one hand, if $(y,t)\in \bigcap_{l\in \Z}\mathcal{R}(E^*_l)$, for  every $l\in \Z$ there exists $x_l$ such that $(y,t)\in \Gamma(x_l)$ and $\Scal_{m,\hh}f(x_l)\leq 2^l$. On the other hand,  $(y,t)\in \mathbb{F}$ implies
\begin{align}\label{convergence-L}
\lim_{r\rightarrow 0}\frac{1}{|B((y,t),r)|}\iint_{B((y,t),r)}
\left||T_tf(y)|^2-|T_sf(x)|^2\right|dx\,ds=0.
\end{align}
Given $r>0$, consider
$$ 
x_l^r:=\begin{cases}
x_l\quad \textrm{if} \quad y=x_l\\
y-\frac{r(y-x_l)}{2|y-x_l|}\quad \textrm{if} \quad y\neq x_l,\,
\end{cases} 
$$
it is easy to see that $B\left((x_l^r,t),\frac{r}{4}\right)\subset \Gamma(x_l)\cap B((y,t),r)$, for all $l\in \Z$ and $0<r<t$.
 Combining all these facts we have that, for $(y,t)\in \bigcap_{l\in \Z}\mathcal{R}(E^*_l)\cap  \mathbb{F}$,
\begin{align*}
|T_tf(y)|^2&=\frac{1}{|B((x_l^r,t),r/4)|}\iint_{B((x_l^r,t),r/4)}\left||T_tf(y)|^2-|T_sf(x)|^2\right|\,dx\,ds
\\&\qquad\quad
+
\frac{1}{|B((x_l^r,t),r/4)|}\iint_{B((x_l^r,t),r/4)}|T_sf(x)|^2\,dx\,ds
\\&
\lesssim
\frac{1}{|B((y,t),r)|}\iint_{B((y,t),r)}\left||T_tf(y)|^2-|T_sf(x)|^2\right|dx\,ds
+
\frac{(t+r)^{n+1}}{r^{n+1}}4^l.
\end{align*}
Then, letting first $l\rightarrow -\infty$ and then  $r\rightarrow 0$,  we conclude \eqref{cero-2} by \eqref{convergence-L}.

Now consider 
the following Calder{\'o}n reproducing formula for $f\in L^p(w)$:
\begin{align}\label{formula-CRF-Lpw}
f(x)=\widetilde{C}\int_{0}^{\infty}\left((t^2L)^me^{-t^2L}\right)^{M+2}f(x)\frac{dt}{t}
=\widetilde{C}\lim_{N\rightarrow \infty}
\int_{N^{-1}}^{N}\left((t^2L)^me^{-t^2L}\right)^{M+2}
f(x)\frac{dt}{t},
\end{align}
with the integral converging in $L^p(w)$.

\begin{remark}\label{remark:calderonreproducing}
A priori, by $L^2(\R^n)$ functional calculus, we have the above equalities for functions in $L^2(\R^n)$. Here we explain how to extend them to functions in $L^p(w)$ for all $p\in \mathcal{W}_w(p_-(L),p_+(L))$. Fixing such a $p$, we first introduce the operator $\mathcal{T}_{t,L}^M= ((t^2 L)^m e^{-t^2 L})^{M+1}$, $M\ge 0$, whose adjoint (in $L^2(\R^n)$) is $(\mathcal{T}_{t, L}^M)^*= ((t^2 L^*)^m e^{-t^2 L^*})^{M+1}=\mathcal{T}_{t, L^*}^M$, and set $\mathcal{Q}_L^M f(x,t)= \mathcal{T}_{t, L^*}^Mf(x)$ for $(x,t)\in \mathbb{R}^{n+1}_+$ and $f\in L^2(\R^n)$.
Since $p\in \mathcal{W}_w(p_-(L),p_+(L))$ then $p'\in \mathcal{W}_{w^{1-p'}}(p_-(L^*),p_+(L^*))$  by \cite[Lemma 4.4]{AuscherMartell:I} and the fact that $p_{\pm}(L^*)=p_{\mp}(L)'$ in \cite{Auscher}. Thus,  the vertical square function defined by $\mathcal{T}_{t, L^*}^M$ is bounded on $L^{p'}(w^{1-p'})$ (see \cite{AuscherMartell:III}). Writing $\mathbb{H}=L^2\left((0,\infty),\frac{dt}{t}\right)$, we obtain 
\begin{multline}\label{eq:QLM-SF}
\left\|\mathcal{Q}_L^M h\right\|_{L^{p'}_{\mathbb{H}}(w^{1-p'})}
=
\big\|\,\|\mathcal{Q}_L^M h\|_{\mathbb{H}}\big\|_{L^{p'}(w^{1-p'})}
\\
=
\left(\int_{\R^n}\left(\int_0^{\infty}|\mathcal{T}_{t, L^*}^M h(x)|^2\frac{dt}{t}\right)^{\frac{p'}{2}}w(x)^{1-p'}dx\right)^{\frac{1}{p'}}
\lesssim 
\|h\|_{L^{p'}(w^{1-p'})}.
\end{multline}
Therefore, $\mathcal{Q}_L^M:L^{p'}(w^{1-p'})\rightarrow L^{p'}_{\mathbb{H}}(w^{1-p'})$ and hence its adjoint $(\mathcal{Q}_L^M)^*$ is bounded from $L^{p}_{\mathbb{H}}(w)$ to $L^{p}(w)$ (see also \cite{Auscher,AuscherMartell:III}). Moreover, for $h\in L^{2}_{\mathbb{H}}(\R^n)$ and $f\in L^2(\R^n)$, we have that
$$
\langle (\mathcal{Q}_L^M)^* h, f \rangle_{L^2(\R^n)}
=
\langle h, \mathcal{Q}_L^M  f \rangle_{L^2_{\mathbb{H}}(\R^n)}
=
\int_{\mathbb{R}^n}\int_0^\infty h(y,t) \overline{(\mathcal{T}_{t, L}^M)^*f(y)}\frac{dt}{t}dy
=
\int_{\mathbb{R}^n}\int_0^\infty \mathcal{T}_{t, L}^M h(y,t) \frac{dt}{t} \overline{f(y)} dy,
$$
where it is implicitly understood that $\mathcal{T}_{t, L}^M h(y,t)=\mathcal{T}_{t, L}^M \big(h(\cdot,t)\big)(y)$. Consequently, for every $h\in {L^2_{\mathbb{H}}(\R^n)}$,
$$
(\mathcal{Q}_L^M)^*h(x)
=
\int_0^{\infty} \mathcal{T}_{t, L}^M h(x,t)\frac{dt}{t}
=
\int_0^{\infty} ((t^2 L)^m e^{-t^2 L})^{M+1} h(x,t)\frac{dt}{t}.
$$
Note that $\widetilde{C}(\mathcal{Q}_L^M)^*  \mathcal{Q}_{L^*}^0 f=f$ for every $f\in L^2(\R^n)$, where according to the notation introduced above $\mathcal{Q}_{L^*}^0f(x,t)= \mathcal{T}_{t, L}^0f(x)=(t^2 L)^m e^{-t^2 L}f(x)$. On the other hand, for every $f\in  L^p(w)$ and $g\in L^2(\R^n)\cap L^p(w)$ we have that
\begin{multline*}
\big\|f- \widetilde{C}(\mathcal{Q}_L^M)^*  \mathcal{Q}_{L^*}^0 f\big\|_{L^p(w)}
\le
\|f-g\|_{L^p(w)}+ \widetilde{C}\big\|(\mathcal{Q}_L^M)^*  \mathcal{Q}_{L^*}^0 (g- f)\big\|_{L^p(w)}
\\
\lesssim
\|f-g\|_{L^p(w)}+ \|\mathcal{Q}_{L^*}^0(g-f)\|_{L_{\mathbb{H}}^p(w)}
\lesssim
\|f-g\|_{L^p(w)}
\end{multline*}
where we have used the boundedness of $(\mathcal{Q}_L^M)^*$ along with the fact that $\mathcal{Q}_{L^*}^0$ is bounded from $L^p(w)$ to $L_{\mathbb{H}}^p(w)$, the latter follows as in \eqref{eq:QLM-SF} with $L^p(w)$ in place of $L^{p'}(w^{1-p'})$ since $p\in \mathcal{W}_w(p_-(L),p_+(L))$.  Using now that $L^2(\R^n)\cap L^p(w)$ is dense in $L^p(w)$ we easily conclude that $f=\widetilde{C}(\mathcal{Q}_L^M)^*  \mathcal{Q}_{L^*}^0 f$ for every $f\in L^p(w)$. This is the first equality in \eqref{formula-CRF-Lpw}.

To obtain the second equality in \eqref{formula-CRF-Lpw} we write $I_N=[N^{-1}, N]$ and observe that for every $h\in L_{\mathbb{H}}^p(w)$, one has that $\chi_{I_{N}}h\longrightarrow h$ in $L_{\mathbb{H}}^p(w)$ as $N\to\infty$, and therefore $(\mathcal{Q}_L^M)^*(\chi_{I_{N}}h)\longrightarrow (\mathcal{Q}_L^M)^*h$ in $L^p(w)$ as $N\to\infty$. Taking now $f\in L^p(w)$, as mentioned above, $\mathcal{Q}_{L^*}^0 f\in L_{\mathbb{H}}^p(w) $ and it follows that $(\mathcal{Q}_L^M)^*(\chi_{I_{N}}\mathcal{Q}_{L^*}^0 f)\longrightarrow (\mathcal{Q}_L^M)^*(\mathcal{Q}_{L^*}^0 f)$ on $L^p(w)$, which is what we were seeking to prove.
\end{remark}

Once we have justified the Calder{\'o}n reproducing formula \eqref{formula-CRF-Lpw} we use 
\eqref{suppT_tf(y)} to obtain
that
\begin{multline}\label{Calderonformula}
f(x)=\widetilde{C}\int_{0}^{\infty}\left((t^2L)^me^{-t^2L}\right)^{M+1}
\Big(\sum_{j\in \N, l\in \Z}\chi_{T_l^j}(\cdot,t)(t^2L)^me^{-t^2L}f(\cdot)\Big)(x)\frac{dt}{t}
\\
=\widetilde{C}\lim_{N\rightarrow \infty}
\int_{N^{-1}}^{N}\left((t^2L)^me^{-t^2L}\right)^{M+1}
\Big(\sum_{j\in \N, l\in \Z}\chi_{T_l^j}(\cdot,t)(t^2L)^me^{-t^2L}f(\cdot)\Big)(x)\frac{dt}{t},
\end{multline}
in $L^p(w)$.
Now, set
$$
\lambda_{l}^j:=2^lw(Q_l^j)\quad \textrm{ and } \quad \mm_{l}^{j}(x):=\frac{1}{\lambda_{l}^j}
\int_{0}^{\infty}\left((t^2L)^me^{-t^2L}\right)^{M+1}
\left({f}_{l,m}^j(\cdot,t)\right)(x)\frac{dt}{t},
$$
where ${f}_{l,m}^j(x,t):=\chi_{T_l^j}(x,t)(t^2L)^me^{-t^2L}f(x)$.
We will show that
\begin{align}\label{fNp}
\sum_{j\in \N, l\in \Z}\widetilde{C}\lambda_l^j\mm_{l}^{j}\quad
\textrm{is a}\, \p\, \textrm{of} \, f.
\end{align}

We start showing that, there exists a uniform constant $C_0$, such that $C_0^{-1}\mm_{l}^{j}$ is a $\mol$, for all $j\in \N$ and $l\in \Z$. To this end, we estimate, for all $0\leq k\leq M$, $1\leq i$, $j\in \N$, and $l\in \Z$, the $L^p(w)$ norms of the functions $(\ell(Q_l^j)^2L)^{-k}\mm_{l}^{j}\,\chi_{C_i(Q_l^j)}$. Before that, we set 
$$
\mathfrak{R}^{\ell(Q_l^j)}(E_{l+1}^*)
:=
\big\{(y,t)\in \mathcal{R}(E_{l+1}^*):y\in Q_l^j,\,0<t\leq 5\sqrt{n}\ell(Q_l^j)\big\}.
$$
For all $(y,t)\in T_l^j$ we have that
$$
t\leq d(y,\R^n\setminus O^*_l)\leq d(Q_l^j,\R^n\setminus O^*_l)+
\textrm{diam}(Q_l^j)\leq 5\textrm{diam}(Q_l^j),
$$
and thus
\begin{align}\label{TcontainedinR}
T_l^j\subset \mathfrak{R}^{\ell(Q_l^j)}(E_{l+1}^*).
\end{align}
 Then,
 for all $(y,t)\in T_l^j$ and $c=11\sqrt{n}$,
\begin{align}\label{ballincube}
B(y,t)\subset cQ_l^j.
\end{align}
Now, by definition of $T_l^j$, we have that for every $(y,t)\in T_l^j$ there exists $y_0\in E_{l+1}^*$ such that $|E_{l+1}\cap B(y_0,t)|\geq \gamma |B(y_0,t)|$ and $|y_0-y|<t$. Besides,  considering $z:=y-\frac{t(y -y_0)}{2|y-y_0|}$, we have that $B\left(z,\frac{t}{2}\right)\subset B(y_0,t)\cap B(y,t)$. Consequently,
\begin{multline*}
\gamma |B(y_0,t)|\leq |E_{l+1}\cap B(y_0,t)|\leq |E_{l+1}\cap B(y,t)|+ |B(y_0,t)\setminus B(y,t)|
\\
\leq |E_{l+1}\cap B(y,t)|+ \left|B(y_0,t)\setminus B\left(z,\frac{t}{2}\right)\right|
= |E_{l+1}\cap B(y,t)|+ |B(y_0,t)|\left(1-\frac{1}{2^n}\right). \end{multline*}
Then, for $\gamma=1-\frac{1}{2^{n+1}}$, we obtain
\begin{align}\label{star}
t^n\lesssim |E_{l+1}\cap B(y ,t)|.
\end{align}

We are now ready to consider the case $i=1$. For every $t>0$, let $\mathcal{T}_t:=(t^2L)^{mM+m-k}e^{-t^2(M+1)L}$, and for every  $h\in L^{p'}(w^{1-p'})$ write $\mathcal{Q}_Lh(x,t):=\mathcal{T}_t^*h(x)$, with $(x,t)\in \R^{n+1}_+$. As in Remark \ref{remark:calderonreproducing} 
one can show that $\mathcal{Q}_L:L^{p'}(w^{1-p'})\rightarrow L^{p'}_{\mathbb{H}}(w^{1-p'})$, since 
$p'\in \mathcal{W}_{w^{1-p'}}(p_-(L^*),p_+(L^*))$. Hence its adjoint $\mathcal{Q}_L^*$ has a bounded extension from $L^{p}_{\mathbb{H}}(w)$ to $L^{p}(w)$, where
$$
\mathcal{Q}_L^*h(x)
=
\int_0^{\infty} \mathcal{T}_t h(x,t)\frac{dt}{t}
=
\int_0^{\infty}(t^2L)^{mM+m-k}e^{-t^2(M+1)L}h(x,t)\frac{dt}{t}.
$$
Here, as before, $\mathcal{T}_t h(x,t)=\mathcal{T}_t\big(h(\cdot,t)\big)(x)$, for $(x,t)\in\R_+^{n+1}$.
Write $\widetilde{g}(x,t):=t^{2k}{f}_{l,m}^j(x,t)$ and
$$
\mathcal{I}:=\big\{h\in L^{p'}(w^{1-p'}): \|h\|_{L^{p'}(w^{1-p'})}=1,\,\supp h\subset 4Q_l^j\big\},
$$
From \eqref{ballincube}, \eqref{star}, and \eqref{TcontainedinR} we obtain
\begin{align}\label{estimatemoleculei=1}
&\left\|((\ell(Q_l^j)^2L)^{-k}\mm_{l}^{j})\chi_{4Q_l^j}\right\|_{L^p(w)}
=
\frac{\ell(Q_l^j)^{-2k}}{\lambda^{j}_l} 
\left\|
\int_{0}^{\infty} (t^2L)^{mM+m-k}e^{-t^2(M+1)L}\widetilde{g}(\cdot,t)
\frac{dt}{t}\right\|_{L^p(4Q_{l}^j,w)}
\\
\nonumber&\qquad\qquad=
\frac{\ell(Q_l^j)^{-2k}}{\lambda^{j}_l} 
\sup_{h\in \mathcal{I}}\left|\int_{\R^n}\mathcal{Q}_L^*\widetilde{g}(y)\cdot h(y)\, dy
\right|
 \\
\nonumber&\qquad\qquad
=
\frac{\ell(Q_l^j)^{-2k}}{\lambda^{j}_l} 
\sup_{h\in \mathcal{I}}\left|\int_{\R^n}\int_0^{\infty}\widetilde{g}(y,t)\cdot \mathcal{T}_t^*h(y)\, \frac{dt\,dy}{t}
\right|
\\ \nonumber
 &\qquad\qquad
\lesssim 
\frac{\ell(Q_l^j)^{-2k}}{\lambda^{j}_l} 
\sup_{h\in \mathcal{I}}\iint_{T_l^j}t^{2k}\left|
(t^2L)^me^{-t^2L}f(y)\cdot \mathcal{T}_t^*h(y)\right|
\int_{B(y,t)\cap E_{l+1}}dx\frac{dt}{t^{n+1}}dy
\\ \nonumber
&\qquad\qquad
\lesssim 
\frac{1}{\lambda^{j}_l} 
\sup_{h\in \mathcal{I}}\int_{cQ_l^j\cap
E_{l+1}}\iint_{\Gamma(x)} \left|
(t^2L)^me^{-t^2L}f(y)\cdot \mathcal{T}_t^*h(y)\right|
\frac{dy \, dt}{t^{n+1}}dx
\\ \nonumber
&\qquad\qquad
\leq 
\frac{1}{\lambda^{j}_l} 
\left\| \Scal_{m,\hh} f\right\|_{L^{p}(cQ_l^j\cap
E_{l+1},w)} \sup_{h\in \mathcal{I}}\left\|\left|\!\left|\!\left|
\mathcal{T}_t^*h\right|\!\right|\!\right|_{\Gamma(\cdot)}\right\|_{L^{p'}(w^{1-p'})}
\\\nonumber
&
\qquad\qquad
\le
\frac{1}{\lambda^{j}_l} w(Q_l^j)^{\frac{1}{p}}2^l \sup_{h\in \mathcal{I}}\left\|\left|\!\left|\!\left|
\mathcal{T}_t^*h\right|\!\right|\!\right|_{\Gamma(\cdot)}\right\|_{L^{p'}(w^{1-p'})}
\\\nonumber
&
\qquad\qquad
=
w(Q_l^j)^{\frac{1}{p}-1}\sup_{h\in \mathcal{I}}\left\|\left|\!\left|\!\left|
\mathcal{T}_t^*h\right|\!\right|\!\right|_{\Gamma(\cdot)}\right\|_{L^{p'}(w^{1-p'})}
,
\end{align}
where in the last inequality we have used that $\Scal_{m,\hh}f(x)\leq 2^{l+1}$ for every  $x\in  E_{l+1}$. To estimate the term with the sup we fix $h\in \mathcal{I}$ and note that changing variable $t$ into $\frac{t}{\sqrt{M+1}}$ and using \cite[Proposition 3.29]{MartellPrisuelos}
\begin{multline*}
\left\|\left|\!\left|\!\left|
\mathcal{T}_t^*h\right|\!\right|\!\right|_{\Gamma(\cdot)}\right\|_{L^{p'}(w^{1-p'})}
=
C_M
\bigg\|\Big|\!\Big|\!\Big|
\mathcal{T}_{\frac{t}{\sqrt{M+1}}}^*h\Big|\!\Big|\!\Big|_{\Gamma^{\frac1{\sqrt{M+1}}}(\cdot)}\bigg\|_{L^{p'}(w^{1-p'})}
\lesssim
C_M
\bigg\|\Big|\!\Big|\!\Big|
\mathcal{T}_{\frac{t}{\sqrt{M+1}}}^*h\Big|\!\Big|\!\Big|_{\Gamma(\cdot)}\bigg\|_{L^{p'}(w^{1-p'})}
\\
=
\left\|\left(\iint_{\Gamma(\cdot)}\left|(t^2L^*)^{mM+m-k}e^{-t^2L^*}h(y)\right|^2\frac{dy\,dt}{t^{n+1}}\right)^{\frac{1}{2}}\right\|_{L^{p'}(w^{1-p'})}
\lesssim
\|h\|_{L^{p'}(w^{1-p'})}=1,
\end{multline*}
where last estimate holds since $p'\in \mathcal{W}_{w^{1-p'}}(p_-(L^*),p_+(L^*))$ (see \cite{MartellPrisuelos}). Plugging this into \eqref{estimatemoleculei=1} we conclude that
\begin{align}\label{molecula1}
\left\|\left((\ell(Q_l^j)^2L)^{-k}\mm_l^j\right)\chi_{4Q_l^j}\right\|_{L^p(w)}\lesssim  w(Q_l^j)^{\frac{1}{p}-1}.
\end{align}
Consider now  $i\geq 2$. Note that $w\in RH_{\left(\frac{p_+(L)}{p}\right)'}$ implies that 
$w\in RH_{\left(\frac{q_0}{p}\right)'}$ for some $q_0$ with $\max\{2,p\}<q_0<p_+(L)$. Then,
\begin{align*}
&\left\|\left((\ell(Q_l^j)^2L)^{-k}\mm_{l}^{j}\right)\chi_{C_i(Q_l^j)}\right\|_{L^p(w)}
\\&\qquad
 \leq \frac{1}{\lambda^{j}_l} \left\|\int_{0}^{\infty}\left|(\ell(Q_l^j)^2L)^{-k}\left((t^2L)^me^{-t^2L}\right)^{M+1}
\left({f}_{l,m}^j(\cdot,t)\right)\right|
\frac{dt}{t}\right\|_{L^p(C_i(Q_l^j),w)}
\\&\qquad
\lesssim \frac{\ell(Q_l^j)^{-2k}}{\lambda^{j}_l}w(2^{i+1}Q_l^j)^{\frac{1}{p}} |2^{i+1}Q_l^j|^{-\frac{1}{q_0}}
\left\|\int_{0}^{\infty}\left|t^{2k}(t^2L)^{mM+m-k}e^{-t^2(M+1)L}
\left({f}_{l,m}^j(\cdot,t)\right)\right|
\frac{dt}{t}\right\|_{L^{q_0}(C_i(Q_l^j))}.
 \end{align*}
Applying Minkowski's inequality, the fact that $\{(t^2L)^me^{-t^2L}\}_{t>0}\in \mathcal{F}_{\infty}(L^2\rightarrow L^{q_0})$, and \eqref{TcontainedinR} we obtain the following estimate for the last integral above:
\begin{align*}
&\left\|\int_{0}^{\infty}\left|t^{2k}(t^2L)^{mM+m-k}e^{-t^2(M+1)L}
\left({f}_{l,m}^j(\cdot,t)\right)\right|
\frac{dt}{t}\right\|_{L^{q_0}(C_i(Q_l^j))}
\\
&\quad \leq\int_{0}^{\infty}t^{2k}\left\|
(t^2L)^{mM+m-k}e^{-t^2(M+1)L}
\left({f}_{l,m}^j(\cdot,t)\right)\right\|_{L^{q_0}(C_i(Q_l^j))}
\frac{dt}{t}
\\
&\quad \lesssim \int_{0}^{5\sqrt{n}\ell(Q_l^j)}\left(\int_{\mathbb{R}^n}
\chi_{Q_l^j}(y)\left| {f}_{l,m}^j(y,t)\right|^{2}dy
\right)^{\frac{1}{2}}t^{2k}t^{-n\left(\frac{1}{2}-\frac{1}{q_0}\right)}e^{-c\frac{4^i\ell(Q_l^j)^2}{t^2}}
\frac{dt}{t}
\\
&\quad \lesssim \ell(Q_l^j)^{2k}\left( \iint_{T_l^j}
\left|
(t^2L)^me^{-t^2L}f(y)\right|^{2}\frac{dy\,dt}{t} \right)^{\frac{1}{2}} \left(\int_{0}^{5\sqrt{n}\ell(Q_l^j)}t^{-2n\left(\frac{1}{2}-\frac{1}{q_0}\right)}
e^{-c\frac{4^i\ell(Q_l^j)^2}{t^2}}
\frac{dt}{t}\right)^{\frac{1}{2}}
\\
&\quad =:II_1\times II_2.
\end{align*}
For $II_1$, we proceed as in the estimate of $I_1$ and obtain after invoking \eqref{star} 
\begin{multline*}
II_1\lesssim
\ell(Q_l^j)^{2k}\left(\iint_{T_l^j}
\left|
(t^2L)^me^{-t^2L}f(y)\right|^{2}\int_{B(y,t)\cap E_{l+1}}dx \, dy \frac{dt}{t^{n+1}}\right)^{\frac{1}{2}}
\\
\lesssim \ell(Q_l^j)^{2k}\left\|\Scal_{m,\hh}f\right\|_{L^2(cQ_l^j\cap
E_{l+1})}
\lesssim \ell(Q_l^j)^{2k}|Q_l^j|^{\frac{1}{2}}2^l.
\end{multline*}
As for $II_2$, changing the variable $t$ into $\frac{2^i\ell(Q_l^j)}{t}$ we get
\begin{align*}
II_2
&\lesssim (2^i\ell(Q_l^j))^{-n\left(\frac{1}{2}-\frac{1}{q_0}\right)}e^{-c4^i}\left(\int_{0}^{\infty}t^{2n\left(\frac{1}{2}-\frac{1}{q_0}\right)}
e^{-ct^2}
\frac{dt}{t}\right)^{\frac{1}{2}}\lesssim (2^i\ell(Q_l^j))^{-n\left(\frac{1}{2}-\frac{1}{q_0}\right)}e^{-c4^i}.
\end{align*}
Hence, for $i\geq 2$, using \eqref{doublingcondition},
\begin{align*}
\left\|\left((\ell(Q_l^j)^2L)^{-k}\mm_{l}^{j}\right)\chi_{C_i(Q_l^j)}\right\|_{L^p(w)}
 \lesssim \frac{1}{\lambda_l^j}e^{-c4^i}2^{-\frac{in}{2}}2^lw(2^{i+1}Q_l^j)^{\frac{1}{p}}
\lesssim
e^{-c4^i}w(2^{i+1}Q_l^j)^{\frac{1}{p}-1}.
\end{align*}
From this and \eqref{molecula1}, we infer that there exists a constant $C_0>0$ such that, for all $j\in \N$ and $l\in \Z$,
$
\|\mm_{l}^{j}\|_{mol,w}\leq C_0.
$
Therefore, for every $j\in \N$ and $l\in \Z$, we have that $C_0^{-1}\mm_{l}^{j}$ are $\mol$s associated with the cubes 
$Q_l^j$.

\medskip

Let us  now prove that $\{\lambda_l^j\}_{j\in \N,l\in \Z}\in \ell^1$. Since for each $l\in \Z$, $\{Q_j^l\}_{j\in \N}$ is a Whitney decomposition of $O_l^*$, by \eqref{OL}, and since $f\in \mathbb{H}_{\Scal_{m,\hh},p}^1(w)$,  we obtain
\begin{multline}\label{in-l1}
\sum_{j\in \N, l\in \Z}|\lambda_l^j|
=\sum_{j\in \N, l\in \Z}2^lw(Q_l^j)=\sum_{l\in \Z}2^lw(O_l^*)
\lesssim\sum_{l\in \Z}2^lw(O_l)
\lesssim \sum_{l\in \Z}\int_{2^{l-1}}^{2^{l}}w(O_l)d\lambda
\\
\leq C\int_{0}^{\infty}w(\{x\in \R^n:\Scal_{m,\hh} f(x)>\lambda\})d\lambda
=C \|\Scal_{m,\hh}f\|_{L^1(w)}<\infty.
\end{multline}

\medskip

Thus to conclude \eqref{fNp}, we finally  show that 
\begin{align}\label{pLP}
f=\sum_{j\in \N, l\in \Z}\widetilde{C}\lambda_l^j\mm_{l}^{j} \quad \textrm{in} \quad L^p(w).
\end{align}
Using the notation in Remark \ref{remark:calderonreproducing}, recalling that $f_{l,m}^j(x,t)=\chi_{T_l^j}(x,t)(t^2L)^me^{-t^2L}f(x)$ where the sets $\{T_l^j\}_{j\in \N, l\in \Z}$ are pairwise disjoint, it follows that 
$$
\Bigg\|\sum_{j\in \N, l\in \Z} f_{l,m}^j\Bigg\|_{L_\mathbb{H}^p(w)}
=
\Bigg\|\sum_{j\in \N, l\in \Z} |f_{l,m}^j|\Bigg\|_{L_\mathbb{H}^p(w)}
\le
\Bigg\|\Big(\int_0^\infty |(t^2L)^me^{-t^2L}f|^2\frac{dt}{t}\Big)^\frac12\Bigg\|_{L^p(w)}
\lesssim
\|f\|_{L^p(w)}.
$$
Hence, by \eqref{Calderonformula}, Remark \ref{remark:calderonreproducing} and the dominated convergence theorem we obtain
\begin{multline}\label{eqn:repre-Lpw-just}
\left\|f-\sum_{j+|l|\leq K}\widetilde{C}\lambda_{l}^j\mm_{l}^{j}\right\|_{L^p(w)}
=
\widetilde{C}\left\| (\mathcal{Q}_L^M)^*\Bigg(\sum_{j+|l|>0}f_{l,m}^j\Bigg)-\sum_{j+|l|\leq K}(\mathcal{Q}_L^M)^* f_{l,m}^j\right\|_{L^p(w)}
\\
=
\widetilde{C}\left\|(\mathcal{Q}_L^M)^*\Bigg(\sum_{j+|l|>K}f_{l,m}^j\Bigg)
\right\|_{L^p(w)}
\lesssim
\left\|\sum_{j+|l|>K}f_{l,m}^j
\right\|_{L_{\mathbb{H}}^p(w)}
\longrightarrow 0, \qquad\mbox{as } K\to \infty.
\end{multline}
This proves \eqref{pLP} and therefore, $\sum_{j+|l|>0}\lambda_l^j\mm_l^j$ is a $\p$ of $f$ such that 
$$
\sum_{j+|l|>0}|\lambda_l^j|\lesssim \|\Scal_{m,\hh}f\|_{L^1(w)}.
$$
Consequently, $f\in \mathbb{H}^1_{L,p,\varepsilon,M}(w)$ and
$
\|f\|_{\mathbb{H}_{L,p,\varepsilon,M}^1(w)}\lesssim \|\Scal_{m,\hh}f\|_{L^1(w)},
$
which completes the proof. \qed

%%%%%%%%%%%%%%%%%%%%%%%%%%%%%%%%%%%%%%%%%%%%%%%%%%%%%%%%%%%%%
%%%%%%%%%%%%%%%%%%%%%%%%%%%%%%%%%%%%%%%%%%%%%%%%%%%%%%%%%%%%%

\subsection{Proof of Proposition \ref{lema:SH-1}, part ($b$)}
Fix $w\in A_{\infty}$, $p,q\in \mathcal{W}_w(p_-(L),p_+(L))$, $\varepsilon>0$, and $m,M\in \N$
such that $M>\frac{n}{2}\left(r_w-\frac{1}{p_-(L)}\right)$.

For $f\in \mathbb{H}^1_{\Scal_{m,\hh},p}(w)$ consider the 
$\p$ of $f$, ($f=\sum_{j+|l|>0}\lambda_l^j\mm_l^j$)
obtained in the proof of Proposition \ref{lema:SH-1}, part ($a$). Then, define for each $N\in \N$
$
f_N:=\sum_{0<j+|l|\leq N}\lambda_l^j\mm_l^j.
$
We have that, for each $N\in \N$, $f_N, f-f_N\in \mathbb{H}^1_{L,p,\varepsilon,M}(w)=\mathbb{H}^1_{\Scal_{m,\hh},p}(w)$. Moreover, since
$\sum_{j+|l|>N+1}\lambda_l^j\mm_l^j$ is a $\p$ of $f-f_N$, we have
\begin{align*}
\|\Scal_{m,\hh}(f-f_N)\|_{L^1(w)}=\|f-f_N\|_{\mathbb{H}^1_{\Scal_{m,\hh},p}(w)}
\lesssim\|f-f_N\|_{\mathbb{H}^1_{L,p,\varepsilon,M}(w)}\leq \sum_{j+|l|>N+1}|\lambda_l^j|\mathop{\longrightarrow}\limits_{N\rightarrow \infty} 0.
\end{align*}
Consequently in order to conclude that $f\in H^1_{\Scal_{m,\hh},q}(w)$, it is enough to show that, for each $N\in \N$, $f_N\in \mathbb{H}^1_{\Scal_{m,\hh},q}(w)$, or equivalently that $f_N\in \mathbb{H}^1_{L,q,\varepsilon,M}(w)$. Let us see the latter,  
for every $N$,  following the same computations done in the proof of part ($a$) to show that the $\mm_l^j$ are $\mol$s, but replacing the $L^p(w)$ norm with the $L^q(w)$ norm, we obtain that, for all $i,j\in \N$, $l\in \Z$, and $0\leq k\leq M$,
$$
\|(\ell(Q_l^j)^2L)^{-k}\mm_l^j\|_{L^q(C_i(Q_l^j),w)}\lesssim e^{-c4^i}w(2^{i+1}Q_l^j)^{\frac{1}{q}-1}.
$$
Hence,  $\mm_l^j$ is a multiple of a $(w,q,\varepsilon,M)-$molecule.
Besides, using \eqref{doublingcondition},
\begin{multline*}
\|f_N\|_{L^q(w)}\lesssim \sum_{i\geq 1}\sum_{0<j+|l|\leq N}|\lambda_l^j|\,\|\mm_l^j\|_{L^q(C_i(Q_l^j),w)}
\lesssim
\sum_{i\geq 1}\sum_{0<j+|l|\leq N}|\lambda_l^j|\, e^{-c4^i}w(2^{i+1}Q_l^j)^{\frac{1}{q}-1}
\\
\lesssim 
\sum_{0<j+|l|\leq N}2^lw(Q_l^j)^{\frac{1}{q}}\lesssim \delta_N^{\frac{1}{q}-1}\sum_{0<j+|l|\leq N}2^lw(Q_l^j)\lesssim \delta_N^{\frac{1}{q}-1} \|\Scal_{m,\hh}f\|_{L^1(w)}<\infty.
\end{multline*}
where $\delta_N:=\min_{0< j+|l|\leq N}w(Q_l^j)$.
Then, for each $N\in \N$, we have that the function $
\sum_{0<j+|l|\leq N}\lambda_l^j\mm_l^j
$ is a $(w,q,\varepsilon,M)-$representation of $f_N$. Hence,
$\{f_N\}_{N\in \N}\subset \mathbb{H}^1_{L,q,\varepsilon,M}(w)=\mathbb{H}^1_{\Scal_{m,\hh},q}(w)$.
\qed

%%%%%%%%%%%%%%%%%%%%%%%%%%%%%%%%%%%%%%%%%%%%%%%%%%%%%%%%%%%%%%%%%
%%%%%%%%%%%%%%%%%%%%%%%%%%%%%%%%%%%%%%%%%%%%%%%%%%%%%%%%%%%%%%%%%%
\subsubsection{Proof of Proposition \ref{lema:SH-1}, part ($c$)}
%%%%%%%%%%%%%%%%%%%%%%%%%%%%%%%%%%%%%%%%%%%%%%%%%%%%%%%%%%%%%%%%%
%%%%%%%%%%%%%%%%%%%%%%%%%%%%%%%%%%%%%%%%%%%%%%%%%%%%%%%%%%%%%%%%%%
For  $f \in \mathbb{H}_{\Gcal_{m,\hh},p}^1(w)$, applying Lemma \ref{lema:comparacion,SH-GH,SK-GK}, part ($a$), and the fact that $\Grm_{m,\hh}f(x)\leq \Gcal_{m,\hh}f(x)$ for every $x\in \R^n$ and for every $m\in \N_0$, we conclude
$$
\|\Scal_{m+1,\hh}f\|_{L^1(w)}\lesssim 
\|\Grm_{m,\hh}f\|_{L^1(w)}\leq \|\Gcal_{m,\hh}f\|_{L^1(w)}.
$$
This and part ($a$) of Proposition \ref{lema:SH-1} imply
$$
\mathbb{H}_{\Gcal_{m,\hh},p}^1(w)\subset 
\mathbb{H}_{\Grm_{m,\hh},p}^1(w)
\subset
\mathbb{H}_{\Scal_{m+1,\hh},p}^1(w)=\mathbb{H}_{L,p,\varepsilon,M}^1(w).
$$
To finish the proof, take $f\in \mathbb{H}_{L,p,\varepsilon,M}^1(w)$. Then, by Proposition \ref{prop:contro-mol-SF}, we have that
$$
\|\Gcal_{m,\hh}f\|_{L^1(w)}\lesssim \|f\|_{\mathbb{H}_{L,p,\varepsilon,M}^1(w)}.
$$ 
Consequently, $\mathbb{H}_{L,p,\varepsilon,M}^1(w)\subset \mathbb{H}_{\Gcal_{m,\hh},p}^1(w).$ \qed

%%%%%%%%%%%%%%%%%%%%%%%%%%%%%%%%%%%%%%%%%%%%%%%%%%%%%%%%%%%%%%%%%%%%%%%%%%%
%%%%%%%%%%%%%%%%%%%%%%%%%%%%%%%%%%%%%%%%%%%%%%%%%%%%%%%%%%%%%%%%%%%%%%%%%%%

\bigskip

%%%%%%%%%%%%%%%%%%%%%%%%%%%%%%%%%%%%%%%%%%%%%%%%%%%%%%%%%%%%%%%%%%%%%%%%%%%
%%%%%%%%%%%%%%%%%%%%%%%%%%%%%%%%%%%%%%%%%%%%%%%%%%%%%%%%%%%%%%%%%%%%%%%%%%%
\section{Characterization of the weighted Hardy spaces defined by  square functions associated with the Poisson semigroup}\label{SKPG}

%%%%%%%%%%%%%%%%%%%%%%%%%%%%%%%%%%%%%%%%%%%%%%%%%%%%%%%%%%%%%%%%%%%%%%%%%%%%%%
%%%%%%%%%%%%%%%%%%%%%%%%%%%%%%%%%%%%%%%%%%%%%%%%%%%%%%%%%%%%%%%%%%%%%%%%%%%%%%

In this section, we prove Theorem \ref{thm:hardychartzPoisson},
which is obtained as a consequence of the following proposition.
\begin{proposition}\label{lemma:SKP}
Given $w\in A_{\infty}$, $p,q\in \mathcal{W}_w(p_-(L),p_+(L))$, $K,\,M \in \N$
such that $M>\frac{n}{2}\left(r_w-\frac{1}{2}\right)$, and
$\varepsilon_0=2M+2K+\frac{n}{2}-nr_w$, there hold
\begin{list}{$(\theenumi)$}{\usecounter{enumi}\leftmargin=1cm \labelwidth=1cm\itemsep=0.2cm\topsep=.2cm \renewcommand{\theenumi}{\alph{enumi}}}
\item $\mathbb{H}^1_{L,p,\varepsilon_0,M}(w)= \mathbb{H}^1_{\Scal_{K,\pp},p}(w),$
with equivalent norms.

\item 
$
H^1_{\Scal_{K,\pp},p}(w)$ and $H_{\Scal_{K,\pp},q}^1(w)
$ are isomorphic.

\item  $
\mathbb{H}_{L,p,\varepsilon_0,M}^1(w)= \mathbb{H}_{\Grm_{K-1,\pp},p}^1(w)= \mathbb{H}_{\Gcal_{K-1,\pp},p}^1(w),
$  with equivalent norms.
\end{list}
\end{proposition}

\subsection{ Proof of Proposition \ref{lemma:SKP}, part $(a)$}
To prove the left-to-the-right inclusion observe that if
$f\in \mathbb{H}_{L,p,\varepsilon_0,M}^1(w)$, in particular $f\in L^p(w)$, and from Proposition \ref{prop:contro-mol-SF}, part ($b$), we have that
    $$
    \|\Scal_{K,\pp}f\|_{L^1(w)}\lesssim \|f\|_{\mathbb{H}^1_{L,p,\varepsilon_0,M}(w)}.
    $$
Therefore, we conclude that
$
\mathbb{H}_{L,p,\varepsilon_0,M}^1(w)\subset\mathbb{H}^1_{\Scal_{K,\pp},p}(w).
$

\medskip

As for proving the converse, take $f\in  \mathbb{H}^1_{\Scal_{K,\pp},p}(w)$ and define the same
sets, $(O_l, O^*_l, T_l^j, etc),$ defined in the proof of Proposition
\ref{lema:SH-1}, pat ($a$), but replacing $\Scal_{m,\hh}$ with $\Scal_{K,\pp}$. Besides, consider the following Calder{{\'o}}n reproducing formula of $f$,
\begin{multline*}
f(x)=C\int_{0}^{\infty}\Bigg(\left((t^2L)^{M+K}e^{-t\sqrt{L}}\right)^{2}f(\cdot)\Bigg)(x)
\frac{dt}{t}
\\
= C\lim_{N\rightarrow \infty}\int_{N^{-1}}^{N}(t^2L)^{2M+K}e^{-t \sqrt{L}}
\left((t^2L)^{K}e^{-t\sqrt{L}}
f(\cdot)\right)(x)\frac{dt}{t}.
\end{multline*}
Following the ideas in Remark \ref{remark:calderonreproducing}, these equalities can be extended from $L^2(\R^n)$ to $L^p(w)$, if we show that the vertical square function associated with $(t^2L^*)^{2M+K}e^{-t \sqrt{L^*}}$ is bounded on $L^{p'}(w^{1-p'})$, but this follows from  
\eqref{controlverticalSK} below with $L^*$ in place of $L$ and \cite{AuscherMartell:III}. After this observation we continue with the proof, again following the same computations as in the proof of Proposition
\ref{lema:SH-1}, part ($a$), considering $T_sf(x):=(t^2L)^{K}e^{-t\sqrt{L}}
f(x)$, we can show that $\supp T_sf(x)\subset \left(\bigcup_{l\in \Z} \widehat{O^*_l}\setminus \widehat{O^*_{l+1}}\right)\bigcup \mathbb{F}$, with $\mu(\mathbb{F})=0$ ($\mu(x,s)=\frac{dx\,ds}{s}$).
Consequently, we have that
\begin{align*}
f(x)= C\lim_{N\rightarrow \infty}\int_{N^{-1}}^{N}(t^2L)^{2M+K}e^{-t \sqrt{L}}
\left(\sum_{j\in \N, l\in \Z}\chi_{T_l^j}(\cdot,t)(t^2L)^{K}e^{-t\sqrt{L}}
f(\cdot)\right)(x)\frac{dt}{t},
\end{align*}
in $L^p(w)$.
Hence, considering
    \begin{align*}
\lambda_{l}^j:=2^lw(Q_l^j), \quad \textrm{and} \quad
\mm_l^j(x):=\frac{1}{\lambda_l^j}
\int_{0}^{\infty}(t^2L)^{2M+K}e^{-t\sqrt{L}}
\left(\chi_{T_l^j}(\cdot,t)(t^2L)^{K}e^{-t\sqrt{L}}
f(\cdot)\right)(x)\frac{dt}{t},
\end{align*}
we show that, for some constant $C>0$, we have the following $(w,p,\varepsilon_0,M)-$representation of $f$:
$$
f=C\sum_{j\in \N, l\in \Z}\lambda_l^j\mm_l^j.
$$
To that end, we have to show the following:
\begin{enumerate}
\item[(a)] $\{\lambda_l^j\}\in \ell^1$,

\item [(b)] there exists a constant
$C_0>0$ such that $C_0^{-1}\mm_l^j$ is a $(w,p,\varepsilon_0,M)-$molecule, for all $j\in \N$ and $l\in \Z$,

\item[(c)] $f=C\sum_{j\in \N, l\in \Z}\lambda_l^j\mm_l^j$ in
$L^p(w)$.
\end{enumerate}
Statement $(a)$ follows from the definition of the cubes $Q_l^j$, and the sets $O_l$ and $O_l^*$,
 and from the fact that $\|\Scal_{K,\pp}f\|_{L^1(w)}<\infty$. Indeed, proceeding as in \eqref{in-l1}, we have
    \begin{align*}
    \sum_{j\in \N, l\in \Z}|\lambda_l^j|=\sum_{j\in \N, l\in \Z}2^lw(Q_l^j)\leq\sum_{l\in \Z}2^lw(O_l^*)\lesssim\sum_{l\in \Z}2^lw(O_l)
    \lesssim \|\Scal_{K,\pp}f\|_{L^1(w)}<\infty.
    \end{align*}
The proofs of $(b)$ and $(c)$ are similar to those of
Proposition \ref{lema:SH-1}, part ($a$), so we shall skip some details. To show $(b)$,
fix $j\in \N$, $l\in \Z$ and $0\leq k\leq M$, $k\in \N$. We need to compute the
following norms, for every $i\geq 1$,
$$
\left\|\left((\ell(Q_l^j)^2L)^{-k}\mm_l^j\right)\chi_{C_i(Q_l^j)}\right\|_{L^p(w)}.
$$
For $i=1$,  let $\mathcal{T}_t:=(t^2L)^{2M+K-k}e^{-t\sqrt{L}}$, for $t>0$, and for every $h\in L^{p'}(w^{1-p'})$ define $\mathcal{Q}_{L}h(x,t):=\mathcal{T}_t^*h(x)$, with $(x,t)\in \R^{n+1}_+$.
Applying the subordination formula 
\begin{align}\label{FR}
e^{-t\sqrt{L}}f(y)=C\int_0^{\infty}\frac{e^{-u}}{\sqrt{u}}e^{-\frac{t^2L}{4u}}f(y) \, du,
\end{align}
we have that, for every $\widetilde{K}\in \N$,
\begin{multline}\label{controlverticalSK}
\Bigg(\int_0^{\infty}\big|(t^2L)^{\widetilde{K}}e^{-t\sqrt{L}}f(x)\big|^2\frac{dt}{t}\Bigg)^{\frac{1}{2}}
\lesssim
\int_{0}^{\infty}e^{-u}u^{\frac{1}{2}}\Bigg(\int_0^{\infty}\big|(t^2L)^{\widetilde{K}}e^{-\frac{t^2}{4u}L}f(x)\big|^2\frac{dt}{t}\Bigg)^{\frac{1}{2}}\frac{du}{u}
\\
\lesssim
\int_{0}^{\infty}e^{-u}u^{\widetilde{K}+\frac{1}{2}}\frac{du}{u}\Bigg(\int_0^{\infty}\big|(t^2L)^{\widetilde{K}}e^{-t^2L}f(x)\big|^2\frac{dt}{t}\Bigg)^{\frac{1}{2}}
\lesssim 
\Bigg(\int_0^{\infty}\big|(t^2L)^{\widetilde{K}}e^{-t^2L}f(x)\big|^2\frac{dt}{t}\Bigg)^{\frac{1}{2}}.
\end{multline}
Therefore, 
\begin{multline*}
\|\mathcal{Q}_Lh\|_{L^{p'}_{\mathbb{H}}(w^{1-p'})}=
\big\| \,\|\mathcal{Q}_Lh\|_{\mathbb{H}}\big\|_{L^{p'}(w^{1-p'})}
=
\Bigg\| \Bigg(\int_0^{\infty}\big|(t^2L^*)^{2M+K-k}e^{-t\sqrt{L^*}}h\big|^2\frac{dt}{t}\Bigg)^{\frac{1}{2}}\Bigg\|_{L^{p'}(w^{1-p'})}
\\
\lesssim
\Bigg\|\Bigg(\int_0^{\infty}\big|(t^2L^*)^{2M+K-k}e^{-t^2L^*}h\big|^2\frac{dt}{t}\Bigg)^{\frac{1}{2}}
\Bigg\|_{L^{p'}(w^{1-p'})}
\lesssim
\|h\|_{L^{p'}(w^{1-p'})},
\end{multline*} 
where we have used that $p'\in (p_-(L^*),p_+(L^*))$ since $p_{\pm}(L^*)=p_{\mp}(L)'$), see \cite{Auscher}, 
\cite[Lemma 4.4]{AuscherMartell:I} and \cite{AuscherMartell:III}. Thus, $\mathcal{Q}_L$ is bounded from $L^{p'}(w^{1-p'})$ to $L^{p'}_{\mathbb{H}}(w^{1-p'})$, and,  as in Remark \ref{remark:calderonreproducing}, we have that its adjoint operator,
$$
\mathcal{Q}_L^*f(x)=\int_{0}^{\infty} (t^2L)^{2M+K-k}e^{-t\sqrt{L}}
f(x,t)\frac{dt}{t},
$$
has a bounded extension from $L^p_{\mathbb{H}}(w)$ to $L^p(w)$. 

After this observations we can treat the case $i=1$. Write $f_{l,K}^j(x,t):=\chi_{T_l^j}(x,t)(t^2L)^Ke^{-t\sqrt{L}}f(x)$, $\widetilde{g}(x,t):=t^{2k}f_{l,K}^j(x,t)$, and consider
$$
\mathcal{I}:=\big\{h\in L^{p'}(w^{1-p'}) : \supp h\subset 4Q_l^j \ \textrm{and} \ \|h\|_{L^{p'}(w^{1-p'})}=1\big\}.
$$
Proceeding as in \eqref{estimatemoleculei=1}, we have 
\begin{multline*}
\left\|\left((\ell(Q_l^j)^2L)^{-k}\mm_l^j\right)\chi_{4Q_l^j}
    \right\|_{L^p(w)} 
 =
\frac{\ell(Q_l^j)^{-2k}}{\lambda^{j}_l}
\sup_{h\in\mathcal{I}}\left|\int_{\R^n}\mathcal{Q}_L^*\widetilde{g}(x)\cdot h(x)dx\right|
\\
    \lesssim
    \frac{1}{\lambda_l^j} 
    \| \Scal_{K,\pp}f(x)\|_{L^p(cQ_l^j\cap E_{l+1},w)}
		\sup_{h\in \mathcal{I}}\left\|\left|\!\left|\!\left|
\mathcal{T}_t^*h\right|\!\right|\!\right|_{\Gamma(\cdot)}\right\|_{L^{p'}(w^{1-p'})}
    \lesssim w(Q_l^j)^{\frac{1}{p}-1}.
    \end{multline*}
The last inequality follows from the fact that $\Scal_{K,\pp}f(x)\leq 2^{l+1}$ for all $x\in E_{l+1}$ and also since the conical square function define by $\mathcal{T}_t^*$ is bounded on $L^{p'}(w^{1-p'})$ as $p'\in \mathcal{W}_{w^{1-p'}}(p_-(L^*),p_+(L^*))$ (see \cite{MartellPrisuelos}).

For $i\geq 2$, take $\max\{2,p\}<q_0<p_+(L)$ close enough to $p_+(L)$ such that $w\in
RH_{\left(\frac{q_0}{p}\right)'}$. Since $\{(t\sqrt{L}\,)^{2\widetilde{K}}e^{-t\sqrt{L}}\}_{t>0}\in \mathcal{F}_{\widetilde{K}+\frac{1}{2}}(L^{2}\rightarrow L^{q_0})$, for every $\widetilde{K}\in \N$, taking $r_w<r<r_w+\frac{1}{n}$ close enough to $r_w$ so that $M>\frac{n}{2}\left(r-\frac{1}{2}\right)$, recalling that $0\leq k\leq M$, and using \eqref{doublingcondition}, we have that
 \begin{align*}
&\left\|\left((\ell(Q_l^j)^2L)^{-k}\mm_l^j\right)\chi_{C_i(Q_l^j)}
    \right\|_{L^p(w)}
 \lesssim \frac{1}{\lambda_l^j}\ell(Q_l^j)^{-2k}w(2^{i+1}Q_l^j)^{\frac{1}{p}}|2^{i+1}Q_l^j|^{-\frac{1}{q_0}}
    \\
    &\hspace*{1.5cm}\times\int_0^{\infty}t^{2k}\left(\int_{C_i(Q_l^j)}\left|(t^2L)^{2M+K-k}e^{-t\sqrt{L}}\left(\chi_{Q_l^j}(\cdot)f_{l,K}^j(\cdot,t)\right)(y)\right|^{q_0}dy\right)^{\frac{1}{q_0}}\frac{dt}{t}
    \\
  &\quad\lesssim 2^{-l}w(Q_l^j)^{-1}w(2^{i+1}Q_l^j)^{\frac{1}{p}}
  |2^{i+1}Q_l^j|^{-\frac{1}{q_0}}
    \\
    &\hspace*{1.5cm}\times\int_0^{5\sqrt{n}\ell(Q_l^j)}t^{-n\left(\frac{1}{2}-\frac{1}{q_0}\right)}\left(1+\frac{c4^i\ell(Q_l^j)^2}{t^2}\right)^{-\left(2M+K-k+\frac{1}{2}+\frac{n}{2}\left(\frac{1}{2}-\frac{1}{q_0}\right)\right)}\left(\int_{Q_l^j}|f_{l,K}^j(y,t)|^{2}dy\right)^{\frac{1}{2}}\frac{dt}{t}
    \\
   &\quad\lesssim 2^{-l}2^{irn}
    w(2^{i+1}Q_l^j)^{\frac{1}{p}-1}
    |2^{i+1}Q_l^j|^{-\frac{1}{q_0}}\left(\int_{cQ_l^j\cap E_{l+1}}
   |\Scal_{K,\pp}f(x)|^2 dx\right)^{\frac{1}{2}}
     \\
    &\hspace*{1.5cm}\times
    \left(\int_{0}^{5\sqrt{n}\ell(Q_l^j)}t^{-2n\left(\frac{1}{2}-
    \frac{1}{q_0}\right)}\left(1+\frac{c4^{i}\ell(Q_l^j)^2}{t^2}
    \right)^{-\left
    (4M+2K-2k+1+n\left(\frac{1}{2}-\frac{1}{q_0}\right)
    \right)}\frac{dt}{t}\right)^{\frac{1}{2}}
    \\
    &\quad \lesssim 2^{-i\left(2M+2K+\frac{n}{2}+1-rn
    \right)}w(2^{i+1}Q_l^j)^{\frac{1}{p}-1}
    \\&\quad
    \leq 2^{-i\varepsilon_0}2^{-i(r_wn-rn+1)}w(2^{i+1}Q_l^j)^{\frac{1}{p}-1}.
    \end{align*}
Therefore, it follows that $\|\mm_l^j\|_{mol,w}\leq C_0$ for some constant $C_0$ uniform  in $j\in \N$ and $l\in \Z$.

Let us finally prove that 
$
f=C\sum_{j\in \N, l\in \Z}\lambda_l^j\mm_l^j\,\, \textrm{in}\,\, L^p(w).
$
We follow the same computations as in the proof of Proposition \ref{lema:SH-1} part ($a$), we first see that by \eqref{controlverticalSK} 
\begin{multline*}
\Bigg\|\sum_{j\in \N, l\in \Z} f_{l,m}^j\Bigg\|_{L_\mathbb{H}^p(w)}
=
\Bigg\|\sum_{j\in \N, l\in \Z} |f_{l,m}^j|\Bigg\|_{L_\mathbb{H}^p(w)}
\le
\Bigg
\|\Big(\int_0^\infty |(t^2L)^Ke^{-t\sqrt{L}}f|^2\frac{dt}{t}\Big)^\frac12\Bigg\|_{L^p(w)}
\\
\lesssim
\Bigg\|\Big(\int_0^\infty |(t^2L)^Ke^{-t^2L}f|^2\frac{dt}{t}\Big)^\frac12\Bigg\|_{L^p(w)}
\lesssim
\|f\|_{L^p(w)}.
\end{multline*}
This allows to obtain 
\eqref{eqn:repre-Lpw-just} where in this case $\mathcal{Q}_L^M g(x)=(t^2L^*)^{2M+K}e^{-t\sqrt{L^*}}g(x)$, $(x,t)\in\mathbb{R}^{n+1}_+$. 
Consequently, $f=C\sum_{j\in \N, l\in \Z}\lambda_l^j\mm_l^j\in \mathbb{H}_{L,p,\varepsilon_0,M}^1(w)$,
and also,
\begin{align*}
\|f\|_{\mathbb{H}_{L,p,\varepsilon_0,M}^1(w)}\lesssim
\sum_{j\in \N, l\in \Z}|\lambda_l^j|=\sum_{j\in \N, l\in \Z}2^lw(Q_l^j)\lesssim \sum_{l\in\Z}2^lw(O_l)\lesssim \|\Scal_{K,\pp}f\|_{L^1(w)}
=\|f\|_{ \mathbb{H}^1_{\Scal_{K,\pp},p}(w)}.
\end{align*}
\qed

\medskip
%%%%%%%%%%%%%%%%%%%%%%%%%%%%%%%%%%%%%%%%%%%%%%%%%%%%%%%%%%%%%%%%%
%%%%%%%%%%%%%%%%%%%%%%%%%%%%%%%%%%%%%%%%%%%%%%%%%%%%%%%%%%%%%%%%%
\subsection{Proof of Proposition \ref{lemma:SKP}, part ($b$).}
%%%%%%%%%%%%%%%%%%%%%%%%%%%%%%%%%%%%%%%%%%%%%%%%%%%%%%%%%%%%%%%%%%
%%%%%%%%%%%%%%%%%%%%%%%%%%%%%%%%%%%%%%%%%%%%%%%%%%%%%%%%%%%%%%%%%%%%
Given $w\in A_{\infty}$ and $p,q\in \mathcal{W}_w(p_-(L),p_+(L))$,  from part ($a$), we have that
$
\mathbb{H}_{L,p,\varepsilon_0,M}^1(w)=\mathbb{H}_{\Scal_{K,\pp},p}^1(w)
$ and $
\mathbb{H}_{L,q,\varepsilon_0,M}^1(w)=\mathbb{H}_{\Scal_{K,\pp},q}^1(w),
$
with equivalent norms.
Hence we have the following isomorphisms
$
H_{L,p,\varepsilon_0,M}^1(w)\approx H_{\Scal_{K,\pp},p}^1(w)$ and
$H_{L,q,\varepsilon_0,M}^1(w)\approx H_{\Scal_{K,\pp},q}^1(w).
$
On the other hand, from Proposition \ref{lema:SH-1}, parts ($a$) and ($b$), we have that 
$$
H_{L,p,\varepsilon_0,M}^1(w)\approx H_{\Scal_{K,\hh},p}^1(w)
\approx H_{\Scal_{K,\hh},q}^1(w)\approx
H_{L,q,\varepsilon_0,M}^1(w).
$$
Therefore, we conclude that the spaces $ H_{\Scal_{K,\pp},p}^1(w)$ and 
$ H_{\Scal_{K,\pp},q}^1(w)$ are isomorphic.
\qed
%%%%%%%%%%%%%%%%%%%%%%%%%%%%%%%%%%%%%%%%%%%%%%%%%%%%%%%%%%%%%%%%%
%%%%%%%%%%%%%%%%%%%%%%%%%%%%%%%%%%%%%%%%%%%%%%%%%%%%%%%%%%%%%%%%%
\subsection{Proof of Proposition \ref{lemma:SKP}, part ($c$).}
%%%%%%%%%%%%%%%%%%%%%%%%%%%%%%%%%%%%%%%%%%%%%%%%%%%%%%%%%%%%%%%%%%
%%%%%%%%%%%%%%%%%%%%%%%%%%%%%%%%%%%%%%%%%%%%%%%%%%%%%%%%%%%%%%%%%%%%
For  $f \in \mathbb{H}_{\Gcal_{K-1,\pp},p}^1(w)$, applying Lemma \ref{lema:comparacion,SH-GH,SK-GK}, part ($b$), and the fact that $\Grm_{K-1,\pp}f(x)\leq \Gcal_{K-1,\pp}f(x)$ for every $x\in \R^n$ and for every $K\in \N$, we conclude that
$$
\|\Scal_{K,\pp}f\|_{L^1(w)}\lesssim 
\|\Grm_{K-1,\pp}f\|_{L^1(w)}\leq \|\Gcal_{K-1,\pp}f\|_{L^1(w)}.
$$
This and Proposition \ref{lemma:SKP}, part ($a$), imply
$$
\mathbb{H}_{\Gcal_{K-1,\pp},p}^1(w)\subset 
\mathbb{H}_{\Grm_{K-1,\pp},p}^1(w)
\subset
\mathbb{H}_{\Scal_{K,\pp},p}^1(w)=\mathbb{H}_{L,p,\varepsilon_0,M}^1(w).
$$
To complete the proof, take $f\in \mathbb{H}_{L,p,\varepsilon_0,M}^1(w)$. In particular we have that $f\in L^p(w)$, and by Proposition \ref{prop:contro-mol-SF}, 
$
\|\Gcal_{K-1,\pp}f\|_{L^1(w)}\lesssim \|f\|_{\mathbb{H}_{L,p,\varepsilon_0,M}^1(w)}.
$
Then, $\mathbb{H}_{L,p,\varepsilon_0,M}^1(w)\subset \mathbb{H}_{\Gcal_{K-1,\pp},p}^1(w)$.
\qed

%%%%%%%%%%%%%%%%%%%%%%%%%%%%%%%%%%%
%%%%%%%%%%%%%%%%%%%%%%%%%%%%%%%%%%%
\section{Non-tangential maximal functions}\label{section:NT}
%%%%%%%%%%%%%%%%%%%%%%%%%%%%%%%%%%%
%%%%%%%%%%%%%%%%%%%%%%%%%%%%%%%%%%%

Before starting with the characterization of the Hardy spaces $H^1_{\Ncal_{\hh}}(w)$ and $H^1_{\Ncal_{\pp}}(w)$. We study the $L^p(w)$ boundedness of $\Ncal_{\hh}$ and $\Ncal_{\pp}$ (see \eqref{nontangential1}--\eqref{nontangential2}). Additionally we need to see how they control  the corresponding square functions. The results are the following:
\begin{proposition}\label{prop:acotacion-N}
Given $w\in A_{\infty}$. There hold
\begin{list}{$(\theenumi)$}{\usecounter{enumi}\leftmargin=1cm \labelwidth=1cm\itemsep=0.2cm\topsep=.2cm \renewcommand{\theenumi}{\alph{enumi}}}

\item $
\Ncal_{\hh}$  is bounded on $L^p(w)
$ for all $p\in \mathcal{W}_w(p_-(L),\infty)$,

\item $\Ncal_{\pp}$  is bounded on $L^p(w)
$ for all $p\in \mathcal{W}_w(p_-(L),p_+(L))$.
\end{list}

\end{proposition}
\begin{proposition}\label{prop:comparacion-N-G}
Given an arbitrary $f\in L^2(\mathbb{R}^n)$, for all $w\in A_{\infty}$ and $0<p<\infty$, there hold
\begin{list}{$(\theenumi)$}{\usecounter{enumi}\leftmargin=1cm \labelwidth=1cm\itemsep=0.2cm\topsep=.2cm \renewcommand{\theenumi}{\alph{enumi}}}
\item $
\|\Gcal_{\pp}f\|_{L^p(w)}\lesssim  \|\Ncal_{\pp}f\|_{L^p(w)}$,

\item
$
\|\Grm_{\hh}f\|_{L^p(w)}\lesssim  \|\Ncal_{\hh}f\|_{L^p(w)}$.
\end{list}
\end{proposition}

\subsection{Proof of Proposition \ref{prop:acotacion-N}, part $(a)$}
Fix $w\in A_{\infty}$ and $p\in \mathcal{W}_w(p_-(L),\infty)$. Take
$p_0\in (p_-(L),2)$ and apply the $L^{p_0}(\R^n)-L^2(\R^n)$ off-diagonal estimates satisfied by the family $\{e^{-t^2L}\}_{t>0}$ to obtain
\begin{multline*}
    \|\Ncal_{\hh}f\|_{L^p(w)}    
    \leq\left(\int_{\R^n}\sup_{t>0}\left(\int_{B(x, 2 t)}
    |e^{-t^2L}f(z)|^2\frac{dz}{t^n}\right)^{\frac{p}{2}}
    w(x)dx\right)^{\frac{1}{p}}
    \\
\lesssim \sum_{j\geq
1}e^{-c4^{j}}\left(\int_{\R^n}\sup_{t>0}\left(\int_{B(x,2^{j+2}
 t)}
    |f(z)|^{p_0}\frac{dz}{  t^n}\right)^{\frac{p}{p_0}}
    w(x)dx\right)^{\frac{1}{p}}
    \\
  \lesssim \sum_{j\geq
1}e^{-c4^{j}}2^{\frac{jn}{p_0}}
\left(\int_{\R^n}\mathcal{M}_{p_0}f(x)^p
    w(x)dx\right)^{\frac{1}{p}}
   \lesssim \|\mathcal{M}_{p_0}f\|_{L^p(w)},
\end{multline*}
where  $\mathcal{M}_{p_0}f:=(\mathcal{M}|f|^{p_0})^{\frac{1}{p_0}}$.

Now, take $p_-(L)<p_0<2$ close enough to $p_-(L)$ so that
$w\in A_{\frac{p}{p_0}}$. Consequently, $\mathcal{M}_{p_0}$ is bounded on
$L^p(w)$, and then, we conclude that
$$
    \|\Ncal_{\hh}f\|_{L^p(w)}\lesssim \|\mathcal{M}_{p_0}f\|_{L^p(w)}
    \lesssim \|f\|_{L^p(w)}.
 $$
\qed
\subsection{Proof of Proposition \ref{prop:acotacion-N}, part $(b)$}
First, notice that we can split $\Ncal_{\pp}$ as follows
\begin{align*}
\Ncal_{\pp}f(x)&\leq \sup_{(y,t)\in \Gamma(x)} \left(\int_{B(y,t)}|(e^{-t\sqrt{L}}-e^{-t^2L})f(z)|^2\frac{dz}{t^n}\right)^{\frac{1}{2}}+
\sup_{(y,t)\in \Gamma(x)} \left(\int_{B(y,t)}|e^{-t^2L}f(z)|^2\frac{dz}{t^n}\right)^{\frac{1}{2}}
\\&
=
\sup_{(y,t)\in \Gamma(x)} \left(\int_{B(y,t)}|(e^{-t\sqrt{L}}-e^{-t^2L})f(z)|^2\frac{dz}{t^n}\right)^{\frac{1}{2}}+\Ncal_{\hh}f(x)
=: \mathfrak{m}_{\pp}f(x)+\mathcal{N}_{\hh}f(x).
\end{align*}
After applying the subordination formula \eqref{FR} and Minkowski's integral inequality, we obtain that
\begin{align*}
\mathfrak{m}_{\pp}f(x)&\lesssim\sup_{t>0} \int_{0}^{\infty}e^{-u}u^{\frac{1}{2}}\left(\int_{B(x,2t)}|(e^{-\frac{t^2}{4u}L}-e^{-t^2L})f(y)|^2\frac{dy}{t^n}\right)^{\frac{1}{2}}\frac{du}{u}
\\&
\lesssim\sup_{t>0} \int_{0}^{\frac{1}{4}}u^{\frac{1}{2}}\left(\int_{B(x,2t)}|(e^{-\frac{t^2}{4u}L}-e^{-t^2L})f(y)|^2\frac{dy}{t^n}\right)^{\frac{1}{2}}\frac{du}{u}
\\&\qquad
+
\sup_{t>0} \int_{\frac{1}{4}}^{\infty}e^{-u}u^{\frac{1}{2}}\left(\int_{B(x,2t)}|(e^{-\frac{t^2}{4u}L}-e^{-t^2L})f(y)|^2\frac{dy}{t^n}\right)^{\frac{1}{2}}\frac{du}{u}=:I+II.
\end{align*}
We first deal with $I$. Take $p_-(L)<p_0<2$, and
apply the $L^{p_0}(\R^n)-L^2(\R^n)$ off-diagonal estimates satisfied by  $\{e^{-t^2L}\}_{t>0}$,
\begin{multline*}
I=\sup_{t>0} \int_{0}^{\frac{1}{4}}u^{\frac{1}{2}}\left(\int_{B(x,2t)}|e^{-\frac{t^2L}{2}}(e^{-\left(\frac{1}{4u}-\frac{1}{2}\right)t^2L}-e^{-\frac{t^2}{2}L})f(y)|^2\frac{dy}{t^n}\right)^{\frac{1}{2}}\frac{du}{u}
\\
\lesssim
\sup_{t>0} \sum_{j\geq 1}e^{-c4^j}\int_{0}^{\frac{1}{4}}u^{\frac{1}{2}}\left(\dashint_{B(x,2^{j+2}t)}|(e^{-\left(\frac{1}{4u}-\frac{1}{2}\right)t^2L}-e^{-\frac{t^2}{2}L})f(y)|^{p_0}dy\right)^{\frac{1}{p_0}}\frac{du}{u}.
\end{multline*}
Now, notice that when $0<u<\frac{1}{4}$, we have
\begin{multline*}
\left|\left(e^{-\left(\frac{1}{4u}-\frac{1}{2}\right)t^2L}-e^{-\frac{t^2}{2}L}\right)f(y)\right|\leq 2\int_{\frac{t}{\sqrt{2}}}^{t\sqrt{\frac{1}{4u}-\frac{1}{2}}}|r^2Le^{-r^2L}f(y)|\frac{dr}{r}
\leq 2\int_{\frac{t}{\sqrt{2}}}^{\frac{t}{2\sqrt{u}}}|r^2Le^{-r^2L}f(y)|\frac{dr}{r}
\\
\lesssim \log(u^{-\frac{1}{2}})^{\frac{1}{2}}
\left(\int_{0}^{\infty}|r^2Le^{-r^2L}f(y)|^2\frac{dr}{r}\right)^{\frac{1}{2}}=:\log(u^{-\frac{1}{2}})^{\frac{1}{2}}\, g_{\hh}f(y).
\end{multline*}
Therefore,
$$
I\lesssim  \sum_{j\geq 1}e^{-c4^j}\int_{0}^{\frac{1}{4}}\log(u^{-\frac{1}{2}})^{\frac{1}{2}}u^{\frac{1}{2}}\frac{du}{u}\sup_{t>0}
\left(\dashint_{B(x,2^{j+2}t)}| g_{\hh}f(y)|^{p_0}dy\right)^{\frac{1}{p_0}}\lesssim \mathcal{M}_{p_0}(g_{\hh}f)(x).
$$
On the other hand, for $\frac{1}{4}\leq u<\infty$,
\begin{align*}
\left|\left(e^{-\frac{t^2}{4u}L}-e^{-t^2L}\right)f(y)\right|\leq 2\int_{\frac{t}{2\sqrt{u}}}^{t}|r^2Le^{-r^2L}f(y)|\frac{dr}{r}
\lesssim 
\log(2\sqrt{u})^{\frac{1}{2}}\left(\int_{\frac{t}{2\sqrt{u}}}^{t}|r^2Le^{-r^2L}f(y)|^2\frac{dr}{r}\right)^{\frac{1}{2}}.
\end{align*}
Hence,
\begin{align*}
II&\lesssim
\sup_{t>0} \int_{\frac{1}{4}}^{\infty}e^{-u}\log(2\sqrt{u})^{\frac{1}{2}}u^{\frac{1}{2}}\left(\int_{B(x,2t)}
\int_{\frac{t}{2\sqrt{u}}}^{t}|r^2Le^{-r^2L}f(y)|^2\frac{dr}{r}
\frac{dy}{t^n}\right)^{\frac{1}{2}}\frac{du}{u}
\\&
\lesssim
\sup_{t>0} \int_{\frac{1}{4}}^{\infty}ue^{-u}
\left(\int_{\frac{t}{2\sqrt{u}}}^{t}\int_{B(x,2t)}|r^2Le^{-r^2L}f(y)|^2
\frac{dy}{t^n}\frac{dr}{r}\right)^{\frac{1}{2}}\frac{du}{u}
\\&
\lesssim
 \int_{\frac{1}{4}}^{\infty}e^{-u}\left(\int_{0}^{\infty}\int_{B(x,4\sqrt{u}r)}
|r^2Le^{-r^2L}f(y)|^2\frac{dy\,dr}{r^{n+1}}
\right)^{\frac{1}{2}}du
=
 \int_{\frac{1}{4}}^{\infty}e^{-u}\, \Scal_{\hh}^{4\sqrt{u}}f(x)du,
\end{align*}
recall the definition of $\Scal_{\hh}^{4\sqrt{u}}$ in \eqref{squarealpha} and \eqref{square-H-1}.
Gathering these estimates gives us, for $p_-(L)<p_0<2$, 
$$
\Ncal_{\pp}f(x)\lesssim \mathcal{M}_{p_0}(g_{\hh}f)(x)+\int_{\frac{1}{4}}^{\infty}e^{-u}\, \Scal_{\hh}^{4\sqrt{u}}F(x)\,du+\Ncal_{\hh}f(x), \quad \forall x\in \R^n.
$$
Let $w\in A_{\infty}$ and $p\in \mathcal{W}_w(p_-(L),p_+(L))$, taking norms on $L^p(w)$ and applying \cite[Proposition 3.29]{MartellPrisuelos}, we obtain, for $r>\max\{p/2,r_w\}$ and  $p_-(L)<p_0<2$,
\begin{multline*}
\|\Ncal_{\pp}f\|_{L^p(w)}\lesssim \|\mathcal{M}_{p_0}(g_{\hh}f)\|_{L^p(w)}+\int_{\frac{1}{4}}^{\infty}u^{\frac{nr}{2p}}e^{-u}\,du\, \|\Scal_{\hh}f\|_{L^p(w)}+\|\Ncal_{\hh}f\|_{L^p(w)}
\\
\lesssim 
\|\mathcal{M}_{p_0}(g_{\hh}f)\|_{L^p(w)}+ \|\Scal_{\hh}f\|_{L^p(w)}+\|\Ncal_{\hh}f\|_{L^p(w)}.
\end{multline*}
Now, taking $p_0$ close enough to $p_-(L)$ so that $w\in A_{\frac{p}{p_0}}$,  we have that the maximal operator $\mathcal{M}_{p_0}$ is bounded on $L^p(w)$. Besides, since $p\in \mathcal{W}_w(p_-(L),p_+(L))\subset \mathcal{W}_w(p_-(L),\infty)$, we have that $g_{\hh}$, $\mathcal{S}_{\hh}$, and $\mathcal{N}_{\hh}$ are bounded operators on $L^p(w)$,  (see \cite[Theorem 7.6, $(a)$]{AuscherMartell:III}, \cite[Theorem 1.12, $(a)$]{MartellPrisuelos}, and Proposition \ref{prop:acotacion-N}, part $(a)$, respectively). Consequently,
we conclude ($b$). \qed

\bigskip

We next establish Lemma \ref{lema:changeofangelN}, whose
proof follows similarly to that of \cite[Lemma 6.2]{HofmannMayboroda}.
Consider, for all $\kappa\geq 1$,
$$
\mathcal{N}^{\kappa}f(x):=\sup_{(y,t)\in \Gamma^{\kappa}(x)}\left(\int_{B(y,\kappa t)}|F(z,t)|^2\frac{dz}{t^n}\right)^{\frac{1}{2}},
$$
and
we simply write $\mathcal{N}$ when $\kappa=1$. 
\begin{lemma}\label{lema:changeofangelN}
Given $w\in A_{r}$, $0<p<\infty$, and $\kappa\geq 1$, 
$$
\|\mathcal{N}^{\kappa}f\|_{L^p(w)}\lesssim \kappa^{n\left(\frac{1}{2}+\frac{r}{p}\right)} \|\mathcal{N}f\|_{L^p(w)}.
$$
\end{lemma}
\begin{proof}
Consider
$
O_{\lambda}:=\{x\in \R^n: \mathcal{N}f(x)>\lambda\},\, E_{\lambda}:=\R^n\backslash O_{\lambda},
$ and, for $\gamma=1-\frac{1}{(4\kappa)^n}$, the set of $\gamma$-density $E_{\lambda}^*:=
\left\{x\in \R^n: \forall\, r>0,\,\frac{|E_{\lambda}\cap B(x,r)|}{|B(x,r)|}\geq\gamma\right\}$. Note that $O_{\lambda}^*:=\R^n\setminus E_{\lambda}^*=\{x\in \R^n:\mathcal{M}(\chi_{O_{\lambda}})(x)>1/(4\kappa)^n\}$.

We claim that
for every $\lambda>0$,
\begin{align}\label{claim:changeofanglesN}
\mathcal{N}^{\kappa}f(x)\leq (3\kappa)^{\frac{n}{2}}\lambda, \quad \textrm{for all }\quad x\in E^{*}_{\lambda}.
\end{align}
Assuming this, let $0<p<\infty$ and $w\in A_{r}$, $1\leq r<\infty$, so that $\mathcal{M}:L^{r}(w)\rightarrow L^{r,\infty}(w)$. Hence, we  have
\begin{multline*}
\|\mathcal{N}^{\kappa}f\|^p_{L^p(w)}=p\int_{0}^{\infty}\lambda^{p-1}w(\{x\in \R^n:\mathcal{N}^{\kappa}f(x)>\lambda\})\,d\lambda
\\
=
p(3\kappa)^{\frac{np}{2}}\int_{0}^{\infty}\lambda^{p-1}w(\{x\in \R^n:\mathcal{N}^{\kappa}f(x)>(3\kappa)^{\frac{n}{2}}\lambda\})\,d\lambda
\leq
p(3\kappa)^{\frac{np}{2}}\int_{0}^{\infty}\lambda^{p-1}w(O_\lambda^*)\,d\lambda
\\
\lesssim
p(3\kappa)^{\frac{np}{2}}(4k)^{nr}\int_{0}^{\infty}\lambda^{p-1}w(O_\lambda)\,d\lambda=
(3\kappa)^{\frac{np}{2}}
(4k)^{nr}\|\mathcal{N}f\|_{L^p(w)}^p,
\end{multline*}
which  finishes the proof.

So it just remains to show \eqref{claim:changeofanglesN}. First, note that if  $x\in E_{\lambda}^*$ then, for every $(y,t)\in \Gamma^{2\kappa}(x)$, $B(y,t)\cap E_{\lambda}\neq \emptyset$. To prove this, suppose by way of contradiction that $B(y,t)\subset O_{\lambda}$. Then, since $B(y,t)\subset B(x,3\kappa t)$,
$$
\mathcal{M}(\chi_{O_{\lambda}})(x)\geq \dashint_{B(x,3\kappa t)}\chi_{{O}_{\lambda}}(x)\,dx\geq\frac{|B(y,t)|}{|B(x,3\kappa t)|}=\frac{1}{(3\kappa)^n}>\frac{1}{(4\kappa)^n},
$$ 
which implies that $x\in O_{\lambda}^*$, a contradiction. Therefore, there exists $y_0\in B(y,t)$ (in particular $(y,t)\in \Gamma(y_0)$) such that $\mathcal{N}f(y_0)\leq \lambda$. Hence, for all  $(y,t)\in \Gamma^{2\kappa}(x)$, with $x\in E_{\lambda}^*$,
\begin{align}\label{menor-lambda}
\left(\int_{B(y,t)}|F(\xi,t)|^2\frac{d\xi}{t^n}\right)^{\frac{1}{2}}
\leq \sup_{(z,s)\in \Gamma(y_0)}
\left(\int_{B(z,s)}|F(\xi,s)|^2\frac{d\xi}{s^n}\right)^{\frac{1}{2}}=\Ncal f(y_0)\leq \lambda.
\end{align}

On the other hand, given $x\in E^*_{\lambda}$ and $(y,t)\in \Gamma^{\kappa}(x)$, we have that 
$B(y,\kappa t)\subset\bigcup_iB(y_i,t)$, where $\{B(y_i,t)\}_i$ is a collection of at most $(3\kappa)^n$ balls such that $y_i\in B(y,\kappa t)$ and then $|y_i-x|<2\kappa t$ (equivalently $(y_i,t)\in \Gamma^{2\kappa}(x)$). Thus,
$$
\int_{B(y,\kappa t)}|F(z,t)|^2\frac{dz}{t^n}\leq \sum_i
\int_{B(y_i,t)}|F(z,t)|^2\frac{dz}{t^n}
\leq (3\kappa)^n \lambda^2,
$$
where we have used \eqref{menor-lambda}, since $x\in E_{\lambda}^*$ and $(y_i,t)\in \Gamma^{2\kappa}(x)$.
Finally taking the supremum over all $(y,t)\in \Gamma^{\kappa}(x)$, we obtain \eqref{claim:changeofanglesN} as desired:
$$
\mathcal{N}^{\kappa}f(x)^2\leq (3\kappa)^n \lambda^2,\quad \forall\,x\in E_{\lambda}^*.
$$
\end{proof}

\subsection{Proof of Proposition \ref{prop:comparacion-N-G}}

We start by proving part ($a$).
Fix $w\in A_{\infty}$, $0<p<\infty$, and $f\in L^2(\mathbb{R}^n)$. For every $N>1$ and $\alpha\geq 1$, we define
\begin{align}\label{N-KN}
K_N:=\{(y,t)\in \R^{n+1}_+: y\in B(0,N), t\in (N^{-1},N)\}
\end{align}
and 
\begin{align*}
\Gcal_{\pp,N}^{\alpha}f(x):=\left(\int_{0}^{\infty}\int_{B(x,\alpha t)}\chi_{K_N}(y,t)|t\nabla_{y,t}e^{-t\sqrt{L}}f(y)|^2\frac{dy \ dt}{t^{n+1}}\right)^{\frac{1}{2}},
\end{align*}
when $\alpha=1$ we just write $\Gcal_{\pp,N}$.
Then,
$
\supp \Gcal_{\pp,N}^{\alpha}f\subset B(0,(\alpha+1)N)$ and, since the vertical square function $\left(\int_0^{\infty}|t\nabla_{y,t}e^{-t\sqrt{L}}f(y)|^2\frac{dt}{t}\right)^{\frac{1}{2}}$ is bounded on $L^2(\R^n)$, we have that $\|\Gcal_{\pp,N}^{\alpha}f\|_{L^p(w)}\leq C N^{\frac{n}{2}}\|f\|_{L^2(\mathbb{R}^n)} w(B(0,(\alpha+1)N))^{\frac{1}{p}}<\infty.
$

Following the ideas used in the proofs of \cite[Theorems 6.1 and 7.1]{HofmannMayboroda}.
For every $\lambda>0$, set
\begin{align*}
O_{\lambda}:=\{x\in \mathbb{R}^n : \Ncal_{\pp}^{\kappa}f(x)>\lambda\} \quad \textrm{and} \quad
E_{\lambda}:=\mathbb{R}^n\setminus O_{\lambda},
\end{align*}
where
$$
\Ncal_{\pp}^{\kappa}f(x)=\sup_{(y,t)\in \Gamma^{\kappa}(x)}\left(\int_{B(y,\kappa t)}|e^{-t\sqrt{L}}f(z)|^2\frac{dz}{t^n}\right)^{\frac{1}{2}},
$$
and $\kappa$ is some positive number that we will determine during the proof.
Besides, consider 
\begin{align*}
E^*_{\lambda}:=\left\{x\in \mathbb{R}^n: \forall r>0, \frac{|E_{\lambda}\cap B(x,r)|}{|B(x,r)|}\geq \frac{1}{2}\right\}, \qquad O^*_{\lambda}:=\mathbb{R}^n \setminus E_{\lambda}^*=\left\{x\in \mathbb{R}^n : \mathcal{M}(\chi_{O_{\lambda}})(x)>\frac{1}{2}\right\}.
\end{align*}
Since $O_{\lambda}$ is open,  $O_{\lambda}\subset O_{\lambda}^*$ and then  $E_{\lambda}^*\subset E_{\lambda}$.
Also, since $w\in A_{\infty}$, for $r>r_w$, we have that  $\mathcal{M}:L^r(w)\rightarrow L^{r,\infty}(w)$. Consequently
$
w(O^*_{\lambda})\leq C_w w(O_{\lambda}).
$
On the other hand, consider the set
$$
\widetilde{O}_{\lambda}:=\{x\in \mathbb{R}^n : \Gcal_{\pp,N}^{\alpha}f(x)> \lambda\}.
$$
Proceeding as in the proof of \cite[Proposition 3.2]{MartellPrisuelos}, part $(a)$, we can show that $\widetilde{O}_{\lambda}$ is open and, since $\|\Gcal_{\pp,N}^{\alpha}f\|_{L^p(w)}<\infty$, then $w(\widetilde{O}_{\lambda})<\infty$ which implies that $\widetilde{O}_{\lambda}\subsetneqq \R^n$. Hence, taking a Whitney decomposition of $\widetilde{O}_{\lambda}$, there exists a family of closed cubes $\{Q_j\}_{j\in \N}$ with disjoint interiors such that
$$
\bigcup_{j\in \N}Q_j=\widetilde{O}_{\lambda} \quad \textrm{and} \quad\textrm{diam}(Q_j)
\leq d(Q_j,\mathbb{R}^n\setminus \widetilde{O}_{\lambda})\leq 4\textrm{diam}(Q_j).
$$
We claim that there exists a positive constant $c_w$, depending on the weight, such that, for every $0<\gamma<1$ and $\alpha=12\sqrt{n}$,
\begin{align}\label{wqj}
w(\{x\in E_{\gamma\lambda}^* : \Gcal_{\pp,N}f(x)>2\lambda, \Ncal_{\pp}^{\kappa}f(x)\leq \gamma \lambda\})\leq C\gamma^{c_w}w(\{x\in \mathbb{R}^n : \Gcal_{\pp,N}^{\alpha}f(x)> \lambda\}).
\end{align}
Assuming this momentarily,
we would get
\begin{multline*}
w(\{x\in \mathbb{R}^n : \Gcal_{\pp,N}f(x)>2\lambda\})
\leq w(O_{\gamma\lambda}^*)
+
w(\{x\in E_{\gamma\lambda}^* : \Gcal_{\pp,N}f(x)>2\lambda, \Ncal_{\pp}^{\kappa}f(x)\leq \gamma\lambda\})
+
w(O_{\gamma\lambda})
\\
\leq C\gamma^{c_w}w(\{x\in \mathbb{R}^n : \Gcal_{\pp,N}^{\alpha}f(x)> \lambda\})
+
Cw(\{x\in \mathbb{R}^n : \Ncal_{\pp}^{\kappa}f(x)>\gamma\lambda\}).
\end{multline*}
Multiplying both sides of the previous inequality by $\lambda^{p-1}$ and integrating in $\lambda>0$, we would have
\begin{align*}
\|\Gcal_{\pp,N}f\|_{L^p(w)}^p\leq C\gamma^{c_w}\|\Gcal_{\pp,N}^{\alpha}f\|_{L^p(w)}^p+ C_{\gamma}\|\Ncal_{\pp}^{\kappa}f\|_{L^p(w)}^p.
\end{align*}
Then, applying \cite[Proposition 3.2]{MartellPrisuelos} and Lemma \ref{lema:changeofangelN} with $\mathcal{N}=\mathcal{N}_{\pp}$ we would obtain
\begin{align*}
\|\Gcal_{\pp,N}f\|_{L^p(w)}^p\leq C_{\alpha}\gamma^{c_w}\|\Gcal_{\pp,N}f\|_{L^p(w)}^p+ C_{\kappa,\gamma}\|\Ncal_{\pp}f\|_{L^p(w)}^p.
\end{align*}
Finally, since $\|\Gcal_{\pp,N}f\|_{L^p(w)}\leq \|\Gcal_{\pp,N}^{\alpha}f\|_{L^p(w)}<\infty$, taking $\gamma$ small enough such that $C_{\alpha}\gamma^{c_w}<\frac{1}{2}$, we would conclude that, for some constant $C>0$ uniform on $N$,
\begin{align*}
\|\Gcal_{\pp,N}f\|_{L^p(w)}\le C  \|\Ncal_{\pp}f\|_{L^p(w)}.
\end{align*}
This and the Monotone Convergence Theorem would readily lead to the desired estimate.
Therefore, to complete the proof we just need to show \eqref{wqj}. Notice that since $\Gcal_{\pp,N}f\leq\Gcal_{\pp,N}^{\alpha}f$, we have
\begin{align*}
&\left\{x\in E_{\gamma\lambda}^*  : \Gcal_{\pp,N}f(x)>2\lambda, \Ncal_{\pp}^{\kappa}f(x)\leq \gamma\lambda\right\}
\subset \bigcup_{j\in \N}
\left\{x\in E^*_{\gamma\lambda}\cap Q_j  : \Gcal_{\pp,N}f(x)>2\lambda, \Ncal_{\pp}^{\kappa}f(x)\leq \gamma\lambda\right\}.
\end{align*}
Consequently, since $w\in A_{\infty}$, to obtain \eqref{wqj} it is enough to show
\begin{align}\label{1qj}
|\{x\in E_{\gamma\lambda}^* \cap Q_j: \Gcal_{\pp,N}f(x)>2\lambda, \Ncal_{\pp}^{\kappa}f(x)\leq \gamma \lambda\}|\leq C\gamma^{2}|Q_j|.
\end{align}
To this end,
consider $u(y,t):=e^{-t\sqrt{L}}f(y)$ and
\begin{multline*}
\mathcal{G}_{\pp,1,j,N}f(x):=\left(\int_{\frac{\ell(Q_j)}{2}}^{\infty}
\int_{B(x,t)}\chi_{K_N}(y,t)\left|t\nabla_{y,t}u(y,t)\right|^2\frac{dy \, dt}{t^{n+1}}\right)^{\frac{1}{2}},
\\
 \textrm{and}\quad \mathcal{G}_{
 \pp,2,j,N}f(x):=\left(\int_0^{\frac{\ell(Q_j)}{2}}\int_{B(x,t)}\chi_{K_N}(y,t)\left|t\nabla_{y,t}u(y,t)\right|^2\frac{dy \, dt}{t^{n+1}}\right)^{\frac{1}{2}}.
\end{multline*}
We have that $\Gcal_{\pp,N}f\leq \Gcal_{\pp,1,j,N}f+\Gcal_{\pp,2,j,N}f$ and  that $\mathcal{G}_{\pp,1,j,N}f(x)\leq \lambda$ for all $x\in Q_j$. Indeed, notice that for each $j$, there exists $x_j\in \R^n\setminus \widetilde{O}_{\lambda}$ such that $d(x_j,Q_j)\leq 4\textrm{diam}Q_j$. Besides, if $(y,t)$ is such that $t\geq \frac{\ell(Q_j)}{2}$, $x\in Q_j$, and $y\in B(x,t)$, then
$$
|x_j-y|\leq |x_j-x|+|x-y|<5\sqrt{n}\ell(Q_j)+t\leq 11\sqrt{n} t.
$$
Hence, for $\alpha=12\sqrt{n}$ and for all $x\in Q_j$, we have
\begin{multline*}
\mathcal{G}_{\pp,1,j,N}f(x)^2=\int_{\frac{\ell(Q_j)}{2}}^{\infty}\int_{B(x,t)}\chi_{K_N}(y,t)\left|t\nabla_{y,t}u(y,t)\right|^2\frac{dy \, dt}{t^{n+1}}
\\
\leq \iint_{\Gamma^{\alpha}(x_j)}\chi_{K_N}(y,t)\left|t\nabla_{y,t}u(y,t)\right|^2\frac{dy \, dt}{t^{n+1}}
=\Gcal_{\pp,N}^{\alpha}f(x_j)^2\leq \lambda^2.
\end{multline*} 
This and Chebychev's inequality imply that
\begin{multline}\label{GP22}
|\{x\in E^*_{\gamma\lambda}\cap Q_j : \Gcal_{\pp,N}f(x)>2\lambda, \Ncal_{\pp}^{\kappa}f(x)\leq \gamma\lambda\}|
\\
\leq |\{x\in E^*_{\gamma\lambda}\cap Q_j : \mathcal{G}_{\pp,2,j,N}f(x)>\lambda\}|\leq \frac{1}{\lambda^2}\int_{E^*_{\gamma\lambda}\cap Q_j}\mathcal{G}_{\pp,2,j,N}f(x)^2 dx
\\
\leq\frac{1}{\lambda^2}\int_{E^*_{\gamma\lambda}\cap Q_j}\int_0^{\frac{\ell(Q_j)}{2}}\int_{B(x,t)}\left|t\nabla_{y,t}u(y,t)\right|^2\frac{dy \, dt}{t^{n+1}}\, dx=:\frac{1}{\lambda^2}\int_{E^*_{\gamma\lambda}\cap Q_j}\Gcal_{\pp,2}f(x)^2\, dx.
\end{multline}
To estimate the last integral above, for  $0<\varepsilon<\frac{\ell(Q_j)}{2}$, consider the function
\begin{align}\label{def:GP2}
\mathcal{G}_{\pp,2,\varepsilon}f(x):=\left(\int_{\varepsilon}^{\frac{\ell(Q_j)}{2}}\int_{B(x, t)}
\left|t\nabla_{y,t}u(y,t)\right|^ 2\frac{dy \, dt}{t^{n+1}}\right)^{\frac{1}{2}}.
\end{align}
Besides, for $\beta >0$, consider the region
$$
\mathcal{R}^{\varepsilon,\ell(Q_j),\beta}(E^*_{\gamma\lambda}\cap Q_j):=\bigcup_{x\in E^*_{\gamma\lambda}\cap Q_j}\left\{(y,t)\in \mathbb{R}^n\times (\beta\varepsilon, \beta \ell(Q_j)) : |y-x|<\frac{t}{\beta}\right\},
$$
and we set
\begin{align}\label{matrix:B}
B(y):= \left( \begin{array}{cc}
A(y) & 0 \\
0 & 1 \end{array} \right),
\end{align}
where $A$ is as in \eqref{matrix:A}. Then, we have that
there exist
$0<\widetilde{\lambda}\le\widetilde{\Lambda}<\infty$ such that
\begin{align}\label{ellipticity:1}
\widetilde{\lambda}\,|\xi|^2
\le
\Re B(x)\,\xi\cdot\bar{\xi}
\qquad
\textrm{and}\qquad
|B(x)\,\xi\cdot \bar{\zeta}|
\le
\widetilde{\Lambda}\,|\xi|\,|\zeta|,
\end{align}
for all $\xi,\zeta\in\mathbb{C}^{n+1}$ and almost every $x\in \R^n$. Moreover, we have that
\begin{align}\label{div:B}
\partial_{t}u(y,t)
=
\div_{y,t}(tB(y)\,\nabla_{y,t}u(y,t)).
\end{align}
Finally notice that
\begin{align*}
\Gcal_{P,2,\varepsilon}f(x)^2
\lesssim 
\int_{\beta\varepsilon}^{\beta\ell(Q_j)}\int_{|x-y|<\frac{t}{\beta}}|t\nabla_{y,t}u(y,t)|^2\frac{dy\,dt}{t^{n+1}},\quad \textrm{for all}\quad \beta\in (2^{-1},1).
\end{align*}
From this we immediately see
\begin{align}\label{GP2-epsilon}
\int_{E^*_{\gamma\lambda}\cap Q_j}\Gcal_{P,2,\varepsilon}f(x)^2dx
\lesssim 
\iint_{\mathcal{R}^{\varepsilon,\ell(Q_j),\beta}(E^*_{\gamma\lambda}\cap Q_j)}&t|\nabla_{y,t}u(y,t)|^2 dy\,dt,\quad \textrm{for all}\quad \beta\in (2^{-1},1).
\end{align}
Applying \eqref{ellipticity:1} and  integration by parts in the last integral above, we have that
\begin{align*}
&\iint_{\mathcal{R}^{\varepsilon,\ell(Q_j),\beta}(E^*_{\gamma\lambda}\cap Q_j)}t|\nabla_{y,t}u(y,t)|^2 dy\,dt
\lesssim \Re
\iint_{\mathcal{R}^{\varepsilon,\ell(Q_j),\beta}(E^*_{\gamma\lambda}\cap Q_j)}tB(y)\nabla_{y,t}u(y,t)\cdot \overline{\nabla_{y,t}u(y,t)} dy\,dt
\\
&\quad
=\frac{1}{2}
 \iint_{\mathcal{R}^{\varepsilon,\ell(Q_j),\beta}(E^*_{\gamma\lambda}\cap Q_j)}\Big[tB(y)\nabla_{y,t}u(y,t)\cdot \overline{\nabla_{y,t}u(y,t)}+ t\overline{B(y)\nabla_{y,t}u(y,t)}\cdot \nabla_{y,t}u(y,t)\Big] dy\,dt
\\
&\quad
=C \iint_{\mathcal{R}^{\varepsilon,\ell(Q_j),\beta}(E^*_{\gamma\lambda}\cap Q_j)}\Big[-\div_{y,t}(tB(y)\nabla_{y,t}u(y,t)) \overline{u(y,t)}- \overline{\div_{y,t}(tB(y)\nabla_{y,t}u(y,t))} u(y,t)\Big] dy\,dt
\\
&\quad \qquad
+ \int_{\partial \mathcal{R}^{\varepsilon,\ell(Q_j),\beta}(E^*_{\gamma\lambda}\cap Q_j)}\Big[tB(y)\nabla_{y,t}u(y,t)\cdot \upsilon_{y,t}(y,t)\overline{u(y,t)}+ t\overline{B(y)\nabla_{y,t}u(y,t)}\cdot \upsilon_{y,t}(y,t)u(y,t)\Big] d\sigma,
\end{align*}
where $\upsilon_{y,t}$ is the outer unit normal associated with the domain of integration.

Now, using \eqref{div:B} in the first integral, \eqref{ellipticity:1} 
in the second one, and the fact that $|\upsilon_{y,t}(y,t)|=1$, we obtain
\begin{align*}
\iint_{\mathcal{R}^{\varepsilon,\ell(Q_j),\beta}(E^*_{\gamma\lambda}\cap Q_j)}&t|\nabla_{y,t}u(y,t)|^2 dy\,dt
\\&
\lesssim \left|
 \iint_{\mathcal{R}^{\varepsilon,\ell(Q_j),\beta}(E^*_{\gamma\lambda}\cap Q_j)}\Big[-\partial_tu(y,t)\cdot \overline{u(y,t)}- \partial_t\overline{u(y,t)}\cdot u(y,t)\Big] dy\,dt\right|
\\
&\hspace*{3.5cm}+ \int_{\partial \mathcal{R}^{\varepsilon,\ell(Q_j),\beta}(E^*_{\gamma\lambda}\cap Q_j)}t|\nabla_{y,t}u(y,t)||u(y,t)| d\sigma
\\
&=
 \left|-\iint_{\mathcal{R}^{\varepsilon,\ell(Q_j),\beta}(E^*_{\gamma\lambda}\cap Q_j)}\partial_t|u(y,t)|^2 dy\,dt\right|
+ \int_{\partial \mathcal{R}^{\varepsilon,\ell(Q_j),\beta}(E^*_{\gamma\lambda}\cap Q_j)}t|\nabla_{y,t}u(y,t)||u(y,t)| d\sigma.
\end{align*}
Then, applying again integration by parts and Cauchy-Schwartz's inequality, we conclude that
\begin{align}\label{integrationbypart}
\iint_{\mathcal{R}^{\varepsilon,\ell(Q_j),\beta}(E^*_{\gamma\lambda}\cap Q_j)}&t|\nabla_{y,t}u(y,t)|^2 dy\,dt
\\&\nonumber
\leq
 \int_{\partial \mathcal{R}^{\varepsilon,\ell(Q_j),\beta}(E^*_{\gamma\lambda}\cap Q_j)}|u(y,t)|^2d\sigma
+ \int_{\partial \mathcal{R}^{\varepsilon,\ell(Q_j),\beta}(E^*_{\gamma\lambda}\cap Q_j)}t|\nabla_{y,t}u(y,t)||u(y,t)| d\sigma
\\&\nonumber
\lesssim
\int_{\partial \mathcal{R}^{\varepsilon,\ell(Q_j),\beta}(E^*_{\gamma\lambda}\cap Q_j)}|u(y,t)|^2d\sigma
+ \int_{\partial \mathcal{R}^{\varepsilon,\ell(Q_j),\beta}(E^*_{\gamma\lambda}\cap Q_j)}|t\nabla_{y,t}u(y,t)|^2 d\sigma.
\end{align}
Now, observe that
\begin{align*}
\partial \mathcal{R}^{\varepsilon,\ell(Q_j),\beta}(E^*_{\gamma\lambda}\cap Q_j)&=\left\{(y,t)\in \R^{n+1}_+:d(y,Q_j\cap E^{*}_{\gamma\lambda})=\frac{t}{\beta},\, \beta\varepsilon\leq t\leq \beta\ell(Q_j)\right\}
\\&\quad
\bigcup \left\{y\in \R^{n}:d(y,Q_j\cap E^{*}_{\gamma\lambda})<\varepsilon\right\}\times \left\{\beta\varepsilon\right\}
\\&\quad
\bigcup \left\{y\in \R^{n}:d(y,Q_j\cap E^{*}_{\gamma\lambda})<\ell(Q_j)\right\}\times \left\{\beta\ell(Q_j)\right\}
\\&
=:\mathcal{H}(\beta)\cup\mathcal{T}(\varepsilon)\times \left\{\beta\varepsilon\right\}\cup\mathcal{T}(\ell(Q_j))\times \left\{\beta\ell(Q_j)\right\}.
\end{align*}
and for every function $h:\R^{n+1}_+\rightarrow \R$
\begin{align*}
\int_{\partial \mathcal{R}^{\varepsilon,\ell(Q_j),\beta}(E^*_{\gamma\lambda}\cap Q_j)}h\,d\sigma
=
\int_{\mathcal{H}(\beta)}h\,d\sigma
+
\int_{\mathcal{T}(\varepsilon)}h(y,\beta\varepsilon\,)dy
+
\int_{\mathcal{T}(\ell(Q_j))}h(y,\beta \ell(Q_j))\,dy.
\end{align*}
Besides, consider
\begin{align*}
\mathcal{B}^{\varepsilon,\ell(Q_j)}(E^*_{\gamma\lambda}\cap Q_j):=\left\{(y,t)\in \mathbb{R}^n\times (2^{-1}\varepsilon,\ell(Q_j)) : 2^{-1}d(y,E^*_{\gamma\lambda}\cap Q_j)<t<d(y,E^*_{\gamma\lambda}\cap Q_j)\right\}
\end{align*}
and
  $F(y,t):=\frac{d(y,Q_j\cap E^*_{\gamma\lambda})}{t}$. We have that
$$
|JF(y,t)|\leq\frac{1}{|t|}+\frac{d(y,Q_j\cap E^*_{\gamma\lambda})}{t^2}, \quad t\neq 0,\,\textrm{ for a.e.} \,y\in \R^n,
$$ 
where $JF$ denotes the Jacobian of $F$. 
Then, integrating in $\beta\in (1/2,1)$ and applying the coarea formula
\begin{align*}
\int_{\frac{1}{2}}^1
\int_{\mathcal{H}(\beta)}hd\sigma\,d\beta
&\leq 
\int_{\frac{1}{2}}^1
\int_{F^{-1}(1/\beta)}h\chi_{\mathcal{B}^{\varepsilon,\ell(Q_j)}(E^*_{\gamma\lambda}\cap Q_j)}d\sigma\,d\beta
\\&
\leq
\iint_{\R^{n+1}_+}h(y,t)\chi_{\mathcal{B}^{\varepsilon,\ell(Q_j)}(E^*_{\gamma\lambda}\cap Q_j)}(y,t)|JF(y,t)|dy\,dt
\\&
\leq\iint_{\mathcal{B}^{\varepsilon, \ell(Q_j)}(E^*_{\gamma\lambda}\cap Q_j)}h(y,t)|JF(y,t)|dy\,dt
\\&
\leq \iint_{\mathcal{B}^{\varepsilon, \ell(Q_j)}(E^*_{\gamma\lambda}\cap Q_j)}h(y,t)
\frac{1}{t}\left(1+\frac{d(y,Q_j\cap E^*_{\gamma\lambda})}{t}\right)dy\,dt
\\&
\lesssim \iint_{\mathcal{B}^{\varepsilon, \ell(Q_j)}(E^*_{\gamma\lambda}\cap Q_j)}h(y,t)
\frac{dy\,dt}{t}.
\end{align*}
On the other hand, doing the change of variables $\beta\varepsilon=t$, we have 
\begin{align*}
\int_{\frac{1}{2}}^1
\int_{\mathcal{T}(\varepsilon)}h(y,\beta\varepsilon)dy\,d\beta
=
\int_{\frac{\varepsilon}{2}}^{\varepsilon}\frac{1}{\varepsilon}
\int_{\mathcal{T}(\varepsilon)}h(y,t)dy\,dt
\lesssim
 \iint_{\mathcal{B}^{\varepsilon}(E^*_{\gamma\lambda}\cap Q_j)}h(y,t)\frac{dy\,dt}{t},
\end{align*}
where
\begin{align*}
\mathcal{B}^{\varepsilon}(E^*_{\gamma\lambda}\cap Q_j):=\left\{(y,t)\in \mathbb{R}^n\times (2^{-1}\varepsilon,\varepsilon) : d(y,E^*_{\gamma\lambda}\cap Q_j)<2t\right\}.
\end{align*}
Analogously
\begin{align*}
\int_{\frac{1}{2}}^1
\int_{\mathcal{T}(\ell(Q_j))}h(y,\beta \ell(Q_j))dy\,d\beta
\lesssim
 \iint_{\mathcal{B}^{\ell(Q_j)}(E^*_{\gamma\lambda}\cap Q_j)}h(y,t)\frac{dy\,dt}{t},
\end{align*}
where
\begin{align*}
\mathcal{B}^{\ell(Q_j)}(E^*_{\gamma\lambda}\cap Q_j):=\left\{(y,t)\in \mathbb{R}^n\times (2^{-1}\ell(Q_j),\ell(Q_j)) : d(y,E^*_{\gamma\lambda}\cap Q_j)<2t\right\}.
\end{align*}
Therefore, applying the previous estimates with  $h(y,t)=|u(y,t)|^2$,
and $h(y,t)=|t\nabla_{y,t}u(y,t)|^2$, and also \eqref{GP2-epsilon} and \eqref{integrationbypart}, we have
\begin{align}\label{below}
&\int_{E^*_{\gamma\lambda}\cap Q_j}\Gcal_{P,2,\varepsilon}f(x)^2dx
=2\int_{\frac{1}{2}}^1
\int_{E^*_{\gamma\lambda}\cap Q_j}\Gcal_{P,2,\varepsilon}f(x)^2dx\,d\beta
\\\nonumber&\qquad
\lesssim 
\int_{\frac{1}{2}}^1
\int_{\partial \mathcal{R}^{\varepsilon,\ell(Q_j),\beta}(E^*_{\gamma\lambda}\cap Q_j)}|u(y,t)|^2d\sigma\,d\beta
+ \int_{\frac{1}{2}}^1\int_{\partial \mathcal{R}^{\varepsilon,\ell(Q_j),\beta}(E^*_{\gamma\lambda}\cap Q_j)}|t\nabla_{y,t}u(y,t)|^2 d\sigma\,d\beta
\\\nonumber&\qquad
\lesssim 
\iint_{\widetilde{\mathcal{B}}(E^*_{\gamma\lambda}\cap Q_j)}|u(y,t)|^2\frac{dy\,dt}{t}
+ \iint_{\widetilde{\mathcal{B}}(E^*_{\gamma\lambda}\cap Q_j)}|t\nabla_{y,t}u(y,t)|^2 \frac{dy\,dt}{t}
\\\nonumber&\qquad
=:I+II,
\end{align}
where
$$
\widetilde{\mathcal{B}}(E^*_{\gamma\lambda}\cap Q_j):=
\mathcal{B}^{\varepsilon}(E^*_{\gamma\lambda}\cap Q_j)
\cup
\mathcal{B}^{\ell(Q_j)}(E^*_{\gamma\lambda}\cap Q_j)
\cup
\mathcal{B}^{\varepsilon,\ell(Q_j)}(E^*_{\gamma\lambda}\cap Q_j).
$$
Hence,
\begin{multline*}
I\lesssim \iint_{\mathcal{B}^{\varepsilon}(E^*_{\gamma\lambda}\cap Q_j)}|u(y,t)|^2\frac{dy \, dt}{t}
+\iint_{\mathcal{B}^{\ell(Q_j)}(E^*_{\gamma\lambda}\cap Q_j)}|u(y,t)|^2\frac{dy \, dt}{t}
\\
+
\iint_{\mathcal{B}^{\varepsilon,\ell(Q_j)}(E^*_{\gamma\lambda}\cap Q_j)}|u(y,t)|^2\frac{dy \, dt}{t}
=:I_1+I_2+I_3,
\end{multline*}
and analogously $II\lesssim II_1+II_2+II_3.$
We start estimating $I_1$. For every $(y,t)\in \mathcal{B}^{\varepsilon}(E^*_{\gamma\lambda}\cap Q_j)$, there exists $y_0\in E^*_{\gamma\lambda}\cap Q_j$ such that
$y\in B(y_0,2t)$. Besides, since $y_0\in E^*_{\gamma\lambda}\cap Q_j$, from the definition of
$E^*_{\gamma\lambda}$ we have that $|E_{\gamma\lambda}\cap B(y_0,2t)|\geq Ct^n$ and then
$|E_{\gamma\lambda}\cap B(y,4t)|\geq Ct^n$. Thus, we have for $\kappa\geq 4$,
\begin{multline*}
I_1
\lesssim \iint_{\mathcal{B}^{\varepsilon}(E^*_{\gamma\lambda}\cap Q_j)}\int_{E_{\gamma\lambda}\cap B(y,4t)}|u(y,t)|^2dx\frac{dy \, dt}{t^{n+1}}
\lesssim \int_{\frac{\varepsilon}{2}}^{\varepsilon}\int_{8Q_j\cap E_{\gamma\lambda}}\int_{ B(x,4t)}|u(y,t)|^2\frac{dy}{t^{n}}\frac{dx \, dt}{t}
\\
\leq \int_{\frac{\varepsilon}{2}}^{\varepsilon}\int_{8Q_j\cap E_{\gamma\lambda}}\Ncal_{\pp}^{\kappa}f(x)^2\frac{dx \, dt}{t}
\lesssim |Q_j|(\gamma\lambda)^2.
\end{multline*}
The second inequality follows applying Fubini and noticing that 
$(y,t)\in \mathcal{B}^{\varepsilon}(E^*_{\gamma\lambda}\cap Q_j)$ and $x\in E_{\gamma\lambda}\cap B(y,4t)$ imply that  $x\in E_{\gamma\lambda}\cap 8Q_j$, $y\in B(x,4t)$, and $t\in \left(\frac{\varepsilon}{2},\varepsilon\right)$, where we recall that $\varepsilon<\frac{\ell(Q_j)}{2}$. 
Similarly, for $II_1$,
\begin{align*}
II_1\lesssim \iint_{\mathcal{B}^{\varepsilon}(E^*_{\gamma\lambda}\cap Q_j)}\int_{E_{\gamma\lambda}\cap B(y,4t)}|t\nabla_{y,t}u(y,t)|^2dx\frac{dy \, dt}{t^{n+1}}
\lesssim \int_{8Q_j\cap E_{\gamma\lambda}}\int_{\frac{\varepsilon}{2}}^{\varepsilon}\int_{ B(x,4t)}|\nabla_{y,t}u(y,t)|^2\frac{dy \, dt}{t^{n-1}}dx.
\end{align*}
%%%%%%%%%%%%%%%%%%%Caccioppoli
Now, consider the elliptic operator $\widetilde{L}u(y,t):=-\div_{y,t}\left(B(y)\nabla_{y,t}u(y,t)\right)$, (where $B$ is the matrix defined in \eqref{matrix:B}). Besides, for each $x\in 8Q_j\cap E_{\gamma\lambda}$, cover the truncated cone $\Gamma^{\frac{\varepsilon}{2},\varepsilon,4}(x)
:=\{(y,t)\in \R^{n}\times (\varepsilon/2,\varepsilon): |x-y|<4t\}$ by dyadic cubes $R_i\subset \R^{n+1}_+$, of side length $\ell_{\varepsilon}$, $\frac{\varepsilon}{16\sqrt{n}}<\ell_{\varepsilon}\leq \frac{\varepsilon}{8\sqrt{n}}$. Then, the family $\{2R_i\}_{i\in \N}$ has control overlap.
Hence
since $\widetilde{L}u=0$, apply Caccioppoli's inequality and obtain for $\kappa\geq 5$
\begin{multline*}
\int_{\frac{\varepsilon}{2}}^{\varepsilon}\int_{ B(x,4t)}|\nabla_{y,t}u(y,t)|^2\frac{dy \, dt}{t^{n-1}}
\lesssim \frac{1}{\varepsilon^{n-1}}\sum_{i=1}^{M}\iint_{R_i}
|\nabla_{y,t}u(y,t)|^2dy \, dt
\\
\lesssim \frac{1}{\varepsilon^{n+1}}\sum_{i=1}^M\iint_{2R_i}
|u(y,t)|^2dy \, dt
\lesssim \frac{1}{\varepsilon^{n+1}}\int_{\frac{\varepsilon}{4}}^{2\varepsilon}\int_{B(x,5t)}
|u(y,t)|^2 dy \, dt
\\
\lesssim
 \frac{1}{\varepsilon^{n+1}}\int_{\frac{\varepsilon}{4}}^{2\varepsilon}t^n\,dt\,\Ncal_{\pp}^{\kappa}f(x)^2\lesssim
\Ncal_{\pp}^{\kappa}f(x)^2.
\end{multline*}
Consequently,
$
II_1\lesssim \int_{8Q_j\cap E_{\gamma\lambda}}\Ncal_{\pp}^{\kappa}f(x)^2 \, dx\lesssim |Q_j|(\gamma\lambda)^2.
$
Arguing in the same way, we obtain that $
I_2\lesssim |Q_j|(\gamma\lambda)^2$ and $II_2\lesssim  |Q_j|(\gamma\lambda)^2.$

Finally, for $I_3$ and $II_3$, we decompose $\mathbb{R}^n\setminus (E_{\gamma\lambda}^*\cap Q_j)=O^*_{\gamma\lambda}\cup (\mathbb{R}^n\setminus Q_j)$, (which is an open set since the cubes $Q_j$ are closed and $O^*_{\gamma\lambda}$ is open), into a family of Whitney balls $\{B(x_k,r_k)\}_{k=0}^{\infty}$, such that $\bigcup_{k=0}^{\infty}B(x_k,r_k)=O^*_{\gamma\lambda}\cup (\R^n\setminus Q_j)$, and for some constants
$0<c_1<c_2<1$ and $c_3\in \N$, $c_1d(x_k,E^*_{\gamma\lambda}\cap Q_j)\leq r_k\leq c_2
d(x_k,E^*_{\gamma\lambda}\cap Q_j)$, and $\sum_{k=0}^{\infty}\chi_{B(x_k,r_k)}(x)\leq c_3$, for all $x\in \R^n$. Besides, consider the set
$$
\widetilde{K}:=\{k : d(x_k, E_{\gamma\lambda}^*\cap Q_j))\leq 2(1-c_2)^{-1}\ell(Q_j)\}.
$$
We are going to see that
\begin{align}\label{claim:II3}
\mathcal{B}^{\varepsilon,\ell(Q_j)}(E^*_{\gamma\lambda}\cap Q_j)\subset \bigcup_{k\in \widetilde{K}}B(x_k,r_k)\times [r_k(c_2^{-1}-1)/2,r_k(c_1^{-1}+1)].
\end{align}
Indeed, for $(y,t)\in \mathcal{B}^{\varepsilon.\ell(Q_j)}(E^*_{\gamma\lambda}\cap Q_j)$, we have that $\varepsilon/2<t<\ell(Q_j)$, $y\in \mathbb{R}^n\setminus (E^*_{\gamma\lambda}\cap Q_j)$, and
\begin{align}\label{t0-tildek}
2^{-1}d(y,E^*_{\gamma\lambda}\cap Q_j)< t< d(y,E^*_{\gamma\lambda}\cap Q_j).
\end{align}
Then, there exists $k$ such that $y\in B(x_k,r_k)$. We see that $k\in \widetilde{K}$ and $r_k(c_2^{-1}-1)/2\leq t\leq r_k(c_1^{-1}+1)$.
On the one hand, we have
$$
d(y,E^*_{\gamma\lambda}\cap Q_j)\leq |y-x_k|+d(x_k,E^*_{\gamma\lambda}\cap Q_j)\leq r_k+c_1^{-1}r_k=(1+c_1^{-1})r_k,
$$
and, on the other hand,
$$
d(y,E^*_{\gamma\lambda}\cap Q_j)\geq d(x_k,E^*_{\gamma\lambda}\cap Q_j)-|y-x_k|\geq (r_kc_2^{-1}-r_k)=(c_2^{-1}-1)r_k.
$$
Therefore, by \eqref{t0-tildek}, we have that $t\in [r_k(c_2^{-1}-1)/2,r_k(c_1^{-1}+1)]$. 
From this and recalling that $t<\ell(Q_j)$, we have
\begin{multline*}
d(x_k,E^*_{\gamma\lambda}\cap Q_j)\leq |y-x_k|+d(y,E^*_{\gamma\lambda}\cap Q_j)\leq r_k+2\ell(Q_j)
\\
\leq
\frac{2t}{(c_2^{-1}-1)}+2\ell(Q_j)
\leq 2(1-c_2)^{-1}\ell(Q_j),
\end{multline*}
which in turn gives us that $k\in \widetilde{K}$.
Moreover, note that for every $k\in \widetilde{K}$, we have that
\begin{align}\label{Bsubsetkinktilde}
B(x_k,r_k)\subset C(c_2)Q_j,\quad \textrm{with}\quad C(c_2):=4(1-c_2)^{-1}(c_2+1)+1.
\end{align}
 Indeed, note that since $E^*_{\gamma\lambda}\cap Q_j\subset Q_j$ then $d(x_k,Q_j)\leq d(x_k,E^*_{\gamma\lambda}\cap Q_j)$. Hence, for
$x_0\in B(x_k,r_k)$ and $x_{Q_j}$ being the center of $Q_j$, we have, 
$$
|x_0-x_{Q_j}|_{\infty}\leq |x_0-x_k|_{\infty}+|x_k-x_{Q_j}|_{\infty}
\leq
r_k+\left(2(1-c_2)^{-1}+\frac{1}{2}\right)\ell(Q_j)
\leq (4(1-c_2)^{-1}(c_2+1)+1)\frac{\ell(Q_j)}{2}.
$$
Now, since $E_{\gamma \lambda}^*\subset E_{\gamma\lambda}$ then
$$
d(x_k,Q_j\cap E_{\gamma\lambda})\leq d(x_k,E^*_{\gamma\lambda}\cap Q_j)\leq c_1^{-1}r_k\leq \frac{2c_2}{c_1(1-c_2)}t,
$$
which implies that, for $\kappa>\frac{2c_2}{c_1(1-c_2)}$ there exists $\widetilde{x}\in
Q_j\cap E_{\gamma\lambda}$ such that $|\widetilde{x}-x_k|<\kappa t$, then
$$
\int_{B\left(x_k,\frac{2c_2}{1-c_2}t\right)}|u(y,t)|^2\frac{dy}{t^{n}}\leq
\int_{B(x_k,\kappa t)}|u(y,t)|^2\frac{dy}{t^{n}}\leq \Ncal_{\pp}^{\kappa}f(\widetilde{x})\leq (\gamma\lambda)^2.
$$
Therefore, by \eqref{Bsubsetkinktilde},
we have
\begin{multline*}
I_3\leq \sum_{k\in \widetilde{K}}\int_{\frac{r_k(c_2^{-1}-1)}{2}}^{r_k(c_1^{-1}+1)}\int_{B(x_k,r_k)}|u(y,t)|^2\frac{dy \, dt}{t}
\lesssim \sum_{k\in \widetilde{K}}r_k^n\int_{\frac{r_k(c_2^{-1}-1)}{2}}^{r_k(c_1^{-1}+1)}\int_{B(x_k,\frac{2c_2}{1-c_2}t)}|u(y,t)|^2\frac{dy \, dt}{t^{n+1}}
\\
\lesssim (\gamma\lambda)^2\sum_{k\in \widetilde{K}}r_k^n\lesssim(\gamma\lambda)^2\left|\bigcup_{k\in \widetilde{K}}B(x_k,r_k)\right|\lesssim |Q_j|(\gamma\lambda)^2.
\end{multline*}
Similarly,  arguing as in the estimate of $II_1$ (taking $\kappa$ larger if necessary), we conclude that $
II_3\lesssim |Q_j|(\gamma\lambda)^{2}.
$
Gathering the estimates obtained for $I$ and $II$ gives us that
\begin{align*}
\int_{E^*_{\gamma\lambda}\cap Q_j}\mathcal{G}_{\pp,2,\varepsilon}f(x)^2 dx \leq C |Q_j|(\gamma\lambda)^2,
\end{align*}
with $C$ independent of $\varepsilon$. Now, recall the definitions of $\mathcal{G}_{\pp,2}$ and $\mathcal{G}_{\pp,2,\varepsilon}$ in \eqref{GP22} and \eqref{def:GP2} respectively. Then, let $\varepsilon\rightarrow 0$ and obtain 
\begin{align*}
\int_{E^*_{\gamma\lambda}\cap Q_j}\mathcal{G}_{\pp,2}f(x)^2 dx
\leq C |Q_j|(\gamma\lambda)^2.
\end{align*}
This, together with  \eqref{GP22}, yields \eqref{1qj}.

\medskip

In order to complete the proof of Proposition \ref{prop:comparacion-N-G}, we need to establish part $(b)$. The argument follows the lines of \cite[Theorem 6.1]{HofmannMayboroda} and the proof of  part ($a$), so we only sketch the main changes.
Consider, for $\alpha\geq 1$, for each $N>1$, $K_N$ as in \eqref{N-KN} and
 $$
 \Grm_{\hh,N}^{\alpha}f(x):=\left(\iint_{\Gamma^{\alpha}(x)}\chi_{K_N}(y,t)\left|t\nabla_ye^{-t^2L}f(y)\right|^2\frac{dy \, dt}{t^{n+1}}\right)^{\frac{1}{2}},
 $$
 we write $\Grm_{\hh,N}$ when $\alpha=1$.
Notice that $\supp  \Grm_{\hh,N}^{\alpha}f\subset B(0,(\alpha+1)N)$ and much as before
$$
\|\Grm_{\hh,N}^{\alpha}f\|_{L^p(w)}\leq C \|f\|_{L^2(\R^n)}N^{\frac{n}{2}}w(B(0,(\alpha+1)N))^{\frac{1}{p}}<\infty.
$$
Hence,  it is enough to show part $(b)$ with $\Grm_{\hh,N}$ in place of $\Grm_{\hh}$
with constants uniform in $N$. 
We follow the proof of part $(a)$,
 replacing $\Gcal_{\pp,N}^{\alpha}$ and $\Ncal_{\pp}$ with $\Grm_{\hh,N}^{\alpha}$ and $\Ncal_{\hh}$, respectively, ($\Gcal_{\pp,N}$ with $\Grm_{\pp,N}$ when $\alpha=1$). We also need to replace  $u(y,t)$ with
 $v(y,t):=e^{-t^2L}f(y)$ and $t\nabla_{y,t}u(y,y)$ with $t\nabla_yv(y,t)$.
We also use the ellipticity of the matrix $A$ (see \eqref{matrix:Aproperties}) instead of the properties of the block matrix $B$ defined in \eqref{matrix:B}. Then, we have that 
\begin{align*}
\int_{E^*_{\gamma\lambda}\cap Q_j}\Grm_{\hh,2,\varepsilon}f(x)^2 dx
\lesssim  \iint_{\widetilde{\mathcal{B}}(E^*_{\gamma\lambda}\cap Q_j)}|v(y,t)|^2 \, dy \, dt+ \iint_{\widetilde{\mathcal{B}}(E^*_{\gamma\lambda}\cap Q_j)}t|\nabla_{y}v(y,t)|^2 \, dy \, dt
=:\widetilde{I}+\widetilde{II}.
\end{align*}
From here the proof proceeds much as the proof of part $(a)$: 
term $\widetilde{I}$ is estimated as term $I$, and term  $\widetilde{II}$ as term
${II}$ but, in this case, as in the proof of \cite[Theorem 6.1]{HofmannMayboroda}, we need to use the following parabolic Caccioppoli inequality (see \cite[Lemma 2.8]{HofmannMayboroda}):
\begin{lemma}
Suppose $\partial_tf=-Lf$ in $I_{2r}(x_0,t_0),$ where $I_r(x_0,t_0)=B(x_0,r)\times [t_0-cr^2,t_0], t_0>4cr^2$ and $c>0$. Then, there exists $C=C(\lambda,\Lambda,c)>0$ such that
$$
\iint_{I_r(x_0,t_0)}|\nabla_xf(x,t)|^2 dx \, dt\leq \frac{C}{r^2}\iint_{I_{2r}(x_0,t_0)}|f(x,t)|^2 dx \, dt.
$$
\end{lemma}
\qed

\begin{remark}\label{remark:classesoffunctions-N}
Following the explanation of \cite[Remark 4.22]{MartellPrisuelos}, one can see that  Proposition \ref{prop:comparacion-N-G} holds for all functions $f\in L^q(w)$ with $w\in A_\infty$ and $q\in \mathcal{W}_{w}(p_-(L), p_+(L))$. Details are left to the interested reader.
\end{remark}

\subsection{Characterization of the weighted Hardy spaces associated with $\Ncal_{\hh}$
and $\Ncal_{\pp}$}

The proof of Theorem \ref{thm:hardychartzNontangential} requires several steps. The first one consists in obtaining that the $L^1(w)$ norms of the non-tangential maximal functions applied to $\mol$s are uniformly controlled.

\begin{proposition}\label{prop:contro-mol-NT}
Let $w\in A_{\infty}$, let $p\in \mathcal{W}_w(p_-(L),p_+(L))$, $\varepsilon>0$, and $M\in \N$ such that $M>\frac{n}{2}\left(r_w-\frac{1}{p_-(L)}\right)$, and let $\mm$ be a $\mol$. There hold:
\begin{list}{$(\theenumi)$}{\usecounter{enumi}\leftmargin=1cm \labelwidth=1cm\itemsep=0.2cm\topsep=.2cm \renewcommand{\theenumi}{\alph{enumi}}}
\item $
\|\Ncal_{\hh}\mm\|_{L^1(w)}+\|\Ncal_{\pp}\mm\|_{L^1(w)}\leq C.
$

\item For all $f\in \mathbb{H}_{L,p,\varepsilon,M}^1(w)$, 
$
\|\Ncal_{\hh}f\|_{L^1(w)}+\|\Ncal_{\pp}f\|_{L^1(w)}\lesssim \|f\|_{\mathbb{H}_{L,p,\varepsilon,M}^1(w)}.
$

\end{list}

\end{proposition}

\begin{proof}
Assuming part $(a)$, the proof of part $(b)$ is similar to that of Proposition \ref{prop:contro-mol-SF}, part ($b$), but applying Proposition \ref{prop:acotacion-N} instead of \cite[Theorems 1.12 and 1.13]{MartellPrisuelos}.

Let us prove part $(a)$. Fix $w\in A_{\infty}$, $p\in \mathcal{W}_w(p_-(L),p_+(L))$, $\varepsilon>0$, $M\in \N$ such that $M>\frac{n}{2}\left(r_w-\frac{1}{p_-(L)}\right)$. Then, take  $\mm$ a $\mol$, and $Q$ a cube associated with $\mm$. Besides we fix $p_0$, $q$, and $\widehat{r}$ with $p_-(L)<p_0<\min\{2,p\}\leq \max\{2,p\}<q<p_+(L)$ and $\widehat{r}>r_w$ so that $w\in A_{\frac{p}{p_0}}\cap RH_{\left(\frac{q}{p}\right)'}$ and 
$M>\frac{n}{2}\left(\widehat{r}-\frac{1}{p_0}\right)$.

We start by dealing with $\Ncal_{\hh}$.
For every $x\in \R^n$, we have
\begin{multline*}
\Ncal_{\hh}\mm(x)\leq \left(\sup_{(y,t)\in \Gamma(x),\, 0<t\leq \ell(Q)}\int_{B(y,t)}
|e^{-t^2L}\mm(z)|^2\frac{dz}{t^n}\right)^{\frac{1}{2}}
\\
+\left(\sup_{(y,t)\in \Gamma(x),\,t>\ell(Q)}\int_{B(y,t)}
|e^{-t^2L}\mm(z)|^2\frac{dz}{t^n}\right)^{\frac{1}{2}}=:F_1\mm(x)+F_2\mm(x).
\end{multline*}
Besides, recalling the notation introduced in \eqref{100}, we can write $\mm=\sum_{i\geq 1}\mm\chi_{C_i(Q)}=:\sum_{i\geq 1}\mm_i.$
Hence,
    \begin{align}\label{splitting}
   \|F_1\mm\|_{L^1(w)}\lesssim
 \sum_{i\geq 1}\|\chi_{16Q_i}F_1\mm_i\|_{L^1(w)}
  +
 \sum_{i\geq 1}  \sum_{j\geq 4}
    \|\chi_{C_j(Q_i)}F_1\mm_i\|_{L^1(w)}
=:\sum_{i\geq 1}I_i+\sum_{i\geq 1}\sum_{j\geq
4}I_{ji}.
    \end{align}
  To estimate $I_i$, we apply H\"older's inequality,   
 Proposition \ref{prop:acotacion-N}, and \eqref{astast} for $k=0$:
\begin{align}\label{NIi}
I_{i}
\lesssim w(Q_i)^{\frac{1}{p'}}\|\Ncal_{\hh}\mm_i\|_{L^p(w)}\lesssim w(Q_i)^{\frac{1}{p'}}\|\mm_i\|_{L^p(w)}\leq 2^{-i\varepsilon}.
\end{align}
As for $I_{ji}$, note that for every $x\in C_{j}(Q_i)$,  
$0<t\leq\ell(Q)$, and $(y,t)\in \Gamma(x)$, we have that
$
B(y,t)\subset 2^{j+2}Q_i\setminus 2^{j-1}Q_i.
$
Then, applying that $\{e^{-t^2L}\}_{t>0}\in \mathcal{F}_{\infty}(L^{p_0}\rightarrow L^2)$ and Lemma \ref{lemma:m-h}, we get
 \begin{multline*}
    F_1	\mm(x)
    \leq
    \left(\sup_{0<t\leq \ell(Q)}\int_{2^{j+1}Q_i\setminus 2^{j-1}Q_i}
|e^{-t^2L}\mm_i(z)|^2\frac{dz}{t^n}\right)^{\frac{1}{2}}
\\
\leq
    \sup_{0<t\leq\ell(Q)}t^{-\frac{n}{p_0}}
    e^{-c\frac{4^{j+i}\ell(Q)^2}{t^2}}\|\mm_i\|_{L^{p_0}(\mathbb{R}^n)}
  \lesssim
  w(Q_i)^{-1}e^{-c4^{j+i}}.
  \end{multline*}
 Therefore, taking the norm in $L^1(w)$ in the previous expression and using that $w\in A_{\infty}$, we obtain that
$ I_{ji}
\lesssim e^{-c4^{j+i}}.
$
This, \eqref{splitting}, and \eqref{NIi} yield
    $
\|F_1\mm\|_{L^1(w)}
\leq C.
   $
   
We turn now to estimate the norm in $L^1(w)$ of $F_2\mm$. Considering
$
B_Q:=(I-e^{-\ell(Q)^2L})^M$, $A_Q:=I-B_Q
$, and $\widetilde{\mm}:=(\ell(Q)^2L)^{-M}\mm$,
and noticing that we can write $\widetilde{\mm}=\sum_{i\geq 1}\widetilde{\mm}\chi_{C_i(Q)}=:\sum_{i\geq 1}\widetilde{\mm}_i$. Then, for every $x\in \R^n$,
\begin{align*}
\mm(x)=B_Q\mm(x)+A_Q\mm(x)= \sum_{i\geq 1}	\left(B_Q\mm_i(x)+\sum_{k=1}^{M}C_{k,M}(k\ell(Q)^2L)^Me^{-k\ell(Q)^2L}\widetilde{\mm}_i(x)\right).
\end{align*} 
Besides, proceeding as in \eqref{NIi} and applying the fact that, for every $1\leq k\leq M$, the operators  $(k\ell(Q)^2L)^Me^{-k\ell(Q)^2L}$ and $B_Q$ are bounded on $L^p(w)$ (see \cite{AuscherMartell:II}), we have that
\begin{align}\label{F2local}
\sum_{i\geq 1}\left(\left\|\chi_{16Q_i}F_2B_Q\mm_i\right\|_{L^1(w)}+\sum_{k=1}^M\left\|\chi_{16Q_i}F_2(k\ell(Q)^2L)^Me^{-k\ell(Q)^2L}\widetilde{\mm}_i\right\|_{L^1(w)}\right)
\leq C.
\end{align}

Next, consider $\theta_M:=\sqrt{M+1}$ and note that, for every $j\geq 4$, $i\geq 1$, $x\in C_{j}(Q_i)$, $\ell(Q)/\theta_M<t\leq 2^{j-3}\ell(Q_i)/\theta_M$, and $(y,\theta_M t)\in \Gamma(x)$, we have that
$B(y,\theta_Mt)\subset 2^{j+2}Q_i\setminus 2^{j-1}Q_i$.
Therefore, since $\{e^{-t^2L}\}_{t>0}\in\mathcal{F}_{\infty}(L^{p_0}-L^2)$ and by the $L^{p_0}(\R^n)-L^{p_0}(\R^n)$ off-diagonal estimates satisfied by the family $\{e^{-t^2L}B_Q\}_{t>0}$ (see \eqref{AB}), applying \cite[Lemma 2.1]{MartellPrisuelos} (see also \cite[Lemma 2.3]{HofmannMartell}), and Lemma \ref{lemma:m-h}, we have
\begin{align*}
F_2B_Q\mm_i(x)
&\lesssim \|\mm_i\|_{L^{p_0}(\R^n)}
\left(\sup_{\frac{\ell(Q)}{\theta_M}<t\leq \frac{2^{j-3}\ell(Q_i)}{\theta_M}}
\left(\frac{\ell(Q)}{t}\right)^{2M}t^{-\frac{n}{p_0}}e^{-c\frac{4^{j+i}\ell(Q)^2}{t^2}}
+\sup_{t>\frac{2^{j-3}\ell(Q_i)}{\theta_M}}\left(\frac{\ell(Q)}{t}\right)^{2M}t^{-\frac{n}{p_0}}\right)
\\&
\lesssim w(Q_i)^{-1}\
2^{-i\left(2M+\varepsilon\right)}2^{-j\left(2M+\frac{n}{p_0}\right)}
.
\end{align*}
Then, using \eqref{doublingcondition}, we easily obtain that
\begin{align}\label{NhF2-BQ}
\|\chi_{C_j(Q_i)}F_2B_Q\mm_i\|_{L^1(w)}\lesssim 
2^{-i\left(2M+\varepsilon\right)}2^{-j\left(2M+\frac{n}{p_0}-\widehat{r}n\right)},
\end{align}
for all  $j\geq 4$ and $i\geq 1$.

Note now that, for every $j\geq 4$, $i\geq 1$, $x\in C_{j}(Q_i)$, $\ell(Q)/\sqrt{2}<t\leq 2^{j-3}\ell(Q_i)/\sqrt{2}$, and $(y,\sqrt{2}t)\in \Gamma(x)$, we have that
$B(y,\sqrt{2}t)\subset 2^{j+2}Q_i\setminus 2^{j-1}Q_i$.
Then, proceeding as in the estimate of $F_2B_Q\mm_i$, but using this time the  off-diagonal estimates satisfied by the family $\{t^2Le^{-t^2L}\}_{t>0}$ instead of the ones satisfied by $\{e^{-t^2L}B_Q\}_{t>0}$,  we have that, for every $1\leq k\leq M$,
\begin{align*}
&F_2\left((k\ell(Q)^2L)^Me^{-k\ell(Q)^2L}\widetilde{\mm}_i\right)(x)
\\&\,\,
\lesssim
\sup_{(y,\sqrt{2}t)\in \Gamma(x),t>\frac{\ell(Q)}{\sqrt{2}}}
\left(\frac{\ell(Q)^2}{t^2+k\ell(Q)^2}\right)^{M}\left(\int_{B(y,\sqrt{2}t)}\left|e^{-t^2L}((t^2+k\ell(Q)^2)L)^Me^{-\left(t^2+k\ell(Q)^2\right)L}\widetilde{\mm}_i(z)\right|^2\frac{dz}{t^n}
\right)^{\frac{1}{2}}
\\&\,\,
\lesssim \|\widetilde{\mm}_i\|_{L^{p_0}(\R^n)}
\left(\sup_{\frac{\ell(Q)}{\sqrt{2}}<t\leq \frac{2^{j-3}\ell(Q_i)}{\sqrt{2}}}
\left(\frac{\ell(Q)}{t}\right)^{2M}t^{-\frac{n}{p_0}}e^{-c\frac{4^{j+i}\ell(Q)^2}{t^2}}
+\sup_{t>\frac{2^{j-3}\ell(Q_i)}{\sqrt{2}}}\left(\frac{\ell(Q)}{t}\right)^{2M}t^{-\frac{n}{p_0}}\right)
\\&\,\,
\lesssim w(Q_i)^{-1}\
2^{-i\left(2M+\varepsilon\right)}2^{-j\left(2M+\frac{n}{p_0}\right)}
.
\end{align*}
Then, $\|\chi_{C_j(Q_i)}F_2A_Q\mm\|_{L^1(w)}\lesssim 2^{-i\left(2M+\varepsilon\right)}2^{-j\left(2M+\frac{n}{p_0}-\widehat{r}n\right)}$,
for all  $j\geq 4$ and $i\geq 1$. This, \eqref{NhF2-BQ}, and \eqref{F2local}, and splitting the norm of $F_2\mm$ as in \eqref{splitting}, allow us to conclude that 
$
\|F_2\mm\|_{L^1(w)}\leq C.
$

\medskip

We now consider  $\Ncal_{\pp}$. Note that, in the proof of Proposition \ref{prop:acotacion-N}, part ($b$), (and following the notation introduced there with $f=\mm$) we saw that
$\Ncal_{\pp}\mm(x)\lesssim \mathfrak{m}_{\pp}\mm(x)+\Ncal_{\hh}\mm(x).$
Then, since we have already proved that $\|\Ncal_{\hh}\mm\|_{L^1(w)}\leq C$, we just need to consider $\mathfrak{m}_{\pp}\mm$. Applying,  the subordination formula \eqref{FR},  we have that
\begin{align*}
\mathfrak{m}_{\pp}\mm(x)&\lesssim
\sup_{(y,t)\in \Gamma(x)} \int_{0}^{\frac{1}{4}}u^{\frac{1}{2}}\left(\int_{B(y,t)}|(e^{-\frac{t^2}{4u}L}-e^{-t^2L})\mm(z)|^2\frac{dz}{t^n}\right)^{\frac{1}{2}}\frac{du}{u}
\\&
+\qquad
\sup_{(y,t)\in \Gamma(x)} \int_{\frac{1}{4}}^{\infty}e^{-u}u^{\frac{1}{2}}\left(\int_{B(y,t)}|(e^{-\frac{t^2}{4u}L}-e^{-t^2L})\mm(z)|^2\frac{dz}{t^n}\right)^{\frac{1}{2}}\frac{du}{u}=:I+II.
\end{align*}
Note that  $II$ is bounded by the term $II$ (with $f=\mm$) in the proof of Proposition  \ref{prop:acotacion-N}, part $(b)$. Hence, applying \cite[Proposition 3.2]{MartellPrisuelos} and Proposition \ref{prop:contro-mol-SF}, part ($a$),  we get
$$
\|II\|_{L^1(w)}\lesssim \int_{\frac{1}{4}}^{\infty}e^{-u}\|\Scal_{\hh}^{4\sqrt{u}}f\|_{L^1(w)}du
\lesssim \int_{\frac{1}{4}}^{\infty}u^{c}e^{-u}du\|\Scal_{\hh}\mm\|_{L^1(w)}\leq C,
$$
recall the definition of $\Scal_{\hh}^{4\sqrt{u}}$ in \eqref{squarealpha} and \eqref{square-H-1}.  

Next, we estimate $I$. We shall use the notation introduced before for  $\mm_i$, $\widetilde{\mm}_i$, $B_Q$, and $A_Q$, and also in \eqref{100}. Proceeding as in the  estimate of the term $I$ (with $f=\mm$) in the proof of Proposition \ref{prop:acotacion-N}, part ($b$), we have
\begin{align}\label{F2Pois}
I&\lesssim
 \sum_{l\geq 1}e^{-c4^l}\sup_{(y,t)\in \Gamma(x)} 
\left(\dashint_{B(y,2^{l+1}t)}\left(\int_{\frac{t}{\sqrt{2}}}^{\infty}|r^2Le^{-r^2L}\mm(z)|^2\frac{dr}{r}\right)^{\frac{p_0}{2}}dz\right)^{\frac{1}{p_0}}
\\\nonumber&
\lesssim
 \sum_{l\geq 1}e^{-c4^l}\sup_{(y,t)\in \Gamma(x), 0<t\leq \ell(Q)} 
\left(\dashint_{B(y,2^{l+1}t)}\left(\int_{0}^{\infty}|r^2Le^{-r^2L}\mm(z)|^2\frac{dr}{r}\right)^{\frac{p_0}{2}}dz\right)^{\frac{1}{p_0}}
\\\nonumber&\qquad +
 \sum_{l\geq 1}e^{-c4^l}\sup_{(y,t)\in \Gamma(x), t> \ell(Q)} 
\left(\dashint_{B(y,2^{l+1}t)}\left(\int_{\frac{t}{\sqrt{2}}}^{\infty}|r^2Le^{-r^2L}\mm(z)|^2\frac{dr}{r}\right)^{\frac{p_0}{2}}dz\right)^{\frac{1}{p_0}}
\\\nonumber&
=:\sum_{l\geq 1}e^{-c4^l}\left(
F_{1,l}\mm(x)+F_{2,l}\mm(x)\right).
\end{align}
We first estimate $F_{1,l}\mm(x)$. 
Note that  considering the following vertical square functions
$$
g_{\hh,1}\mm(x):=\left(\int_{0}^{\ell(Q)}|r^2Le^{-r^2L}\mm(x)|^2\frac{dr}{r}\right)^{\frac{1}{2}}\qquad \textrm{and}$$
$$
g_{\hh,2}\widetilde{\mm}(x):=\left(\int_{\ell(Q)}^{\infty}\left(\frac{\ell(Q)^2}{r^2}\right)^{2M}|(r^2L)^{M+1}e^{-r^2L}\widetilde{\mm}(x)|^2\frac{dr}{r}\right)^{\frac{1}{2}}.
$$
 We have that
\begin{multline}\label{F1l-p}
F_{1,l}\mm(x)\lesssim\sup_{(y,t)\in \Gamma(x), 0<t\leq \ell(Q)} 
\left(\dashint_{B(y,2^{l+1}t)}|g_{\hh,1}\mm(z)|^{p_0}dz\right)^{\frac{1}{p_0}}
\\
+\sup_{(y,t)\in \Gamma(x), 0<t\leq \ell(Q)} 
\left(\dashint_{B(y,2^{l+1}t)}|g_{\hh,2}\widetilde{\mm}(z)|^{p_0}dz\right)^{\frac{1}{p_0}}=:F_{1,l}^1\mm(x)+F_{1,l}^2\widetilde{\mm}(x).
\end{multline}
Applying H\"older's inequality, \eqref{doublingcondition} and by the boundedness on $L^p(w)$ of the maximal operator $\mathcal{M}_{p_0}$ (recall that $w\in A_{\frac{p}{p_0}}$) and the vertical square function $g_{\hh,1}$ (see \cite{AuscherMartell:III}), and by \eqref{astast}, we have that
\begin{align*}
\left\|\chi_{2^{l+3}Q_i}F_{1,l}^1\mm_i\right\|_{L^1(w)}\lesssim
\|\chi_{2^{l+3}Q_i}\mathcal{M}_{p_0}(g_{\hh,1}\mm_i)\|_{L^1(w)}
\lesssim w(2^{l}Q_i)^{\frac{1}{p'}}\|\mathcal{M}_{p_0}(g_{\hh,1}\mm_i)\|_{L^p(w)}\lesssim 2^{ln\widehat{r}}2^{-i\varepsilon}.
\end{align*}
Now observe that for every $i\geq 1$, $j\geq l+3$, $x\in C_{j}(Q_i)$, $0<t\leq \ell(Q)$, and $(y,t)\in \Gamma(x)$ we have that $B(y,2^{l+1}t)\subset 2^{j+2}Q_i\setminus 2^{j-1}Q_i$. Then, applying H\"older's inequality, Minkowski's integral inequality, the fact that $w\in RH_{\left(\frac{q}{p}\right)'}$ and $\{r^2Le^{-r^2L}\}_{r>0}\in \mathcal{F}(L^{p_0}-L^q)$, and Lemma \ref{lemma:m-h}, we obtain that
\begin{multline*}
\left\|\chi_{C_j(Q_i)}F_{1,l}^1\mm_i\right\|_{L^1(w)}
\lesssim
w(2^{j+1}Q_i)^{\frac{1}{p'}}\|\mathcal{M}_{p_0}(\chi_{2^{j+2}Q_i\setminus 2^{j-1}Q_i}g_{\hh,1}\mm_i)\|_{L^p(w)}
\\
\lesssim
w(2^{j+1}Q_i)^{\frac{1}{p'}}\|\chi_{2^{j+2}Q_i\setminus 2^{j-1}Q_i}g_{\hh,1}\mm_i\|_{L^p(w)}
\lesssim
w(2^{j+1}Q_i)|2^{j+1}Q_i|^{-\frac{1}{q}}\|\chi_{2^{j+2}Q_i\setminus 2^{j-1}Q_i}g_{\hh,1}\mm_i\|_{L^q(\R^n)}
\\
\lesssim
w(2^{j+1}Q_i)|2^{j+1}Q_i|^{-\frac{1}{q}}\|\mm_i\|_{L^{p_0}(\R^n)}
\left(\int_0^{\ell(Q)}e^{-c\frac{4^{j+i}\ell(Q)^2}{r^2}}r^{-\frac{2n}{p_0}+\frac{2n}{q}}\frac{dr}{r}\right)^{\frac{1}{2}}
\lesssim
e^{-c4^{j+i}}.
\end{multline*}
Therefore,
\begin{align}\label{splitting2}
\left\|F_{1,l}^1\mm\right\|_{L^1(w)}\lesssim \sum_{i\geq 1}\|\chi_{2^{l+3}Q_i}F_{1,l}^1\mm_i\|_{L^1(w)}+\sum_{i\geq 1}\sum_{j\geq l+3}\|\chi_{C_j(Q_i)}F_{1,l}^1\mm_i\|_{L^1(w)}\lesssim 2^{ln\widehat{r}}.
\end{align}
Similarly, noticing that $g_{\hh,2}$, (disregarding the factor $(\ell(Q)^2/r^2))^{2M}$ since it is controlled by one), is bounded on $L^p(w)$ (see \cite{AuscherMartell:III}), we get
$$
\left\|\chi_{2^{l+3}Q_i}F_{1,l}^2\widetilde{\mm}_i\right\|_{L^1(w)}\lesssim
 w(2^{l}Q_i)^{\frac{1}{p'}}\|\mathcal{M}_{p_0}(g_{\hh,2}\widetilde{\mm}_i)\|_{L^p(w)}\lesssim 2^{ln\widehat{r}}2^{-i\varepsilon},
$$
and, since $\{(r^2L)^{M+1}e^{-r^2L}\}_{r>0}\in \mathcal{F}(L^{p_0}-L^q)$, proceeding as before,
\begin{align*}
&\left\|\chi_{C_j(Q_i)}F_{1,l}^2\widetilde{\mm}_i\right\|_{L^1(w)}
\\
&\qquad\lesssim
w(2^{j+1}Q_i)|2^{j+1}Q_i|^{-\frac{1}{q}}\|\widetilde{\mm}_i\|_{L^{p_0}(\R^n)}
\left(\int_{\ell(Q)}^{\infty}\left(\frac{\ell(Q)^2}{r^2}\right)^{2M}e^{-c\frac{4^{j+i}\ell(Q)^2}{r^2}}r^{-\frac{2n}{p_0}+\frac{2n}{q}}\frac{dr}{r}\right)^{\frac{1}{2}}
\\&\qquad
\lesssim
2^{-j\left(2M+\frac{n}{p_0}-n\widehat{r}\right)}2^{-i(2M+\varepsilon)}.
\end{align*}
Hence, splitting $\left\|F_{1,l}^2\widetilde{\mm}\right\|_{L^1(w)}$ as in \eqref{splitting2}, and by \eqref{F1l-p}, we obtain that
$
\|F_{1,l}\mm\|_{L^1(w)}\leq 2^{ln\widehat{r}}.
$ 

Let us turn to the estimate of $F_{2,l}\mm$.
Consider the vertical square function
$$
\mathfrak{g}_{\hh,t}\widetilde{\mm}(x):=\left(\int_{\frac{t}{\sqrt{2}}}^{\infty}\left(\frac{\ell(Q)}{r}\right)^{4M}|(r^2L)^{M+1}e^{-r^2L}\widetilde{\mm}(x)|^2\frac{dr}{r}\right)^{\frac{1}{2}},
$$
and note that
\begin{align}\label{F2lP}
\|F_{2,l}\mm\|_{L^1(w)}&\leq \sum_{i\geq 1}\sum_{j\geq 1}\left(\left\|\chi_{C_j(Q_i)}\sup_{(y,t)\in \Gamma(\cdot),\ell(Q)<t\leq 2^{j-l-4}\ell(Q_i)}\left(\dashint_{B(y,2^{l+1}t)}|\mathfrak{g}_{\hh,t}\widetilde{\mm}_i(z)|^{p_0}dz\right)^{\frac{1}{p_0}}\right\|_{L^1(w)}
\right.
\\\nonumber&\hspace*{2.5cm}
\left. +
\left\|\chi_{C_j(Q_i)}\sup_{(y,t)\in \Gamma(\cdot),t>2^{j-l-4}\ell(Q_i)}\left(\dashint_{B(y,2^{l+1}t)}|\mathfrak{g}_{\hh,t}\widetilde{\mm}_i(z)|^{p_0}dz\right)^{\frac{1}{p_0}}\right\|_{L^{1}(w)}\right)
\\\nonumber&
=:\sum_{i\geq 1}\sum_{j\geq 1}\left(\|\chi_{C_j(Q_i)}F_{2,l}^1\widetilde{\mm}_i\|_{L^1(w)}+\|\chi_{C_j(Q_i)}F_{2,l}^2\widetilde{\mm}_i\|_{L^1(w)}\right)
\\&\nonumber
\lesssim
\sum_{i\geq 1}\|\chi_{2^{l+3}Q_i}F_{2,l}^1\widetilde{\mm}_i\|_{L^1(w)}+\sum_{i\geq 1}\sum_{j\geq l+3}\|\chi_{C_j(Q_i)}F_{2,l}^1\widetilde{\mm}_i\|_{L^1(w)}
\\\nonumber&\qquad
+
\sum_{i\geq 1}
\|\chi_{2^{l+3}Q_i}F_{2,l}^2\widetilde{\mm}_i\|_{L^1(w)}
+
\sum_{i\geq 1}\sum_{j\geq l+3}\|\chi_{C_j(Q_i)}F_{2,l}^2\widetilde{\mm}_i\|_{L^1(w)}
.
\end{align}
Next, for every  $
\ell(Q)<t<\sqrt{2}r$ we have that $\mathfrak{g}_{\hh,t}$ is controlled by $g_{\hh}$ (where $g_{\hh}$ is defined in the proof of Proposition \ref{prop:acotacion-N}, part $b$) and $g_{\hh}$ is bounded on $L^p(w)$ (see \cite{AuscherMartell:III}), hence, for $a=1,2$, 
\begin{align*}
\left\|\chi_{2^{l+3}Q_i}F_{2,l}^a\widetilde{\mm}_i\right\|_{L^1(w)}
\lesssim
 w(2^{l}Q_i)^{\frac{1}{p'}}\|\mathcal{M}_{p_0}({g}_{\hh}\widetilde{\mm}_i)\|_{L^p(w)}
\lesssim w(2^{l}Q_i)^{\frac{1}{p'}}\|\widetilde{\mm}_i\|_{L^p(w)}\lesssim 2^{ln\widehat{r}}2^{-i\varepsilon}.
\end{align*}
We observe now that for every $i\geq 1$, $j\geq l+3$, $x\in C_j(Q_i)$, $\ell(Q)<t\leq \frac{2^{j-3}}{2^{l+1}}\ell(Q_i)$, and $(y,t)\in \Gamma(x)$, we have that
$B(y,2^{l+1}t)\subset 2^{j+2}Q_i\setminus 2^{j-1}Q_i$. Therefore, arguing as in the estimate of $\left\|\chi_{C_j(Q_i)}F_{1,l}^1{\mm}_i\right\|_{L^1(w)}$ and $\left\|\chi_{C_j(Q_i)}F_{1,l}^2\widetilde{\mm}_i\right\|_{L^1(w)}$, we have that
$$
\left\|\chi_{C_j(Q_i)}F_{2,l}^1\widetilde{\mm}_i\right\|_{L^1(w)}
\lesssim
 w(2^{j+1}Q_i)^{\frac{1}{p'}}\|\mathcal{M}_{p_0}(\chi_{2^{j+2}Q_i\setminus 2^{j-1}Q_i}\mathfrak{g}_{\hh,\ell(Q)}\widetilde{\mm}_i)\|_{L^p(w)}
\lesssim 
2^{-i\left(2M+\varepsilon\right)}
2^{-j\left(2M+\frac{n}{p_0}-n\widehat{r}\right)},
$$
 and
\begin{align*}
\left\|\chi_{C_j(Q_i)}F_{2,l}^2\widetilde{\mm}_i\right\|_{L^1(w)}
&\lesssim 
w(2^{j+1}Q_i)|2^{j+1}Q_i|^{-\frac{1}{q}}\|\widetilde{\mm}_i\|_{L^{p_0}(\R^n)}\left(\int_{\frac{2^{j-l-4}\ell(Q_i)}{\sqrt{2}}}^{\infty}\left(\frac{\ell(Q)}{r}\right)^{4M}r^{-2n\left(\frac{1}{p_0}-\frac{1}{q}\right)}\frac{dr}{r}\right)^{\frac{1}{2}}
\\
&\lesssim 2^{cl}
2^{-i\left(2M+\varepsilon\right)}
2^{-j\left(2M+\frac{n}{p_0}-n\widehat{r}\right)}.
\end{align*}
Consequently,  by \eqref{F2lP} we conclude that $\|F_{2,l}\mm\|_{L^1(w)}\leq 2^{lc}$. Then, in view of \eqref{F2Pois}, this and the estimate obtained for $\|F_{1,l}\mm\|_{L^1(w)}$ imply that
$\|I\|_{L^1(w)}\leq C$, which finishes the proof.
\end{proof}
%%%%%%%%%%%%%%%%%%%%%%%%%%%%%%%%%%%%%%%%%%%%%%%%%%%%%%%%%%%%%%%%%%%%%%%%%%%%%%%%%%%%%%%
%%%%%%%%%%%%%%%%%%%%%%%%%%%%%%%%%%%%%%%%%%%%%%%%%%%%%%%%%%%%%%%%%%%%%%%%%%%%%%%%%%%%%%%

\bigskip

Hence, we are ready to  prove the next proposition which easily implies Theorem \ref{thm:hardychartzNontangential}.

\begin{proposition}\label{lemma:hardy-N}
Let $w\in A_{\infty}$, $p\in \mathcal{W}_w(p_-(L),p_+(L))$, $M\in \N$ such that $M>\frac{n}{2}\left(r_w-\frac{1}{2}\right)$, and $\varepsilon_0=2M+2+\frac{n}{2}-r_wn$, there hold
\begin{list}{$(\theenumi)$}{\usecounter{enumi}\leftmargin=1cm \labelwidth=1cm\itemsep=0.2cm\topsep=.2cm \renewcommand{\theenumi}{\alph{enumi}}}

\item[(a)]$
\mathbb{H}^1_{\Ncal_{\hh},p}(w)=\mathbb{H}^1_{\Scal_{\hh},p}(w)= \mathbb{H}^1_{L,p,\varepsilon_0, M}(w),
$
 with equivalent norms.

\item[(b)]
$
\mathbb{H}^1_{\Ncal_{\pp},p}(w)=\mathbb{H}^1_{\Gcal_{\pp},p}(w)= \mathbb{H}^1_{L,p,\varepsilon_0, M}(w),
$
 with equivalent norms.
\end{list}
\end{proposition}
\begin{proof}
Fix $w\in A_{\infty}$, $p\in \mathcal{W}_w(p_-(L),p_+(L))$, $M\in \N$ such that $M>\frac{n}{2}\left(r_w-\frac{1}{2}\right)$, and $\varepsilon_0=2M+2+\frac{n}{2}-r_wn$.

In order to prove part $(a)$, note that 
for $f\in \mathbb{H}^1_{L,p,\varepsilon_0, M}(w)$, Proposition \ref{prop:contro-mol-NT}, part ($b$) yields that
$$
\|f\|_{\mathbb{H}^1_{\Ncal_{\hh},p}(w)}=\|\Ncal_{\hh}f\|_{L^1(w)}\lesssim\|f\|_{\mathbb{H}^1_{L,p,\varepsilon_0, M}(w)}.
$$ 
Therefore, since in particular $f\in L^p(w)$, we have that $f\in \mathbb{H}^1_{\Ncal_{\hh},p}(w)$.

\medskip

Take now  $f\in \mathbb{H}^1_{\Ncal_{\hh},p}(w)$. Lemma \ref{lema:comparacion,SH-GH,SK-GK}, part ($a$), Proposition \ref{prop:comparacion-N-G}, part ($b$), and Remark \ref{remark:classesoffunctions-N} imply
$$
\|\Scal_{\hh}f\|_{L^1(w)}\lesssim\|\Grm_{\hh}f\|_{L^1(w)}\lesssim \|\Ncal_{\hh}f\|_{L^1(w)}.
$$ 
Then $f\in \mathbb{H}^1_{\Scal_{\hh},p}(w)$. Consequently, from Proposition \ref{lema:SH-1}, part ($a$), $f\in \mathbb{H}^1_{L,p,\varepsilon_0,M}(w)$ and 
$$
\|f\|_{\mathbb{H}^1_{L,p,\varepsilon_0, M}(w)}\lesssim \|f\|_{\mathbb{H}^1_{\Scal_{\hh},p}(w)}\lesssim \|f\|_{\mathbb{H}^1_{\Ncal_{\hh},p}(w)}.
$$

\medskip 

As for part ($b$), take $f\in \mathbb{H}^1_{L,p,\varepsilon_0,M}(w)$ and apply Proposition \ref{prop:contro-mol-NT}, part ($b$), to obtain
$$
\|f\|_{\mathbb{H}^1_{\Ncal_{\pp},p}(w)}=\|\Ncal_{\pp}f\|_{L^1(w)}\lesssim\|f\|_{\mathbb{H}^1_{L,p,\varepsilon_0,M}(w)}.
$$
Hence, since again $f\in L^p(w)$, we have that
$f\in \mathbb{H}_{\Ncal_{\pp},p}^1(w).$

Finally, notice that for $f\in \mathbb{H}^1_{\Ncal_{\pp},p}(w)$  Proposition \ref{prop:comparacion-N-G}, part ($a$), and Remark \ref{remark:classesoffunctions-N} imply that 
$$
\|\Gcal_{\pp}f\|_{L^1(w)}\lesssim\|\Ncal_{\pp}f\|_{L^1(w)}.
$$
Therefore,  $f\in \mathbb{H}^1_{\Gcal_{\pp},p}(w)$.
Then, applying Proposition \ref{lemma:SKP}, part ($c$), we conclude that
$$
\|f\|_{\mathbb{H}^1_{L,p,\varepsilon_0, M}(w)}\lesssim \|f\|_{\mathbb{H}^1_{\Gcal_{\pp},p}(w)}\lesssim\|f\|_{\mathbb{H}^1_{\Ncal_{\pp},p}(w)}.
$$
and  $f\in \mathbb{H}^1_{L,p,\varepsilon_0, M}(w)$.
\end{proof}

\end{document}